\documentclass[sn-mathphys,Numbered]{sn-jnl}

\usepackage{graphicx}%
\usepackage{multirow}%
\usepackage{amsmath,amssymb,amsfonts}%
\usepackage{amsthm}%
\usepackage{mathrsfs}%
\usepackage[title]{appendix}%
\usepackage{xcolor}%
\usepackage{textcomp}%
\usepackage{manyfoot}%
\usepackage{booktabs}%
\usepackage{algorithm}%
\usepackage{algorithmicx}%
\usepackage{algpseudocode}%
\usepackage{listings}%

\usepackage{graphicx,latexsym}
\usepackage{epsfig}   
\usepackage{xcolor}
\usepackage{multicol}
\usepackage{pgfplots}
\usepackage{color}
\usepackage{amsthm}
\usepackage{subcaption}
\usepackage{cleveref}
\usepackage{soul}
\newcommand{\re}{\mathbb{R}}

\newcommand{\rev}[1]{{#1}}
\newcommand{\revnew}[1]{{#1}}
\newcommand{\chri}[1]{{#1}}
\newcommand{\gianni}[1]{{#1}}
\newcommand{\revmax}[1]{#1}
\newcommand{\revgianni}[1]{#1}
\newcommand{\ie}{\textit{i.e.}, }
\newcommand{\eg}{\textit{e.g.}, }

\newcommand{\revnewmax}[1]{\textcolor{black}{#1}}

\newcommand{\revnewgianni}[1]{\textcolor{black}{#1}}

\numberwithin{equation}{section}

\setstcolor{red}

\theoremstyle{thmstyleone}%
\newtheorem{theorem}{Theorem}[section]
\newtheorem{lemma}[theorem]{Lemma}%
\newtheorem{observation}[theorem]{Observation}%

\theoremstyle{thmstyletwo}%

\theoremstyle{thmstylethree}%
\newtheorem{definition}[theorem]{Definition}%

\begin{document}
\title[Benchmarking optimization algorithms...]{
\vspace*{-0.7truecm} \huge Benchmarking Optimization Algorithms with Quality Profiles and Test Set Profiles}

\author*[1]{\fnm{Giovanni} \sur{Fasano}}\email{fasano@unive.it}

\author[2]{\fnm{Christian} \sur{Piermarini}}\email{piermarini@diag.uniroma1.it}

\author[2]{\fnm{Massimo} \sur{Roma}}\email{roma@diag.uniroma1.it}

\affil[1]{\orgdiv{Venice School of Management},  \orgname{University Ca' Foscari},  \orgaddress{\street{San~Giobbe~Cannaregio~873}, \city{Venezia}, \postcode{30121}, \country{Italy}}}

\affil[2]{\orgdiv{Dep. of Computer, Control and Management Engineering ``A. Ruberti''},   \orgname{SAPIENZA University of Rome},  \orgaddress{\street{via Ariosto 25}, \city{Rome}, \postcode{00185}, \country{Italy}}}


\abstract{We propose a couple of novel tools  
for benchmarking optimization algorithms which possibly converge to different solutions on a test set: the
\emph{quality profiles} and the \emph{test set profiles}. \rev{Their aim is to} \rev{assess and compare algorithms} in terms of quality \gianni{(\ie value of the objective function)} of the obtained solutions, as well as to assess the consistency of the test set. A key distinguishing feature of the quality profiles we propose is its comparative deterministic procedure that emphasizes the accuracy of the solution, rather than the computational burden of solvers. In this regard, several \revnewgianni{test set--dependent} approaches for both \rev{comparing} and ranking algorithms have already been proposed in the literature, \rev{representing widely used} benchmarking procedures. \rev{We believe that the joint use of such procedures, along with the novel quality profiles detailed here, should enhance the benchmarking process, in all those cases where the comparison encompasses exact methods as well as heuristics. Moreover, the literature on numerical optimization seems to have paid less attention, in the last decade, to determining how appropriate a test set used for benchmarking the selected solvers may be. This motivates the introduction of test set profiles, which assess the appropriateness of a test set and represent the flip side of evaluating the robustness of the solvers on that test set. This paper also includes extensive numerical experiments, showing the usefulness of quality profiles in both smooth and nonsmooth (\eg derivative--free) optimization, along with the reference to a {\tt MATLAB} code for plotting quality profiles and test set profiles.}
}

\keywords{Benchmarking solvers, Accuracy assessment of codes,  Quality profiles, Performance evaluation, Absolute and relative benchmarking.}



\maketitle

\section{Introduction}
\label{sec:intro}
Benchmarking algorithms for local constrained and unconstrained optimization
has always been a challenging issue. Any \rev{proposed} \gianni{benchmarking} procedure needs to be unbiased with respect to the features of the compared solvers, as well as robust with respect to the test set used for benchmarking.
In recent years, many tools have been proposed to perform comparisons \gianni{among} optimization \rev{algorithms}\gianni{,}
in terms of various performance measures. These comparative studies are \rev{also} very beneficial for users to understand which optimization method fits their needs the most.
\rev{Of course, in order to obtain \gianni{meaningful} conclusions} and to detect potential \rev{weaknesses and/or strengths,}
solvers must be tested on a variety of different scenarios\gianni{. Hence,} benchmarking \rev{procedures} must be capable of \rev{clearly} summarizing the overall results obtained \rev{on} a large set of test problems. Moreover, benchmarking procedures are often adapted to better take into account the peculiarities of the optimization framework where solvers are embedded (see, \eg the benchmarking procedures for derivative--free/\rev{black--box} optimization algorithms \cite{more.wild, hansen2021coco}).
\par
\rev{As is well--known (see, \eg \cite{beiranvand2017best}),  performance measures for assessing an optimization algorithm}  are typically rated over three main \rev{categories}: \rev{\em i)} \textit{efficiency},  \rev{\em ii)} \textit{reliability \rev{and robustness}}, \rev{\em iii)} \textit{quality of 
\rev{the achieved solution}}.
As regards the \textit{efficiency}, a number of comparative procedures \gianni{from the literature} have already proven to be helpful to gauge it.
\emph{Performance profiles} by Dolan and Mor\'e \cite{dolan.more} along with \emph{data profiles} by Mor\'e and Wild \cite{more.wild}, are undoubtedly among the most adopted tools in \gianni{both constrained and unconstrained} \rev{continuous optimization} \gianni{frameworks}.
The former are widely considered a standard tool for benchmarking optimization solvers; they are based on measuring the convergence effort through indicators \gianni{like} CPU time, number of performed iterations or number of required function evaluations. On the other hand, data profiles typically target derivative--free optimization solvers\gianni{, assuming different scenarios where} increasing values for the accuracy are considered.
\rev{However, while in a data profile the track associated to a specific solver is, to some extent, independent \gianni{of {\em all} the} other solvers included in the comparison, this is not the case for performance profiles. In fact, the latter only highlight performance with respect to the best algorithm and they do not allow ranking \gianni{the $i$-th} solver relative to \gianni{the $j$-th one} which is not the best (see \cite{gould2016note,hekmati2019nested}).}
\par
\textit{Reliability} \rev{and \textit{robustness}} of an optimization algorithm \rev{represent} a solver's ability to \rev{successfully \gianni{and appropriately} solve problems. They  can be measured in terms of number of test problems, in the test set, solved by an algorithm. The \emph{success rate} is commonly adopted for evaluating algorithmic reliability (see \cite{beiranvand2017best} and the references reported therein).}
Clearly, the success rate is not a profiling method compared to performance profiles or data profiles\gianni{,}
nor it benefits from the same properties of the aforementioned benchmarking procedures. \rev{Note that, performance profiles also include information on the robustness of an algorithm (this can be evinced by observing the corresponding \gianni{algorithm plot,} for sufficiently large values of the abscissa).}
\par
\gianni{Finally,} \rev{comparing} solvers \gianni{with respect to} the \textit{quality of the achieved solution} is extremely valuable as well, and it represents the main focus of this work. \gianni{Furthermore, we are persuaded that the issue {\em iii)} has not yet had a reasonably wide consideration in the literature of numerical optimization, with the exclusion of few proposals (see, \eg  \cite{beiranvand2017best,audet.hare.book,hare2006benchmark} and therein references). In this regard,}
\rev{denoting by $f^{(p)}(x)$ the objective function of the test problem $p$, if a (local) optimal solution $x^\star$  of the problem $\min_{x \in {\cal F}}f^{(p)}(x)$, with ${\cal F} \subseteq \re^n$, \gianni{were} available, once a solver $s$ has determined 
the solution point $\hat{x}$, the difference $f^{(p)}(x^\star)-f^{(p)}(\hat x)$ can represent a measure of the ``quality'' (accuracy) of the solution \gianni{by the solver}. Such difference is usually normalized, \ie the quantity
$$f^{(p)}_{\rm acc}=
\frac{f^{(p)}(x^\star)-f^{(p)}(\hat x)}{f^{(p)}(x^\star)-f^{(p)}(x_0^{(p)})}
$$
}%
\gianni{can be} monitored, where \gianni{$x_0^{(p)}$ is the (common)} starting point adopted by \gianni{any} algorithm $s$ for solving the test problem $p$. Note that, since in most cases the optimal solution for all the problems of the test set is unavailable, $f^{(p)}(x^\star)$ can be replaced \gianni{by} the best solution found  \revmax{for the problem $p$ by the different solvers}\gianni{, or an appealing reference value suggested by the literature}. The base $10$ logarithm of $f^{(p)}_{\rm acc}$ may be  considered, in order to \gianni{provide a quantitative} indication of the number of digits of accuracy obtained (see, \eg \cite{beiranvand2017best,audet.hare.book}). Then, this measure of accuracy can be used to construct {\em accuracy profiles}: they show the proportion of problems for which a \gianni{given} solver is able to determine a solution within a certain accuracy, with respect to the \gianni{solution obtained by the best solver}. It is important to note that this benchmarking procedure can be only applied if \gianni{a} {\em fixed--cost} approach is adopted, \ie if a prefixed number of function evaluations, number of iterations or a prefixed CPU time is used as stopping criterion \gianni{for the compared algorithms}. This motivates the use of the  accuracy profiles mostly within  derivative--free/black--box optimization (see \cite{audet.hare.book}), where typically an overall budget of function evaluations is fixed. However, quoting from \cite[page 275]{audet.hare.book}, ``{\em accuracy profiles ignore the number of function evaluations required to achieve the presented results, and thus accuracy profiles can be strongly biased when different algorithms use significantly different numbers of function evaluations}''.
\par
In this paper, drawing inspiration also from accuracy profiles, we propose a novel general benchmarking procedure that can be used for comparing exact and heuristic continuous optimization algorithms. The aim is first to introduce a \gianni{specific} useful tool that allows comparisons among algorithms in terms of ``quality'' of the found solution. They can be \gianni{coupled with} the performance (or data) profiles, to have more complete information on algorithms behaviour. There is a further motivation to introduce and adopt quality profiles. Indeed, observe that when performance profiles are used, for a fair comparison, test problems where algorithms converge to different solutions are excluded from the comparison. It could be meaningless indeed to compare the performance, in terms of computational burden, for algorithms converging towards different solution points.
In this regard, to the best of these authors' knowledge, a complete tool for providing in a condensed graphical form \gianni{the} capability of an algorithm to determine \gianni{better/worse quality} solutions is actually unavailable.

\par
\rev{Therefore,} \gianni{this paper addresses} a new benchmarking procedure, to evaluate the {\em quality} of the solution points computed by different optimization algorithms, in a similar fashion with respect to the aforementioned accuracy profiles. \rev{In particular, given the starting point $x_{0}^{(p)}$ for a test problem $p$ and a reference value $f^{(p)}_{L}$ \gianni{associated to the problem $p$}, the quality profiles are obtained based on the following inequality
\begin{equation}
	f_s^{(p)}(x^\ast) - f_L^{(p)}  \leq \tau \left[ f^{(p)} ( x_0^{(p)} ) - f_L^{(p)} \right],
\end{equation}
where $f_s^{(p)}(x^\ast)$ is the optimal function value determined by the solver $s$ on the problem $p$}, and $\tau$ represents a  parameter in the interval $[0,1]$.
Among the basic relevant features of the proposal in the current paper we find the following ones (see the subsequent sections for details):
\begin{itemize}
    \item our benchmarking system inherits almost the same appealing features of the profiling methods in \cite{dolan.more} (\textit{performance profiles}) and \cite{more.wild} (\textit{data profiles}). Moreover, some additional results are explicitly proved which do not hold for the proposals in \cite{dolan.more} and \cite{more.wild};
    \item \gianni{unlike \cite{dolan.more}, \cite{more.wild} and \cite{audet.hare.book} the proposal in the current paper does not require any {\em a priori} identification of parameters/thresholds, which risks to introduce biases in the benchmarking\footnote{Indeed in \cite{dolan.more} and \cite{more.wild} the parameters $r_M$ and $\mu_f$ need to be set. On the other hand, in \cite{audet.hare.book} the choice of a priori parameters is definitely a less relevant issue for the following reason. In \cite{audet.hare.book} the authors require to define a threshold for the quantity $1-f_{acc}^{(p)}$, when it becomes close to zero and $\log_{10}(1-f_{acc}^{(p)})$ turns out to be undefined. This equivalently corresponds to assume that a finite (though very large) maximum precision for the solvers is decided a priori, \revnewmax{which} is typically sufficient in most of the practical applications. Conversely, the proposal in the current paper allows for gauging a precision for solvers which is possibly not bounded by any finite positive value (see the choice of the parameter $r_1$ in \eqref{eq:def_Qs_b}).};}
    \item we allow a clearer comparison among different algorithms through plots, where zooming opportunities are also allowed  with respect to the axes in the plots. We strongly remark that the last property is not immediately extendable neither to \cite{dolan.more} nor to \cite{more.wild} \gianni{or \cite{audet.hare.book}};
    \item the proposed tool is able to suggest a ranking procedure \gianni{among} many codes, by using the objective function values at the solution points. In this regard, benchmark test problems where solvers possibly do not converge or converge to different solutions, do not need to be excluded.
\end{itemize}

This paper also addresses a slightly different issue: namely we are persuaded that a quantitative tool that can measure how appropriate is the {\em test set} for {\em solvers}, may be impactful. This last issue may be declined from two complementary perspectives: the consistency of the test set for the chosen solvers (see, \eg the seminal paper \cite{wolpert1997no}), and the robustness of the solvers on the given test set. Hence, apart from quality profiles, we intend to propose here a second tool, namely the {\em test set profiles}, where the last two perspectives are jointly studied and effectively depicted using a practical graphical representation.
\par
{An important remark is that the goal of quality profiles is not to replace existing benchmarking tools through a new axiomatic principle. Rather, quality profiles are intended as a complementary device, specifically designed to assess the quality of the final objective values returned by competing solvers. Methodologically, this differs from the line of work by Liu et al., which proceeds axiomatically: desirable properties of a comparison rule are postulated and violations are exhibited as paradoxes, such as cycle ranking and survival of the nonfittest. Of the several related contributions, we focus on \cite{liu2020paradoxes}, where these paradoxes are first identified, and \cite{yan2022paradox,ZHI2024}, where sufficient conditions are derived and applied. Since these conditions concern how a rule is built rather than what it measures, they apply in principle to quality profiles too; a preliminary look, detailed in Section 2.1, suggests a mixed picture depending on the reference value adopted. This points to a genuine literature gap: no quality-oriented profile has yet been examined from this axiomatic viewpoint. Indeed, our approach here proceeds constructively rather than axiomatically. A condensed tool for comparing solution quality when solvers converge to different points was, to our knowledge, still lacking. We fill this gap by directly defining a new procedure and establishing its local properties (affine invariance, monotonicity), addressing a dimension of benchmarking distinct from, but assessable through, the cross-set consistency properties studied by Liu et al.}
\par
{Benchmarking literature has seen advancements
aimed at eliminating evaluation biases induced by arbitrarily chosen target levels or
localized performance thresholds. In \cite{da2010attainment} an attainment-function
approach is formalized for stochastic multiobjective optimizers: grounded in random
closed-set theory, the attainment function is cast as a mean of the distribution of
optimizer outcomes in the objective space, with higher-order versions capable of
capturing further aspects of this distribution and, in principle, a full distributional
characterization. 
More recently, in \cite{lopez2024using} this paradigm is extended to single-objective
black-box optimization, showing that a traditional target-based Empirical Cumulative Distribution Function (ECDF), such as the one adopted in black-box
optimization benchmark suites, is an approximation of the Empirical Attainment Function (EAF). The EAF offers several advantages over the target-based ECDF,
most notably that it does not require defining a priori quality targets, while
still capturing performance differences more precisely. The authors further show
that the average Area Over the Convergence Curve (AOCC) is an equivalent, but
computationally simpler, summary statistic of anytime performance. Our proposed quality profiles share a philosophical alignment with this paradigm,
as they too avoid the need to specify a priori quality targets that risk biasing the
comparison. However, while the EAF and its multiobjective precursor track probabilistic
anytime performance across continuous runtime distributions, quality profiles operate
on a deterministic, budget-driven paradigm particularly suitable for derivative-free
and black-box optimization, evaluating solution precision under a sequence of fixed,
expanding computational budgets. We remark that this does not make quality profiles
parameter-free in an absolute sense. The scaling exponents $r_1, r_2$ introduced in
Section~2.1 retain a role, though a purely visualization-oriented one, distinct from
the a priori success thresholds that the attainment-function paradigm dispenses with.}
\par
{Several further strands of research are relevant for positioning the
contribution of quality profiles. \cite{smith2014objective} emphasized that benchmark
conclusions may depend strongly on the chosen instances, and proposed an
instance-space methodology to identify where algorithms are strong or weak and
whether a test set is sufficiently representative. This perspective is complementary
to ours: while quality profiles summarize the distribution of attained solution
quality over a benchmark set, instance-space analysis explains how such behavior
varies across classes of instances. Methodology-oriented contributions, such as
\cite{willemsen2024methodology}, focus instead on fair experimental design, budget
definition, stochasticity, and reproducible reporting; this addresses a different
level of the benchmarking problem, since quality profiles are not intended to replace
such methodological frameworks, but to provide a descriptive tool for one specific
performance dimension, namely the final objective value reached by the solver.
Earlier discussions on benchmarking procedures, such as \cite{opara2023benchmarking},
provide broader guidelines on sound experimental practice, complementary to the
narrower focus of quality profiles on solution quality. Finally, \cite{porcelli2017note} warns that performance and data profiles, if used naively as an objective function to train or tune an algorithm, may induce overfitting or degenerate behavior. This caveat would apply equally to quality profiles if they were used in future work to guide algorithm tuning. Taken together, these contributions frame quality profiles as one component of a broader benchmarking ecosystem, alongside representative test sets, sound experimental methodology, and structurally consistent comparison rules.}

\revnewmax{
A related line of work concerns the design and selection of benchmark suites. For instance, \cite{https://doi.org/10.1002/wics.70028} review and classify benchmark and test functions according to properties such as modality, dimensionality, separability, smoothness, constraints, and noise/stochasticity, and propose best practices for selecting informative and representative benchmarks. This perspective is complementary to the one adopted here. Indeed, the goal of test set profiles is not to replace ex ante recommendations for benchmark construction, but rather to provide an ex post diagnostic tool for assessing whether a selected test set is actually informative, consistent, and sufficiently discriminating with respect to the solvers under comparison. In this sense, benchmark-selection guidelines and test set profiles address two different stages of the benchmarking process: the former concern how to build a suitable test suite, while the latter concern how to evaluate the adequacy of that suite once the experiments have been performed.
}

\revnewmax{
Another important related contribution is the work \cite{Eckman_et_al_2023} by Eckman et al., who develop diagnostic tools for evaluating and comparing simulation-optimization algorithms and explicitly advocate the use of bootstrapping to obtain error estimates for the corresponding estimators and plots. This reference is important for the present work, since it shows that bootstrap-based uncertainty assessment is already a relevant idea in benchmarking and diagnostic analysis. However, the primary focus of the two contributions is different. In \cite{Eckman_et_al_2023}, bootstrapping supports the statistical assessment of solver-comparison diagnostics under stochastic simulation. In the present work, by contrast, test set profiles are introduced to assess the adequacy, informativeness, and discriminating power of the test set itself with respect to a family of solvers. Thus, bootstrapping should be viewed here as a supporting inferential device rather than as the core novelty of the proposal. The novelty of the present proposal lies not in bootstrapping as a statistical technique, but in the introduction of test set profiles as the primary object of diagnostic assessment.
}
\par

The paper is structured as follows: \rev{in Section~\ref{sec:qp} quality profiles for smooth optimization are defined and some important features are shown\gianni{, as well as some insights in a comparison with the current literature of benchmarking procedures}. Section~\ref{sec:application} reports the results of two different numerical experimentations where quality profiles are adopted for comparing algorithms, in terms of quality of the achieved solution. Section~\ref{sec:quality_profiles_derivative--free} is devoted to use quality profiles when benchmarking derivative--free solvers. 
Section~\ref{sec:test_set_profiles} introduces the test set profiles and their potential application. Finally, in Section~\ref{sec:conclusions} some  concluding remarks are reported.}
\par

\section{Quality Profiles for smooth optimization}\label{sec:qp}
The main purpose of the present section is to introduce some basic theoretical properties of a novel class of profiles, that we name {\em quality profiles}. In the sequel we use as a reference optimization problem, without loss of generality, the following one:
    $$ \min_{x \in {\cal F}} f(x), $$
\rev{where $f : {\cal F} \longrightarrow \re$}, being ${\cal F} \subseteq \re^n$ the feasible set. \revgianni{Of course any unconstrained (${\cal F} \equiv \re^n$) and constrained (${\cal F}  \subset  \re^n$) optimization problem falls within the last category}.
\par
Assume the set ${\cal S}$ of solvers and the set ${\cal P}$ of test problems are considered, {\color{black}being $|{\cal S}| \geq 2$ and $|{\cal P}|\geq 1$ their cardinalities}. Let $x_0^{(p)} \in \re^n$ be the \revgianni{common} starting point for {all} the solvers on the {test} problem $p \in {\cal P}$, and {let $f^{(p)}(x)$ be the objective function of the problem $p \in {\cal P}$.} Let $f_L^{(p)}$ {\color{black}be a} reference value for the objective function of the problem $p$. If $x^\ast$ indicates the best point found by the solver $s \in {\cal S}$ on the test problem $p$ (observe that for the sake of simplicity{\color{black}---and with the idea of reducing redundancy---}we drop from $x^\ast$ the dependency on both $s$ and $p$), as remarked also in \cite{dolan.more} and \cite{more.wild}, possible choices for $f_L^{(p)}$ can be, {\color{black}for instance}:
\begin{enumerate}
	\item $\displaystyle f_L^{(p)} = \min_{s \in {\cal S}} \left\{ f^{(p)}_s(x^\ast) \right\}$ \gianni{or \  $\displaystyle f_L^{(p)} = \alpha f^{\ast}$, with \  $f^{\ast}= \displaystyle \min_{s \in {\cal S}} \left\{ f^{(p)}_s(x^\ast) \right\}$, and where $0 < \alpha < 1$ in case $f^{\ast}>0$ or $\alpha > 1$ in case $f^{\ast}<0$,}
    \label{page_4}
	\item a {\em reference value} generated by {\color{black}an algorithm (typically an efficient one) possibly not included in the list of the compared solvers},
	\item $\displaystyle f_L^{(p)} = \displaystyle \frac{1}{|{\cal S}|} \sum_{s \in {\cal S}} \left\{ f^{(p)}_s(x^\ast) \right\}$,
\end{enumerate}
{where $f^{(p)}_s(x^\ast)$ denotes the optimal function value determined by the solver $s$ on the problem $p$.}
In this regard, {\color{black} in accordance with \cite{dolan.more,more.wild}}, we adopt the first {one} among the above three choices setting $\alpha=1$, namely $ f_L^{(p)} = \min_{s \in {\cal S}} \left\{ f^{(p)}_s(x^\ast) \right\}$, and define for each solver $s \in {\cal S}$ the ratio
\begin{equation}
	Q_s(\tau) = \frac{1}{|{\cal P}|} {\rm size} \left\{ p \in {\cal P} \ : \ f_s^{(p)}(x^\ast) - f_L^{(p)}  \leq \tau \left[ f^{(p)} ( x_0^{(p)} ) - f_L^{(p)} \right] \right\},        \label{eq:def_Qs}
\end{equation}
being $\tau \geq 0$. Since evidently for {\color{black}any solver $s$ we have
	\begin{equation}
		f_s^{(p)}(x^\ast) \leq f^{(p)} ( x_0^{(p)}),        \label{equ:better_than_x0}
	\end{equation}
then for any $s \in {\cal S}$ we obtain $Q_s(\tau)=1$ {for all} $\tau \geq 1$. This immediately suggests that} it suffices to consider hereafter for the parameter $\tau$ the range
$$ \tau \in [0,1]. $$

{We can now} introduce the {following} formal definition of {\em quality profiles}.
\par\medskip
\revgianni{
\begin{definition}
	\label{def_quqlity_profiles}
	Given the benchmark set {of test problems} ${\cal P}$ and the set of solvers ${\cal S}$, the corresponding {\em Weak Quality Profiles} are defined as the collection of plots of all the functions $\{Q_s(\tau)\}$ {in \eqref{eq:def_Qs}} when $\tau \in [0,1]$. In case in \eqref{eq:def_Qs} the quantity
		$$ f^{(p)}(x_0^{(p)}) $$
	were replaced by the quantity
		$$ \max_{s \in {\cal S}}f_s^{(p)}(x^*), $$
	then we will use the name {\em Strong Quality Profiles}.
\end{definition}
}
\par\medskip
A trivial consequence of Definition~\ref{def_quqlity_profiles} is the following. Similarly to performance profiles and data profiles, when the plot $Q_s(\tau)$ of a solver $s \in {\cal S}$ is mostly located in an upper position, for $\tau \in T \subseteq [0,1]$, it means that for $\tau \in T$ the code $s$ solves a large percentage of test problems where a good minimizer is detected. \revnewgianni{As another similarity with data profiles (see relation (2.2) of \cite{more.wild}), the quantity on the left hand side of \eqref{eq:def_Qs} represents a normalization term, relying on both $x_0^{(p)}$ and $f_L^{(p)}$. The rationale behind this choice is suggested by the aim of monitoring the progress of algorithms with respect to the same given initial iterate, even possibly a poor one. This is justified by the {\em a priori} lack of specific knowledge about a more appropriate choice of the initial point, that might yield a less distortive comparison.}

As preliminary facts, the next considerations hold for the quality profiles:
\begin{itemize}
	\item in case of strong quality profiles the user is interested about the mutual behavior among solvers (\ie the focus is on addressing the performance with respect to the {\em worst solution found} by a solver), possibly disregarding the bias that the reference to the initial point $x_0^{(p)}$ may introduce. Conversely, weak quality profiles both reveal a reference to the initial point $x_0^{(p)}$, and address	the reciprocal ranking among solvers, with specific attention to the actual range of values of the objective function, on each test problem, \ie $f^{(p)} ( x_0^{(p)} ) - f_L^{(p)}$;
	
	\item weak quality profiles seem appropriate in a wide range of applications. Conversely, strong quality profiles seem more appropriate in those cases where solvers show close values of the objective function at the solution;
	
	\item by \eqref{eq:def_Qs} we have $0 \leq Q_s(\tau) \leq 1$, for any solver $s\in {\cal S}$ and for any value $\tau \in [0,1]$;
	\item for $\tau=0$ the {\color{black}value} $Q_s(0)$ indicates the percentage of problems where the solver $s$ provides the best reference value of the objective function $f_L^{(p)}$ (note that {\color{black}an almost identical} property holds for the performance profiles when the abscissa value is equal to 1);
	\item quality profiles suitably take into account possible failures of the solver $s$ on some benchmark problems: {\color{black}indeed in the last case the corresponding problem will be discarded and will not contribute to increase the quantity $size \{ \cdot \}$ in \eqref{eq:def_Qs}. Hence, we will have $Q_s(1) < 1$ for the corresponding solver $s$ (again, here we can immediately realize that a similar property holds for performance profiles)}.
\end{itemize}
\rev{Observe that} performance profiles can hardly be seen as a possible alternative to quality profiles, for at least a couple of reasons:
	\begin{itemize}
		\item since for a given solver $s\in {\cal S}$ and a test problem $p\in {\cal P}$ we might have $f_s^{(p)}(x^\ast) \gtreqless 0$,  \rev{the use of a performance profile \gianni{where} the \textit{objective function value} is adopted as benchmarking measure, } might not be well posed; indeed, some of the positive ratios $r_{p,s}$  needed to compute performance profiles \rev{(see \cite{dolan.more})}, might be undefined (when $f_s^{(p)}(x^*)=0$)  or inconsistent (when $|f_s^{(p)}(x^*)| \approx 0$ or $f_s^{(p)}(x^*)<0$);
		\item if performance profiles were used to benchmark the final value of the objective function, then they would completely disregard the {starting point} $x_0^{(p)}$, while the definition \eqref{eq:def_Qs} explicitly aims at monitoring the progress of each solver with respect to the function value at the initial iterate {$x_0^{(p)}$} (weak quality profiles) or at $ \hat x_0^{(p)} \in argmax_{s \in {\cal S}} \{f_s^{(p)}(x^*)\}$ (strong quality profiles).
	\end{itemize}
A number of additional theoretical results can be proved for quality profiles, as \gianni{remarked} in the next lemmas. For the sake of simplicity hereafter we refer to weak quality profiles, holding analogous properties for strong quality profiles, too.
\par\medskip
\begin{lemma}
	\label{lem1}
	Weak quality profiles are invariant under any consistent affine transformation of the objective functions.
\end{lemma}
\begin{proof}
	Let us consider, without loss of generality, the benchmark problem $p \in {\cal P}$ and the corresponding objective function $f^{(p)}(x)$. For any $a>0$ (to make the affine transformation consistent) and $b \in \re$ we define the affine transformation $a f^{(p)}(x) + b$ so that from the definition of $\{Q_s(\tau)\}$, and for any $s \in {\cal S}$, it is
	$$ \begin{array}{l}
		\left[ \left( a f_s^{(p)}(x^\ast) + b \right)  - \left( a f_L^{(p)} + b \right) \right] = a \left[ f_s^{(p)}(x^\ast) - f_L^{(p)} \right]       \\
		\                                       \\
		\left[ \left( a f^{(p)}(x_0^{\gianni{(p)}}) + b \right)  - \left( a f_L^{(p)} + b \right) \right] = a \left[ f^{(p)}(x_0^{\gianni{(p)}}) - f_L^{(p)} \right].
	\end{array} $$
	Thus, the next three inequalities are equivalent
	$$ \begin{array}{c}
		\left[ \left( a f_s^{(p)}(x^\ast) + b \right)  - \left( a f_L^{(p)} + b \right) \right] \leq \tau \left[ \left( a f^{(p)}(x_0^{\gianni{(p)}}) + b \right)  - \left( a f_L^{(p)} + b \right) \right],      \\
		\           \\
		a \left[ f_s^{(p)}(x^\ast) - f_L^{(p)} \right] \leq \tau a \left[ f^{(p)}(x_0^{\gianni{(p)}}) - f_L^{(p)} \right],	\\
			\           \\
		\left[ f_s^{(p)}(x^\ast) - f_L^{(p)} \right] \leq \tau \left[ f^{(p)}(x_0^{\gianni{(p)}}) - f_L^{(p)} \right],
	\end{array} $$
	so that the definition of $\{Q_s(\tau)\}$ proves the result.
\end{proof}
\begin{lemma}
	\label{lem2}
	\gianni{Let $f^{(p)}(x_0^{\gianni{(p)}}) \geq f_L^{(p)}$, for any $p \in {\cal P}$}. Then, for any $s \in {\cal S}$ the function $Q_s(\tau): \ [0,1] \rightarrow [0,1]$ is nondecreasing.
\end{lemma}
\begin{proof}
	Given $\tau_1,\tau_2 \in [0,1]$, with $\tau_1 \leq \tau_2$, let $p \in {\cal P}$. Then, by the definition of $Q_s(\tau)$, if
	$$ f_s^{(p)}(x^\ast) - f_L^{(p)} \leq \tau_1 \left[ f^{(p)}(x_0^{\gianni{(p)}}) - f_L^{(p)} \right]  $$
	then we have also
	$$ f_s^{(p)}(x^\ast) - f_L^{(p)} \leq \tau_1 \left[ f^{(p)}(x_0^{\gianni{(p)}}) - f_L^{(p)} \right] \leq \tau_2 \left[ f^{(p)}(x_0^{\gianni{(p)}}) - f_L^{(p)} \right].  $$
	Thus, the subset of benchmark problems where the leftmost inequality (in the last relation)
	is fulfilled, is included in the set of benchmark problems where 
	  $$	f_s^{(p)}(x^\ast) - f_L^{(p)} \leq \tau_2 \left[ f^{(p)}(x_0^{\gianni{(p)}}) - f_L^{(p)} \right]. $$
\end{proof}
\begin{lemma}
\label{lem3}
Let $f^{(p)}(x_0^{\gianni{(p)}}) \neq f_L^{(p)}$ for any $p \in {\cal P}$, and let
	$$ \tau_M = \max_{s \in {\cal S}, \ p \in {\cal P}} \left\{ \frac{f_s^{(p)}(x^\ast) - f_L^{(p)}}{f^{(p)}(x_0^{\gianni{(p)}}) - f_L^{(p)}} \right\}. $$
	Then $Q_s(\tau) = 1$ for any $\tau \geq \tau_M$ and $s \in {\cal S}$.
\end{lemma}
\begin{proof}
	The proof straightforwardly follows {\color{black}from \eqref{equ:better_than_x0} and} the definition of $Q_s(\tau)$.
\end{proof}
\begin{observation}
	\label{observ:5}{\em
		The functions $\{Q_s(\tau)\}$, for $\tau \in [0,1]$, may be difficult to compare in case the number of solvers increases, in particular when $\tau \in \{0,1\}$ (\ie when $\tau$ assumes the values at the extremes of its interval). The last fact may be due to close tracks associated to the functions $\{Q_{s}(\tau)\}$ when $\tau$ belongs to $\{0,1\}$. This suggests that a more sophisticated manipulation needs to be introduced, in order to compress or expand portions of the plot of the functions $\{Q_s(\tau)\}$, {\color{black}both along the abscissa and the ordinate axes}.
	}
\end{observation}
The last observation suggests that to compress/expand portions of a quality profile, in order to better sort and rank the different tracks associated with codes, we can modify \eqref{eq:def_Qs} into
\begin{equation}
	\left[Q_s(\tau)\right]^{r_2} = \frac{1}{|{\cal P}|} {\rm size} \left\{ p \in {\cal P} \ : \ f_s^{(p)}(x^\ast) - f_L^{(p)}  \leq \tau^{r_1} \left[ f^{(p)} ( x_0^{(p)} ) - f_L^{(p)} \right] \right\},        \label{eq:def_Qs_b}
\end{equation}
with $\tau \in [0,1]$, where $r_1$ and $r_2$ are positive real numbers. \rev{It results that:}
\begin{itemize}
	\item for $r_1 \in \{1,2,3,4, \ldots\}$ we obviously have again $0 \leq \tau^{r_1} \leq 1$. Moreover, for a given solver $s \in S$, when $r_1$ increases, the portion of the quality profile corresponding to {\em small} values of the abscissa (say equivalently $\tau$ close to zero) will be {\em expanded}. Conversely, when $r_1$ increases, the portion of the quality profile corresponding to {\em large} values of the abscissa (say equivalently $\tau$ close to one) will be {\em compressed}. As an example, note that when $r_1=1$ the mid--value of the abscissa axis is $1/2$, while for $r_1=5$ the mid--value of the abscissa axis becomes $(1/2)^5 = 1/32$, showing that the portion of the plot close to the origin has been expanded along the abscissa axis. \\ 
	\item for $r_1 \in \{1,1/2,1/3,1/4, \ldots\}$ we obviously have again $0 \leq \tau^{r_1} \leq 1$. Moreover, now for smaller values of $r_1$ the portion of the quality profile corresponding to {\em large} values of the abscissa (say equivalently $\tau$ close to one) will be {\em expanded}, and for smaller values of $r_1$ the portion of the quality profile corresponding to {\em small} values of the abscissa (say equivalently $\tau$ close to zero) will be {\em compressed};
	\item a dynamics similar to the one detailed in the previous items is experienced by increasing $r_2 \in \{1,2,3,4, \ldots\}$ or decreasing $r_2 \in \{1,1/2,1/3,1/4, \ldots\}$. As an example, when $r_2=1$ the mid--value of the ordinate axis is $1/2$ while for $r_2=6$ the mid--value of the ordinate axis is $(1/2)^6 = 1/64$, \ie the portion of the plot close to the origin has been expanded along the ordinate axis. \\
\end{itemize}

A picture reporting a summary of the contents in the above items is given in Figure~\ref{fig:Expand-Compress}. 
\begin{figure}[tbph]
	\centering
	\hspace*{-1.1truecm}
	\includegraphics[width=1.2\linewidth]{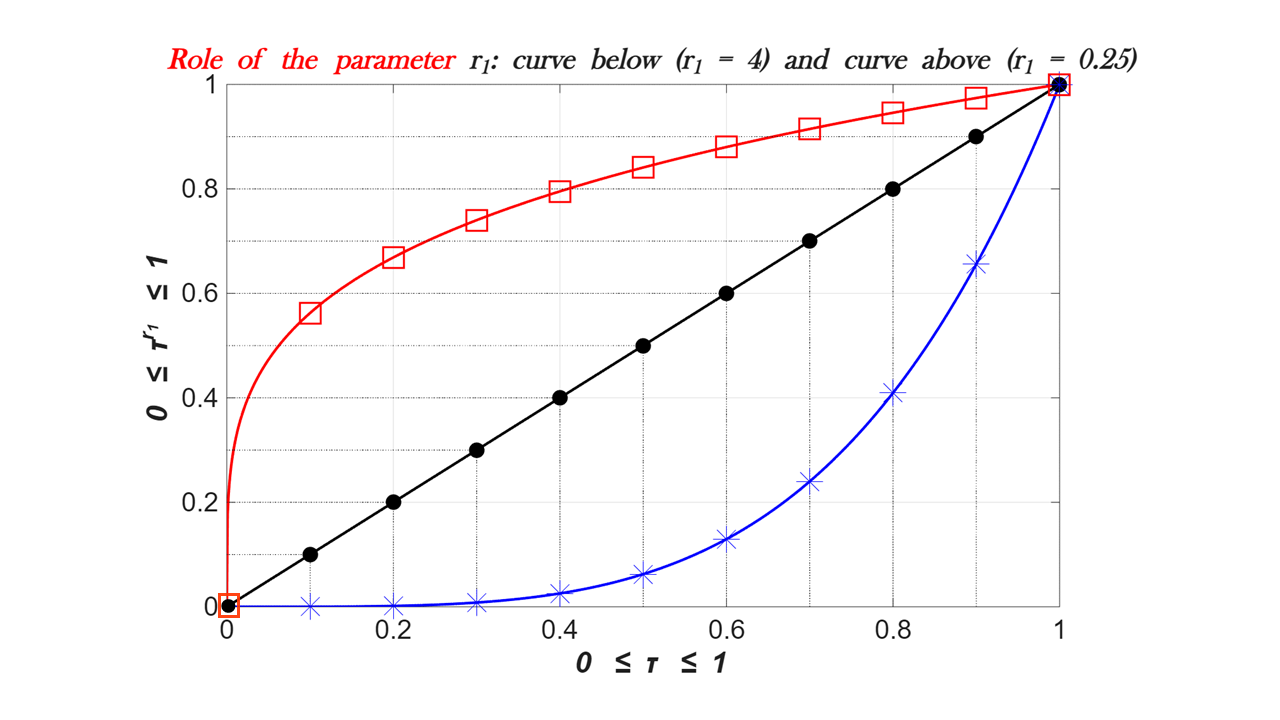}
	\caption{The role of the parameter $r_1$ for the quality profiles in \eqref{eq:def_Qs_b}. The values of $r_1$ larger than 1 imply an {\em expansion} of the plot nearby the point $(0,0)$ (\ie the ordinates of the blue star points -- on the curve $f(\tau)=\tau^{r_1}$ -- are more compressed nearby the origin and more expanded near the point $(1,1)$). The values of $r_1$ smaller than 1 imply a {\em compression} near the point $(1,1)$ (\ie the ordinates of the red square points -- on the curve $f(\tau)=\tau^{r_1}$ -- are more expanded nearby the origin and more compressed near the point $(1,1)$).}
	\label{fig:Expand-Compress}
\end{figure}
Here the idea of how to implement the compression/expansion of the abscissa axis, according to the parameter $r_1$, is emphasized. In particular, the last figure shows how 10 sample values of the precision $\tau$ are generated in our practical implementation of quality profiles, depending on the parameter $r_1$:
\begin{enumerate}
	\item if $r_1=1$ then {\em equally spaced samples} of $0 \leq \tau^{r_1}\leq 1$ (corresponding to ordinates of the black dots on the segment joining $(0,0)$ and $(1,1)$) are generated; 
	\item if $r_1 > 1$ then {\em NOT equally spaced samples} of $0 \leq \tau^{r_1} \leq 1$ (corresponding to the ordinates of the blue stars on the lower curve) are generated; 
	\item if $r_1 < 1$ then {\em NOT equally spaced samples} of $0 \leq \tau^{r_1} \leq 1$ (corresponding to the ordinates of the red squares on the upper curve) are generated.
\end{enumerate}
Observe that according to the first item (where $r_1=1$), in our practical {\tt MATLAB} \cite{MATLAB} implementation for plotting quality profiles (see https://github.com/fasano-g/Quality-Profiles.git) we generated $1000$ sample values for $\tau^{r_1}$. Conversely, in the last two items (where $r_1 \neq 1$) we recommend increasing the number of sample values for $\tau^{r_1}$ (from $1000$ to $\lfloor1000r_1\rfloor$ and $\lfloor 1000 / r_1 \rfloor$, respectively), so that the different density of sample values for $\tau^{r_1}$ in the interval $[0,1]$ is partially compensated (see again the ordinates of the points in Figure~\ref{fig:Expand-Compress}).    

\par\bigskip
\revnewgianni{
\begin{lemma}
		\label{lem4}
		Given the test set ${\cal P}$ and the set of solvers ${\cal S}=\{s_1 , \ldots , s_r\}$, the relative ranking among the nonnegative quantities
		$\left[Q_{s_1}(\tau)\right]^{r_2} , \ldots , \left[Q_{s_r}(\tau)\right]^{r_2}$, when $\tau\in [0, 1]$, is invariant with respect to the choice of the positive parameters $r_1$ and $r_2$.
\end{lemma}
\begin{proof}
Without loss of generality let us focus on the solvers $1$ and $2$, and let us divide the abscissa interval $[0,1]$ into subsets\footnote{Each subset may possibly be a singleton.}, such that in each subset the ranking between the solvers $1$ and $2$ does not change. Let us consider for instance the subset corresponding to $0 \leq \tau_\ell \leq \tau^{r_1} \leq \tau_u \leq 1$, where for simplicity we assume $Q_1(\tau) \geq Q_2(\tau)$. This also implies $\left[Q_1(\tau)\right]^{r_2} \geq  \left[Q_2(\tau)\right]^{r_2}$, for any $r_2 \geq 0$. This proves the invariance of the ranking with respect to the choice of the parameter $r_2$, regardless of the choice of the pair of solvers and of the chosen subset.
\\
Observe that when $\tau_\ell \leq \tau^{r_1} \leq \tau_u$ then equivalently $\tau_\ell^{1/r_1} \leq \tau \leq \tau_u^{1/r_1}$. Hence, if the solvers $1$ and $2$ satisfy $Q_1(\tau) \geq Q_2(\tau)$, for  $\tau_\ell \leq \tau^{r_1} \leq \tau_u$, then we also have $Q_1(\tau^{1/r_1}) \geq Q_2(\tau^{1/r_1})$, for any value of $\tau$ in the interval $\tau_\ell^{1/r_1} \leq \tau \leq \tau_u^{1/r_1}$. This shows that the ranking between  $Q_1(\tau)$ and $Q_2(\tau)$, in the interval $\tau_\ell \leq \tau^{r_1} \leq \tau_u$, is maintained also by $Q_1(\tau^{1/r_1})$ and $Q_2(\tau^{1/r_1})$ in the interval $\tau_\ell^{1/r_1} \leq \tau \leq \tau_u^{1/r_1}$.
\\
The arbitrary choice of the pair of solvers and of the chosen subset proves that the relative ranking among the solvers is preserved after changing $r_1$, too. 
\end{proof}
}

\subsection{Relative and absolute accuracy via quality profiles}
\label{sec:rel_abs_quality}
This section is devoted to analyzing to what extent quality profiles outcomes may address both the {\em relative} and the {\em absolute} precision of solvers. In particular, \revnewgianni{recalling that the proposal in the current paper is inherently test set--dependent,} the last terminology refers to the following distinction:
\begin{itemize}
    \item {\em relative precision} of a solver refers to the comparative performance of a given solver with respect to other solvers, on a given test set\gianni{;}
    \item {\em absolute precision} of a solver intends to monitor the numerical accuracy of a solver, as expressed in terms of number of correct digits in the solution it provides.
\end{itemize}
It is straightforward to see that, on a given test set, performance profiles \gianni{\cite{dolan.more}} typically yield comparative (\ie relative) results in terms of  solvers' performance, so that no specific information is provided as regards the \gianni{absolute precision}. Conversely, data profiles \gianni{\cite{more.wild} on one hand provide outcomes on the relative precision of solvers (by selecting appropriate precision levels), and on the other hand are also} focused on exploiting the accuracy of solvers, for a required precision level. Hence, we can summarize that in case a solver is omitted from the comparison, in the resulting performance profile each of plots associated to the other solvers is likely influenced. On the contrary, omitting a solver from a data profile implies that the plots relative to the other solvers will remain unchanged.
\\
We show here that quality profiles indeed yield, to large extent, information on both the relative and the absolute precision of solvers, following the guidelines indicated in \gianni{\cite{beiranvand2017best}}\revgianni{, too}. Indeed, considering for any $p \in {\cal P}$ and for $s \in {\cal S}$ the assumptions
    $$ f^{(p)} (x_0^{(p)}) \geq f_s^{(p)} (x^{\ast}) \geq  f_L^{(p)} $$
(see also Lemma \ref{lem3}, along with the second and the third choice for $f_L^{(p)}$ at page \pageref{page_4}), we have
    $$ 1 \geq \tau^{r_1} \geq \frac{f_s^{(p)} (x^{\ast}) - f_L^{(p)}}{f^{(p)} (x_0^{(p)}) - f_L^{(p)}} \geq  0. $$
Hence, using the base $10$ logarithm we obtain
    $$ 0 \leq -r_1 \log_{10}(\tau) \leq - \log_{10} \left[\frac{f_s^{(p)} (x^{\ast}) - f_L^{(p)}}{f^{(p)} (x_0^{(p)}) - f_L^{(p)}} \right]. $$
\revgianni{This last relation is subject to the following  couple of relevant interpretations, in case $\tau=0.1$ and $r_1 \in \{1,2, \ldots\}$:
\begin{itemize}
	\item when $f_L^{(p)}=0$, then $r_1$ may be associated with the {\em minimum number of exact digits} that the solver $s$ has correctly computed, when minimizing the initial objective function value $f^{(p)} (x_0^{(p)})$ of the test problem $p$;
	\item when $f_L^{(p)} \neq 0$, then $r_1$ may be associated with the {\em minimum number of exact digits} that the solver $s$ has correctly computed when decreasing the quantity ${f^{(p)} (x_0^{(p)}) - f_L^{(p)}}$ to eventually obtain the quantity ${f_s^{(p)} (x^{\ast}) - f_L^{(p)}}$.
\end{itemize}
}
\revgianni{
Hence, let us consider now the entire test set, along with a vertical line by any of the abscissa values $0.1,0.1^{2}, \ldots, 0.1^{r_1}$. These lines  intersect the quality profile of a given solver $s \in {\cal S}$ at values indicating {\em the percentage of test problems where the solver $s$ is able to successfully decrease the quantity ${f_s^{(p)} (x^{\ast}) - f_L^{(p)}}$, $p \in {\cal P}$, by at least $1,2 \ldots, r_1$ orders of magnitude}. This observation allows us to appropriately select the value of $r_1$, that should be chosen so that the vertical lines corresponding to the values $0.1,0.1^{2}, \ldots, 0.1^{r_1}$ are clearly distinguishable. As a direct consequence, {\em the use of semilog scaling plot} (\ie setting log scaling on the abscissa axis and linear scaling on the ordinate axis) {\em in quality profiles is definitely recommended}. 
}

\revgianni{
Furthermore, the ordinate of the intersection point between the $s$-th quality profile and the vertical line $\tau = 0.1^q$, with $1 \leq q \leq r_1$, represents the ordinate of the (right) horizontal asymptote to the plot of the $s$-th solver, in a data profile,  where the chosen precision level is exactly $0.1^q$.	
}
\revgianni{
To better grasp the different information summarized by quality profiles, with respect to data profiles, we can observe\footnote{We warn the reader that with a little abuse notation, but for the sake of simplicity, we preferred to use the symbol $[Q_s(\tau)]^{r_2}$ on the ordinate axis in Figures \ref{fig:QP-prototype} and \ref{fig:QP-prototype-bis}, rather than $[Q_s(\tau^{r_1})]^{r_2}$. A similar choice was adopted along the entire paper for the remaining figures.} Figure~\ref{fig:QP-prototype} and Figure~\ref{fig:QP-prototype-bis}, where two solvers are easily compared over a scale of 30 orders of magnitude on the objective function value. By suitably choosing values of $r_1$ we can easily zoom in a (right) neighborhood of the value $\tau=0$. Indeed, here the combined use of the semilog scale and the value of $r_1$ yield a clear idea of both the relative and the absolute decrease of the objective function value, obtained through both the solvers.
} 
\begin{figure}[tbph]
	\centering
	\includegraphics[width=0.9\textwidth]{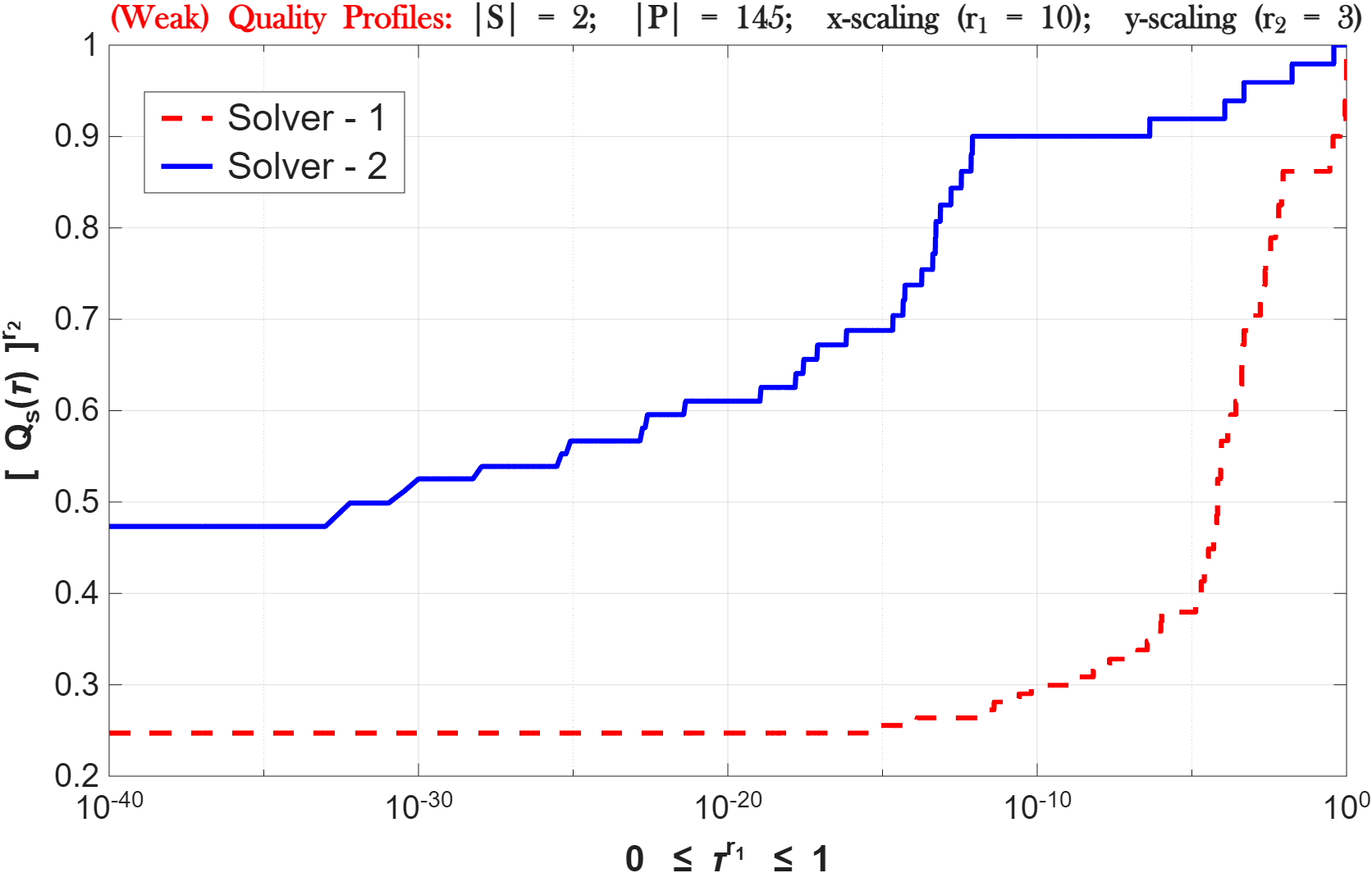}
	\caption{Quality profiles obtained setting $r_1=10$ and $r_2=3$ in \eqref{eq:def_Qs_b}.}
	\label{fig:QP-prototype}
\end{figure}
\begin{figure}[tbph]
	\centering
	\includegraphics[width=0.9\textwidth]{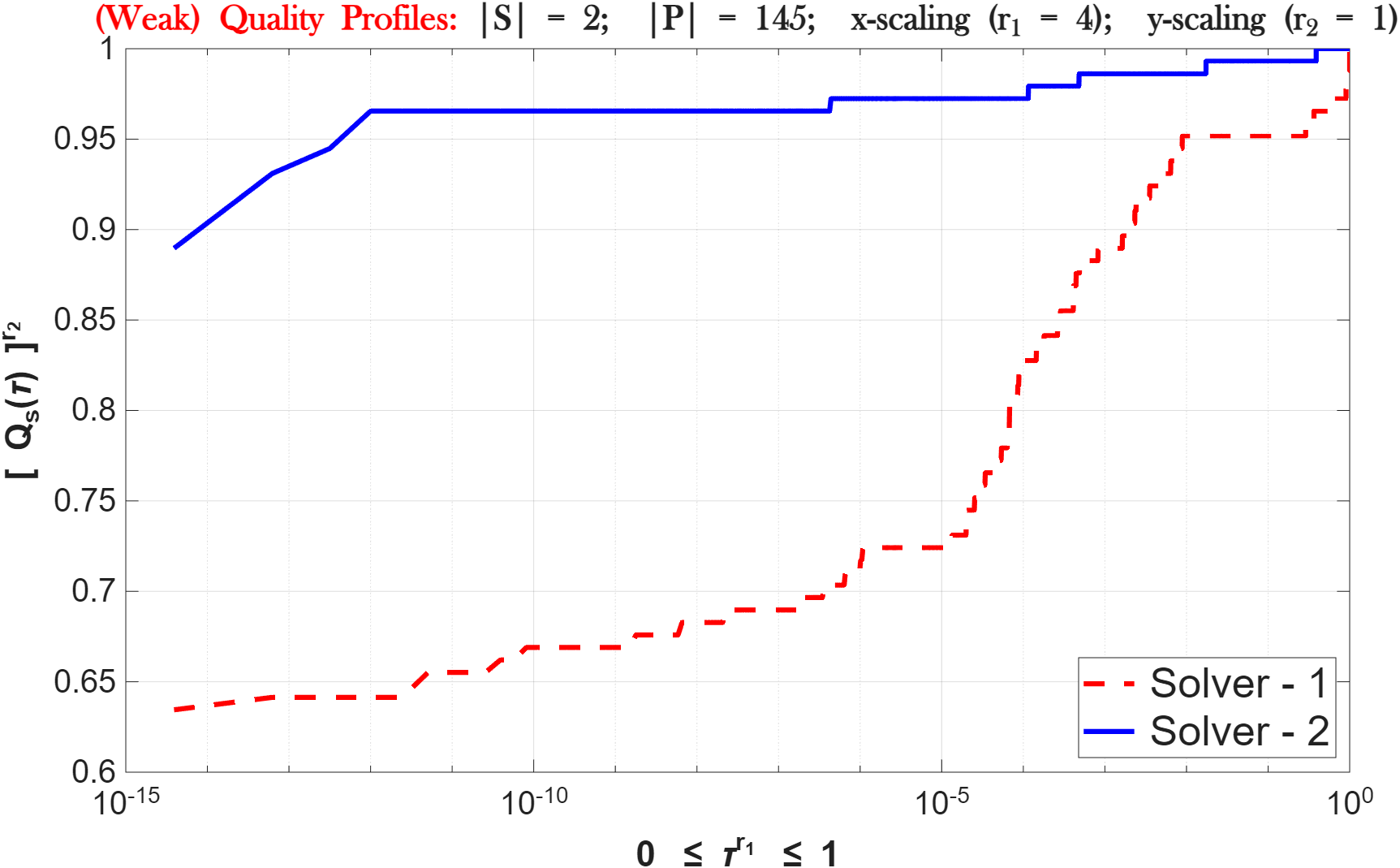}
	\caption{Quality profiles obtained setting $r_1=4$ and $r_2=1$ in \eqref{eq:def_Qs_b}.}
	\label{fig:QP-prototype-bis}
\end{figure}
In particular, from Figure \ref{fig:QP-prototype} we can immediately realize that the Solver 1 is able to beat (by at least 10 orders of magnitude) the initial value $f^{(p)} ( x_0^{(p)} ) - f_L^{(p)}$ in about 30\% of the problems in the test set. Conversely, the Solver 2 is able to decrease the same quantity by about 12 orders of magnitude, for more than 90\% of the problems in the test set. In any case, since $\left[Q_2(\tau)\right]^{r_2} \geq \left[Q_1(\tau)\right]^{r_2}$, for any $\tau \in [0,1]$, then the second solver is uniformly (\ie for any $\tau \in [0,1]$) preferable. Also observe that neither of the solvers quality profile includes the point $(0,0)$ (\ie in each of the solvers $s \in {\cal S}$ we have $f_s^{(p)}(x^*) = f_L^{(p)}$, for at least one problem $p \in {\cal P}$), and only $\left[Q_2(\tau)\right]^{r_2}$ includes the point $(1,1)$ (\ie the first solver has a failure for at least one problem $p \in {\cal P}$). Moreover, the use of the semilog scale on the abscissa axis and the choice of the positive parameter $r_1$ are independent issues, and they may both contribute to enhance the final readability of the outcomes.
\\
We also remark that setting $r_2 \neq 1$ for a quality profile has a less relevant role with respect to carefully tuning the parameter $r_1$. Indeed, $r_2$ can be regarded as a scaling parameter in order to better distinguish and rank the tracks associated to different solvers, in case of very similar (and possibly messy) drawings within a restricted area of the picture. We can observe the impact of setting $r_2 \neq 1$ in Figure \ref{fig:QP-prototype}, where the ordinates of the points in both the quality profiles are scaled through $r_2$. This fact may be possibly useful in some circumstances, but it also implies the next misleading performance for the solvers. Comparing Figure \ref{fig:QP-prototype} (where $r_2=3$) and Figure \ref{fig:QP-prototype-bis} (where $r_2=1$), and  selecting for instance the value $10^{-10}$ on the abscissa axis, we can observe that the reported percentages of solved problems associated to each solver \revnewgianni{differ} from one figure to the other. Hence, the user should pay attention about this possible drawback when setting $r_2 \neq 1$ and drawing conclusions. \revnewgianni{As a general indication on the choice of the parameters $r_1$ and $r_2$, we advise to select $r_1=r_2=1$ as default values. Then, in case the resulting drawings associated with the solvers are hardly distinguishable, then the user may consider to change $r_1$ (to compress/expand the drawings nearby a value on the abscissa axis), and $r_2$ (to compress/expand the drawings nearby a value on the ordinate axis).}

\revgianni{A similar narrative is suggested also for the proposal in \cite{beiranvand2017best} (where namely the {\em accuracy profiles} are described), that provides an analogous analysis but requires the assessment of an exogenous parameter $M$, in order to plot the profiles. Moreover, the use of the scaling parameter $r_1$ in \eqref{eq:def_Qs_b} is not immediately extendable to the analysis in \cite{beiranvand2017best}. Anyway, we can conclude that both the proposal in the current paper and the one in \cite{beiranvand2017best} give benchmarking information in terms of relative and absolute precision. However, the proposal in the current paper does not require {\em any pre--processing operation} (\eg clustering too close values provided by solvers): indeed, the joint use of $r_1$ and semilog scale allows to clearly discriminate the accuracy among solvers, up to any desired accuracy level (say number of decimal digits). 
}
\\
For the convenience of the reader, the next lemma summarizes a few relevant properties of quality profiles (similarly to data profiles) in terms of relative precision.
\\

\begin{lemma}
\label{lem:robustness_ratios}
Let us consider the set ${\cal P}$ of test problems. Then, 
\begin{itemize}
	\item in case the quantity $f_L^{(p)}$ in \eqref{eq:def_Qs_b} is computed independently of the solvers, then for $s,t \in {\cal S}$ the ratios $Q_s(\tau)$, $Q_t(\tau)$, for any $\tau \in [0,1]$, are independent plots, \ie each of them is not influenced by the others;
	\item in case the quantity $f_L^{(p)}$ in \eqref{eq:def_Qs_b} is defined as $f_L^{(p)}= \min_{s \in \cal S} \{f_s^{(p)}(x^*)\}$ (weak quality profiles), or $f_L^{(p)}= \max_{s \in \cal S} \{f_s^{(p)}(x^*)\}$ (strong quality profiles), then for $t \in {\cal S}$ the ratio $Q_t(\tau)$, for any $\tau \in [0,1]$, will uniquely depend on the solvers in ${\cal S}$ used to assess the quantity $f_L^{(p)}$, with $p \in {\cal P}$.
\end{itemize}
\end{lemma}

{The two cases in Lemma~\ref{lem:robustness_ratios} bear directly on the paradox-freeness discussion in the Introduction: an algorithm-independent value for $f_L^{(p)}$ avoids by construction the set-dependence mechanism behind the paradoxes in \citep{liu2020paradoxes}, whereas the default weak/strong formulations do not offer this guarantee against the sufficient conditions derived in \citep{yan2022paradox}, falling into the survival of the nonfittest paradox. A full verification of this property is left for future work.}

\par\medskip
\begin{observation}
	{\em 
\revgianni{First we observe the substantial complementary nature of {\em quality profiles} information with respect to {\em data profiles}. Indeed, in each quality profile we essentially disregard the maximum computational budget allowed for the $s$--th solver (\eg the maximum number of function evaluations allowed, the maximum number of iterations, etc.), which is a typical setting of most of the available solvers. On the contrary, we monitor the performance (\ie the percentage of successfully solved problems in the test set) for the $s$--th solver as by \eqref{eq:def_Qs_b}, and for any given precision value (\ie $0 \leq \tau \leq 1$). Conversely, in a data profile the required precision $\tau$ for all the compared solvers is assigned, so that for any solver, the ratio of successfully solved problems in the test set is reported vs. the number of simplex gradients (say equivalently the multiple of $n_p+1$ function evaluations, being $n_p$ the number of variables for the test problem $p \in {\cal P}$).}}
\end{observation}

\subsection{Application of the quality profiles to smooth optimization}
\label{sec:application}
\rev{\chri{To show the actual effectiveness of quality profiles} we first focus on large scale unconstrained optimization problems. Hence, we select all the large scale unconstrained problems in the {\sf CUTEst} collection \cite{cutest} \chri{which represent our test set $\cal{P}$}. Namely, we consider all the problems with variable dimension, and for each problem we consider three instances corresponding to ${n=1,000}$, $n=5,000$ and $n=10,000$ (unless some specific \gianni{limits} on the problem dimension are \gianni{imposed}). The resulting set ${\cal P}$ of test problems consists of 166 problems and they are included in Table~\ref{tab:testproblems}.}
\begin{table}[htbp]
{
		\centering
		\rev{
			\begin{tabular}{l@{\hskip 1.0in}l@{\hskip 1.0in}l}
				\textit{Problem} & \textit{Problem} & \textit{Problem} \\
				\midrule
				ARWHEAD & DIXMAANK$^*$ & NONCVXUN2 \\
				BDQRTIC & DIXMAANL$^*$ & NONDIA \\
				BROYDN7D& DQDRTIC & NONDQUAR \\
				BRYBND  & DQRTIC & PENALTY1 \\
				CHAINWOO& EDNSCH & POWELLSG \\
				COSINE  & ENGVAL1 & POWER \\
				CRAGGLVY& FLETCBV2 & QUARTC \\
				CURLY10 & FLETCBV3 & SCHMVETT \\
				CURLY20 & FLETCHCR & SINQUAD \\
				CURLY30 & FMINSURF$^{**}$ & SPARSINE \\
				DIXMAANA$^*$ & FREUROTH & SPARSQUR \\
				DIXMAANB$^*$ & GENHUMPS & SPMSRTLS \\
				DIXMAANC$^*$ & GENROSE & SROSENBR \\
				DIXMAAND$^*$ & LIARWHD & TESTQUAD \\
				DIXMAANE$^*$ & MOREBV & TOINTGSS \\
				DIXMAANF$^*$ & MSQRTALS$^{***}$ & TQUARTIC\\
				DIXMAANG$^*$ & MSQRTBLS$^{***}$ & TRIDIA\\
				DIXMAANH$^*$ & NCB20 & VARDIM \\
				DIXMAANI$^*$ & NCB20B & VAREIGVL \\
				DIXMAANJ$^*$ & NONCVXUN & WOODS\\
				\midrule
			\end{tabular}	
			\caption{\rev{List of the 166 \gianni{large scale unconstrained optimization} problems from {\sf CUTEst} collection included in the set ${\cal P}$.
					Each problem is considered in three different dimensions, namely $n \gianni{ \in \{1,000; 5,000; 10,000\}}$, with  the following few  exceptions:
					$^*$: $n \gianni{ \in \{1,500; 3,000\}}$;  $^{**}$: $n \gianni{ \in \{1,024; 5,625; 10,000\}}$;  and  finally $^{***}$: $n \gianni{ \in \{1,024; 4,900\}}$.}}\label{tab:testproblems}}
}
\end{table}
	%
	
	%
	%

\subsubsection{Results of a first numerical experience with quality profiles for smooth optimization}\label{sec:results-first}
\rev{As a first example of the application of the weak quality profiles, we select four well--known \chri{algorithms} for large scale unconstrained optimization from literature.}
%
%
We wish to recall that the choice of the codes was mainly suggested by the \rev{following} motivations:
\begin{itemize}
    \item all the selected \chri{optimization methods} are well--known and of easy access, so that a full \rev{reproducibility} of our numerical outcomes, along with their interpretation, is \gianni{easily} allowed for readers;
    \item \chri{each} code can be interfaced through a unique platform, and each \gianni{code} is available for testing on a widely acknowledged environment (namely {\sf CUTEst} \cite{cutest}). Moreover, observe that the interfaces of the \gianni{selected} codes within {\sf CUTEst} test environment are already available \chri{as well};
    \item our choice of different and nonhomogeneous solvers, \ie Truncated Newton's methods, \rev{Quasi--Newton methods}, Nonlinear Conjugate Gradient schemes, etc. is done on purpose. \rev{Also note that the selected \chri{algorithms} have different requirements, in terms of first-order or second-order information needed.}
    The idea is that of using complete numerical results possibly containing dissimilarities of performance, so that quality profiles could be adequately tested on a meaningful \chri{and diverse} set of codes. {\em Hence, we categorically exclude from our numerical experience any aim of appointing preferences for specific codes, based on \gianni{the description and the analysis of} their performance}.
\end{itemize}
\rev{
The list of the selected solvers follows\gianni{. It} includes: the code name, the name of the corresponding {\sf CUTEst} package, the reference paper[s] and a brief description.
\rev{As regards the parameters, for all the codes we use \gianni{their own} default ones. \gianni{Moreover, for all codes we} set the maximum number of iterations allowed to 500,000.}
\begin{itemize}
	\item {\sf L--BFGS} ({\tt lbfgs}) \cite{nocedal:80,liu.nocedal}: this code implements the Limited-Memory BFGS method
	designed for unconstrained minimization by J. Nocedal.
    We adopt the default ``memory'' value, namely $M=5$; the  stopping criterion \chri{imposed} by the author is
$ \| \nabla f(x_k) \|_2  \leq 10^{-5}  \max \{ 1 , \|x_k\|_2\}$.
\\
\item {\sf CG+}   ({\tt cgplus}) \cite{gilbert.nocedal}: it is a Nonlinear Conjugate Gradient algorithm
designed by G. Liu, J. Nocedal and R. Waltz. \chri{It allows to select three different}
Conjugate Gradient strategies, namely Fletcher--Reeves, Polak--Ribi\`ere,
and positive Polak--Ribi\`ere. We adopt the positive Polak--Ribi\`ere \gianni{formula} without restart with the default stopping criterion
$\|\nabla f(x_k)\|_{\infty} \leq 10^{-5}( 1 + |f(x_k)|)$.
\\
\item {\sf TRON} ({\tt tron}) \cite{lin.more}: it implements a Trust Region Newton method for large scale unconstrained or bound constrained minimization, designed by C. Lin and J. Mor\'e. The default stopping criterion is
$\|\nabla f(x_k)\|_{\infty} \leq 10^{-5}$.
\\
\item {\sf CG$\_$DESCENT}  ({\tt cg$\_$descent}) \cite{hager.zhang:05,hager.zhang:06b}: it is a Nonlinear Conjugate Gradient
method designed by W. Hager and H. Zhang. An efficient linesearch is adopted and a novel convergence criterion (approximate Wolfe conditions) allows \chri{to obtain} high accuracy. The default stopping criterion is
$\|\nabla f(x_k)\|_{\infty} \leq 10^{-5}$.
\end{itemize}
\gianni{Thus, on the overall the} \chri{complete set \gianni{${\cal S}$} of adopted solvers} is \gianni{given by}
    $${\cal S}=\{\text{{\sf L--BFGS}, {\sf CG+}, {\sf TRON}, {\sf CG$\_$DESCENT}}\}. $$}
\rev{
The performance of these codes have been \chri{extensively} studied. Now we aim at assessing their capability to determine better local minimizers (\ie solution points with lower objective function value) by using the quality profiles introduced in the previous section.}
\par
\rev{In the following we report the quality profiles concerning the comparison among the four codes, recalling that \gianni{for the problem $p \in {\cal P}$} we use as reference value \gianni{the quantity (see also \revmax{items }1.--3. at page \pageref{page_4})}
	$$ f_L^{(p)} = \min_{s \in {\cal S}} \left\{ f^{(p)}_s(x^\ast) \right\}.$$
In particular, Figure~\ref{fig:1new} and Figure~\ref{fig:1new-b} depict the plot of the quality profiles obtained by setting respectively $r_1=1$ with $r_2=1$, and $r_1=5$ with $r_2=1$ in \eqref{eq:def_Qs_b}.}
	\begin{figure}[htbp]
		\centering
		\includegraphics[width=0.85\textwidth]{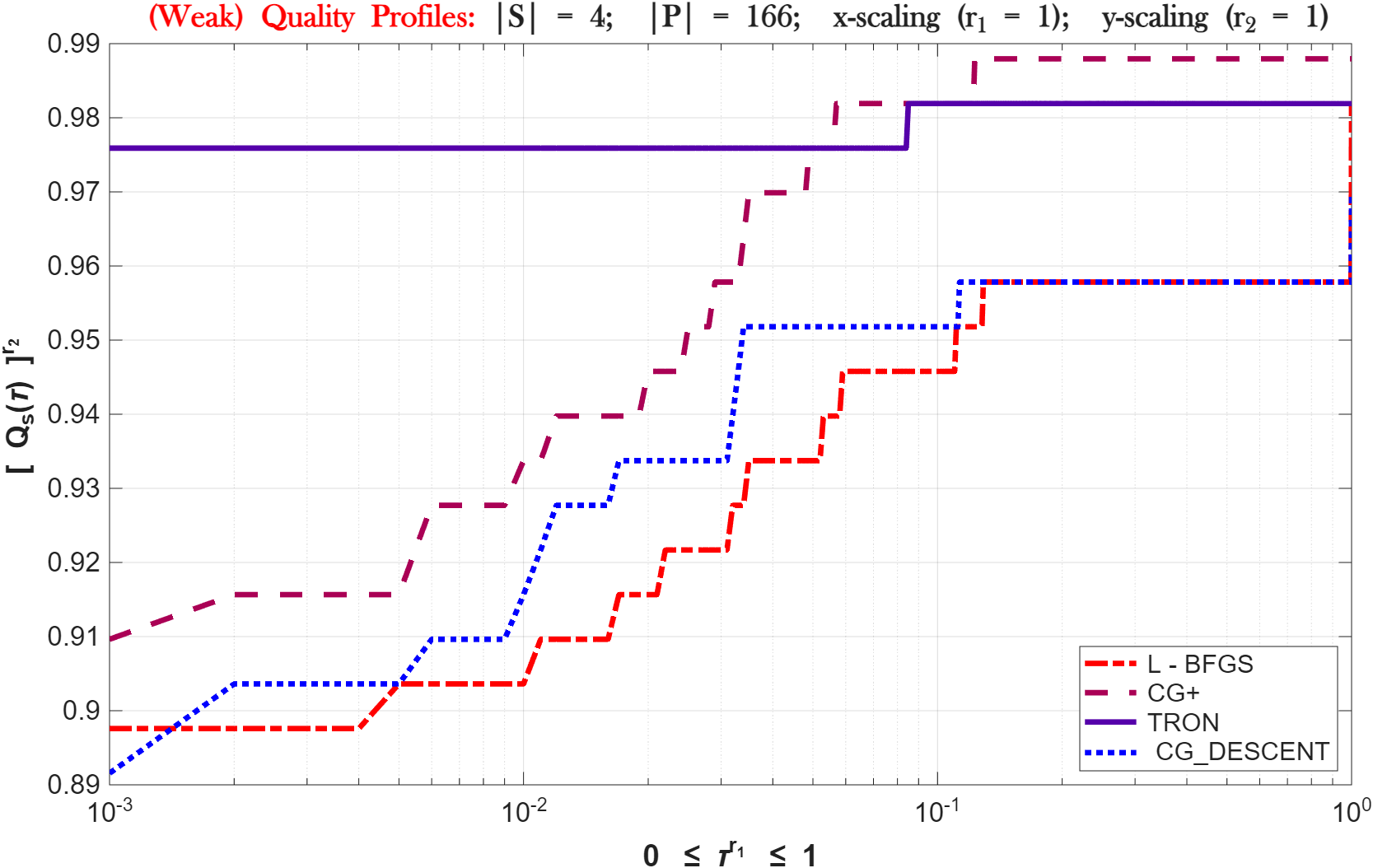}
		\caption{Quality profiles comparing {\sf L--BFGS}, {\sf CG+}, {\sf TRON}, {\sf CG$\_$DESCENT}, obtained with $r_1=1$ and $r_2=1$ in \eqref{eq:def_Qs_b}. {\sf TRON} appears to outperform the other three solvers, when {\em  a relatively low accuracy interval is investigated}.}
		\label{fig:1new}
	\end{figure}
\begin{figure}[htbp]
	\centering
	\includegraphics[width=0.85\textwidth]{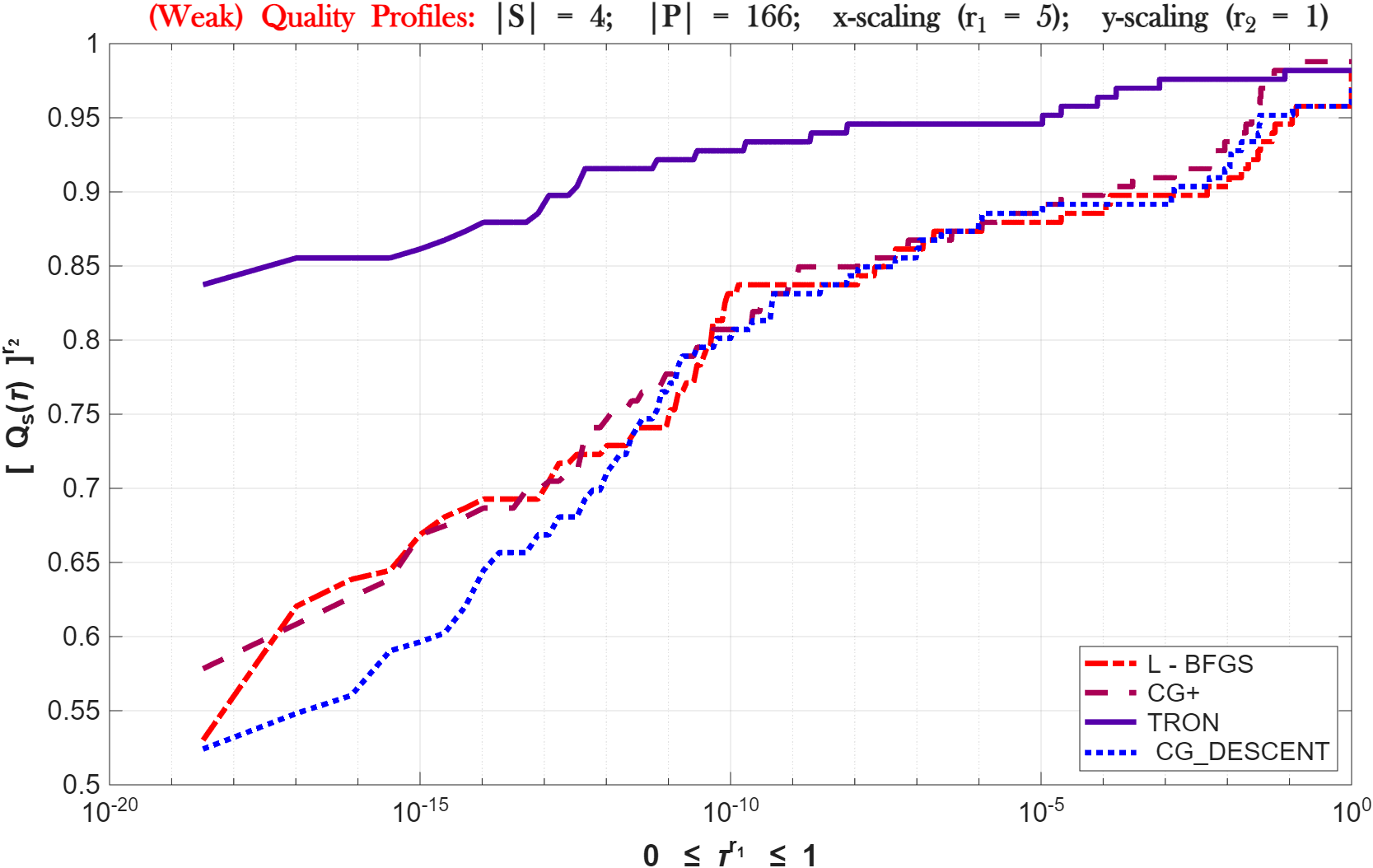}
	\caption{Quality profiles comparing {\sf L--BFGS}, {\sf CG+}, {\sf TRON}, {\sf CG$\_$DESCENT}, obtained with $r_1=5$ and $r_2=1$ in \eqref{eq:def_Qs_b}. Again {\sf TRON} appears to outperform the other three solvers, even in case {\em a relatively high accuracy interval is investigated}.}
	\label{fig:1new-b}
\end{figure}

\gianni{In other words the abscissa axis reports $\tau^{r_1}$ while the ordinate axis reports the quantities $\{[Q_s(\tau)]^{r_2}\}$. By using the semilog scaling, equally spaced tickers on the abscissa axis correspond to a $\times 10$ multiplication. Hence, $r_1=1$ (Figure~\ref{fig:1new}) implies that the quality profile reports results for a precision range given by $\tau^{r_1} \in [10^{-3},1]$ (low accuracy interval). Conversely, in case $r_1 > 1$ (Figure~\ref{fig:1new-b}) the quality profile, on average, reports results for a precision range $[10^{-3(r_1+1)},1]$ (high accuracy interval), as motivated in Observation \ref{observ:5}.}
\rev{The best code in terms of percentage of problems where the best value of the objective function  \chri{is reached} (\ie the intersection of the profiles with the leftmost vertical line) seems to be {\sf TRON}. Conversely, as concerns the percentage of failures (\ie the intersection of the profiles with the vertical line $\tau^{r_1}=1$) {\sf CG+} appears slightly preferable. However, from Figure~\ref{fig:1new} and Figure~\ref{fig:1new-b}, we easily realize how difficult it can be to clearly compare the solvers especially for small values of $\tau$, due to close tracks associated to the profiles. Therefore, based on Observation~\ref{observ:5}, we can rescale the axes to expand portions of the plots, by choosing different values of $r_1$ and $r_2$. Figures~\ref{fig:new6plots} and \ref{fig:new6plots_b} report plots corresponding to different choices of $r_1$ and $r_2$.}
\begin{figure}[htbp]
	\centering
	\begin{subfigure}{0.75\textwidth}
		\includegraphics[width=\textwidth]{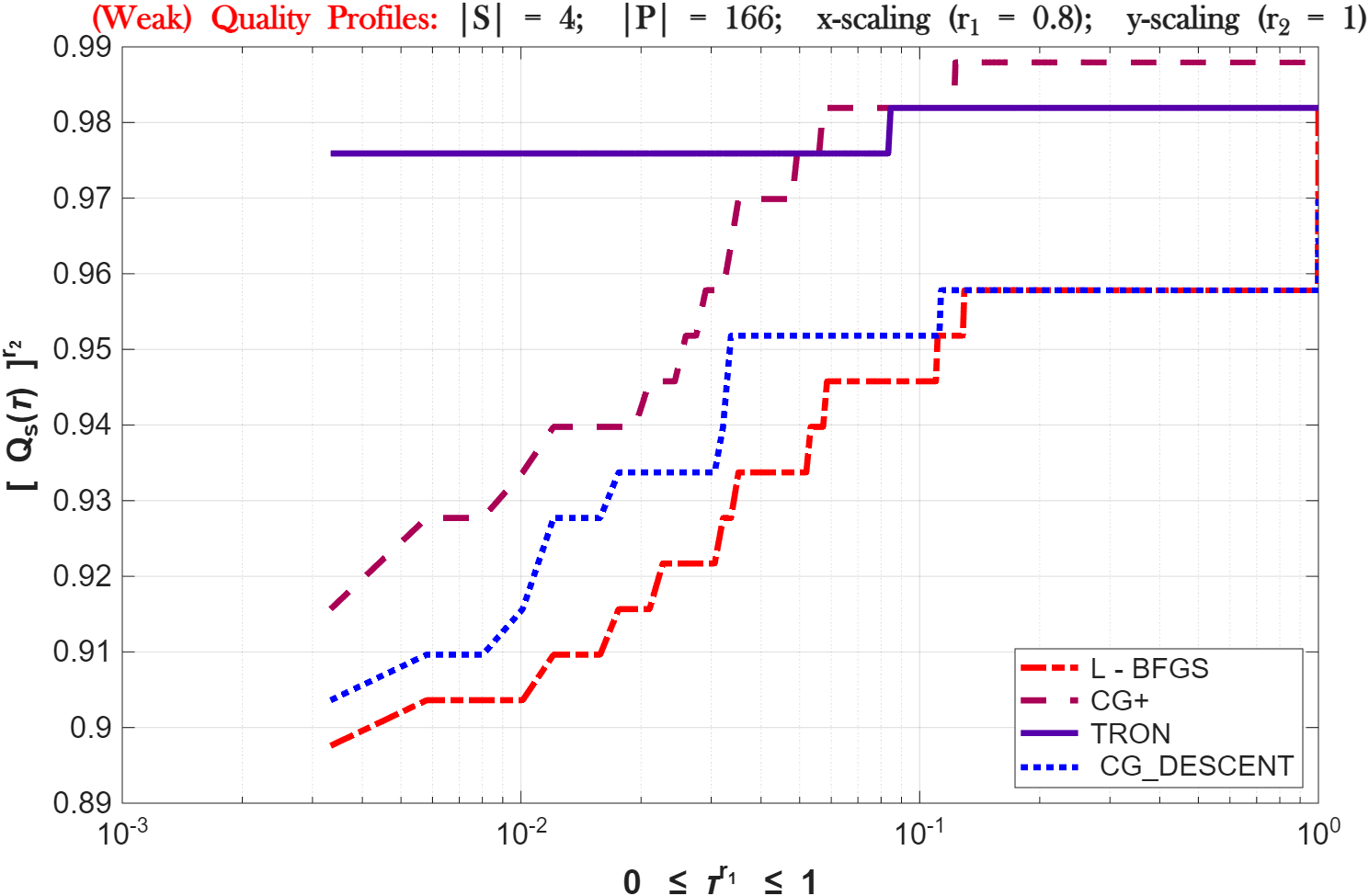}
		\label{fig:newa}
	\end{subfigure}
	\hfill
	\begin{subfigure}{0.75\textwidth}
		\includegraphics[width=\textwidth]{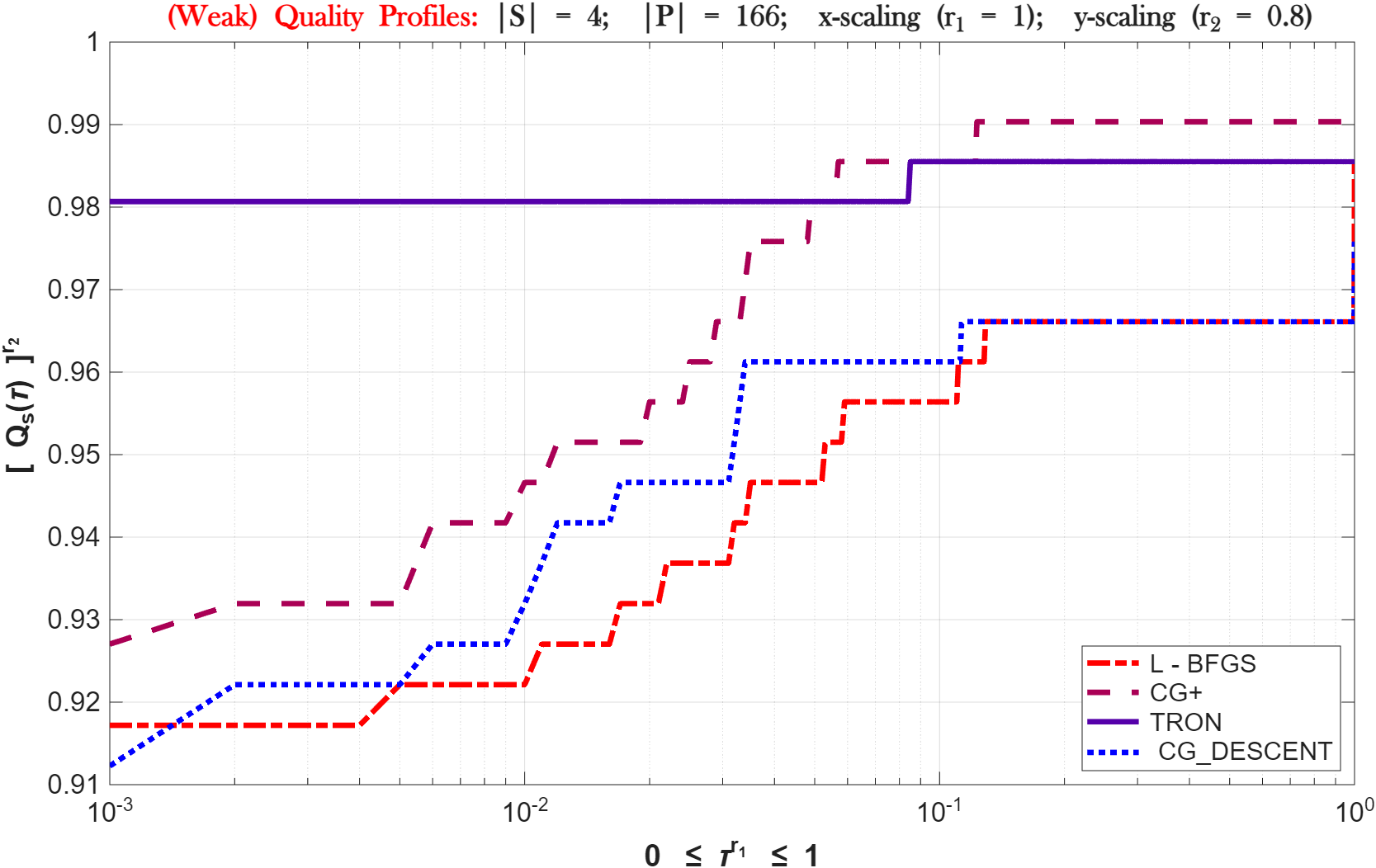}
		\label{fig:newb}
	\end{subfigure}
	\begin{subfigure}{0.75\textwidth}
		\includegraphics[width=\textwidth]{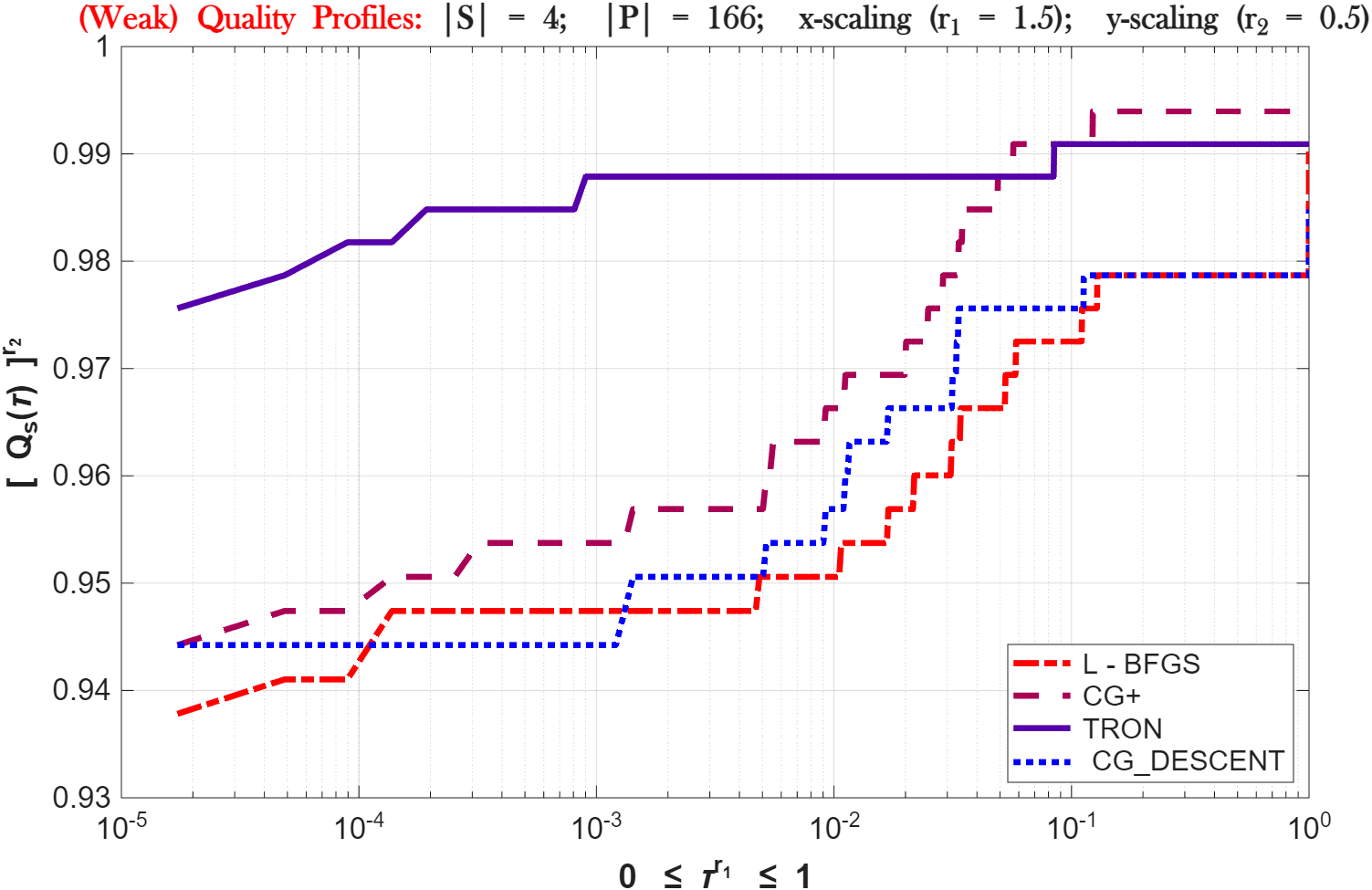}
		\label{fig:newc}
	\end{subfigure}
	\caption{\gianni{Ex}amples of quality profiles for the solvers {\sf L--BFGS}, {\sf CG+}, {\sf TRON}, {\sf CG$\_$DESCENT}\gianni{.}}
	\label{fig:new6plots}
\end{figure}
\begin{figure}[htbp]
	\centering
	\begin{subfigure}[b]{0.75\textwidth}
		\includegraphics[width=\textwidth]{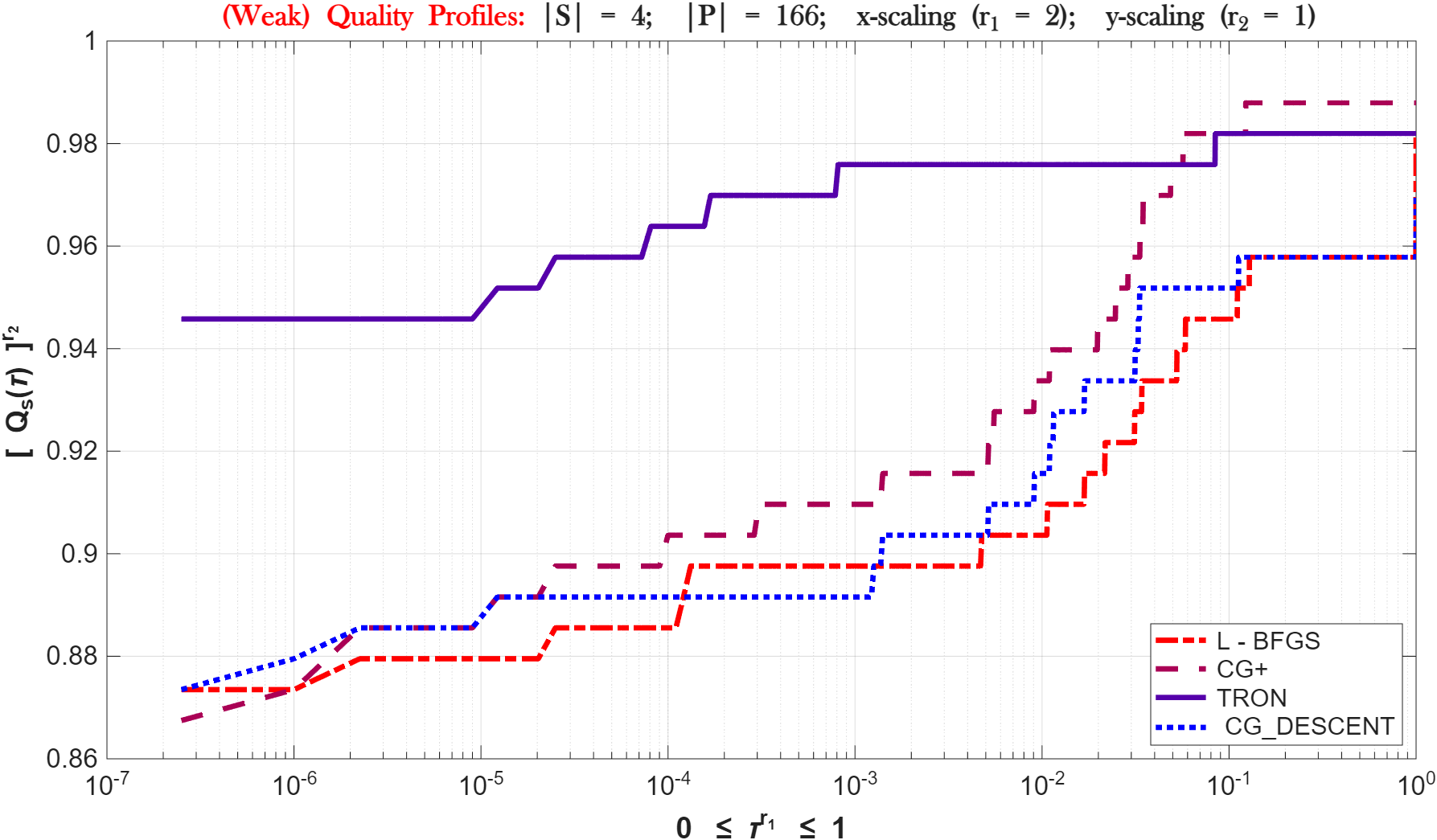}
		\label{fig:newd}
	\end{subfigure}
	\begin{subfigure}[b]{0.89\textwidth}
		\includegraphics[width=\textwidth]{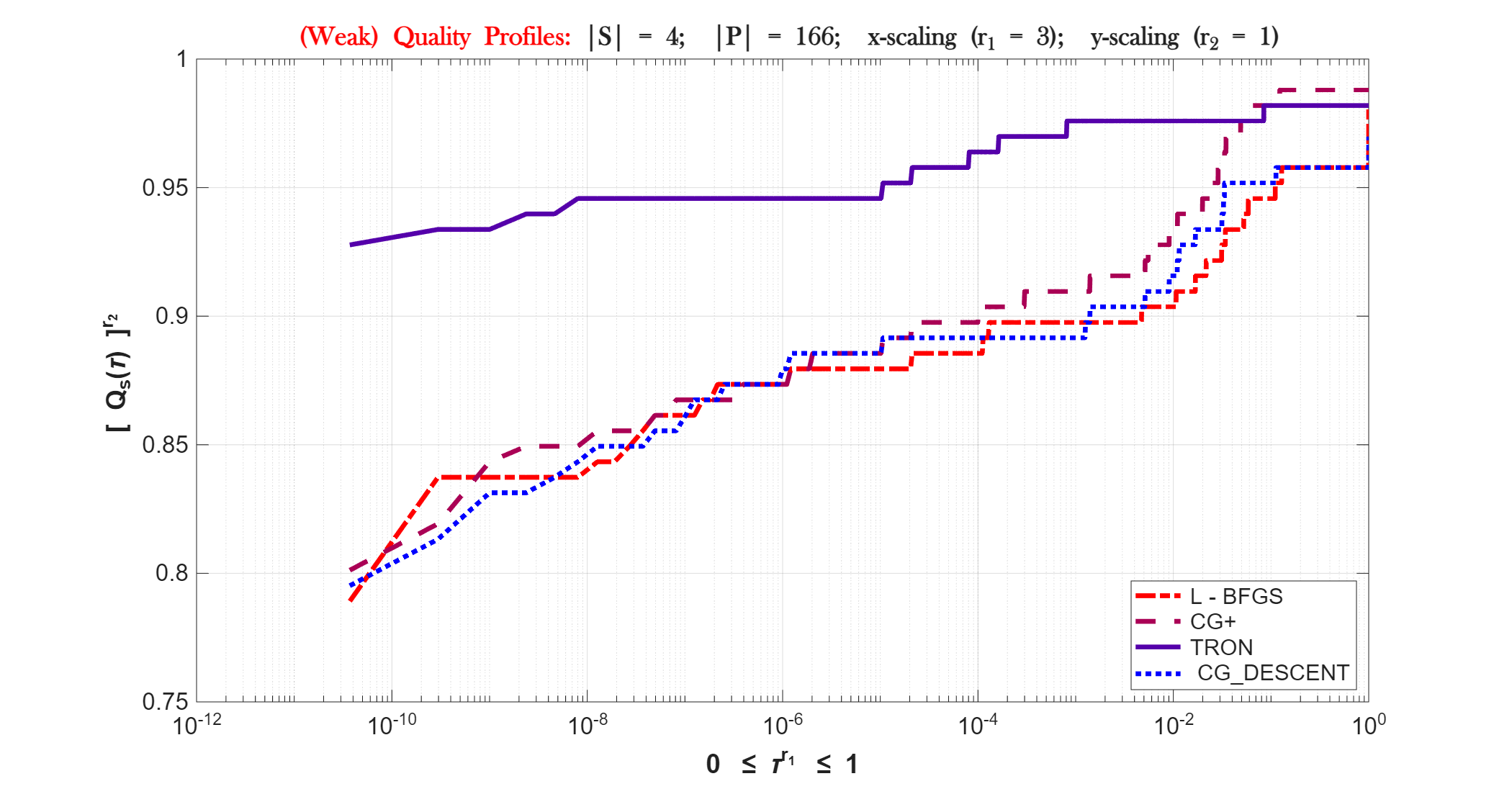}
		\label{fig:newe}
	\end{subfigure}
	\hfill
	\begin{subfigure}[b]{0.75\textwidth}
		\includegraphics[width=\textwidth]{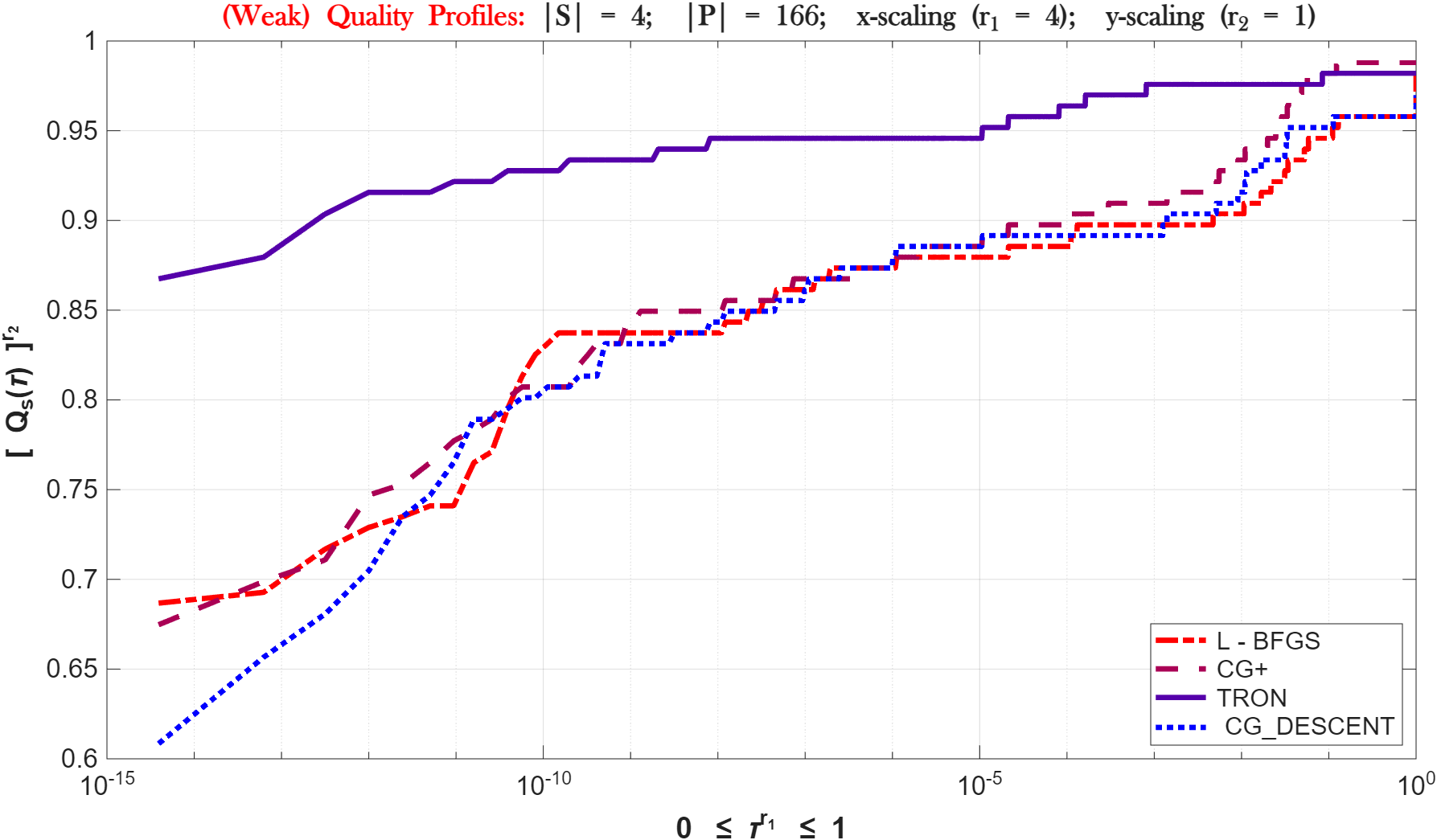}
		\label{fig:newf}
	\end{subfigure}
	\caption{\gianni{Ex}amples of quality profiles for the solvers {\sf L--BFGS}, {\sf CG+}, {\sf TRON}, {\sf CG$\_$DESCENT}\gianni{.}}
	\label{fig:new6plots_b}
\end{figure}
The expansion of the plots as consequence of choosing $r_1$ and/or $r_2$ \gianni{larger} than one is clearly shown. 

Conversely, expanding the portion of the plots corresponding to large value of the abscissa (\ie $\tau^{r_1}$ close to one) clearly highlights that not all the four tracks reach the value one.
This can be explained by recalling that the piece of information related to the percentage of failures associated to each solver is properly taken into account by quality profiles\gianni{, too.} The last consideration can be regarded as a counterpart of a \chri{corresponding} \gianni{well--known} property that holds for performance and data profiles.
This \gianni{also} allows us to draw conclusions in terms of robustness of the four algorithms (similarly to the performance profiles). 
\\
\revgianni{We complete this section by reporting in Figure~\ref{fig:1new-lin} the same quality profiles in Figure~\ref{fig:1new}, but where the semilog scaling is replaced by linear scaling. Drawing vertical lines by the abscissa values $0.1, 0.1^2, \ldots, 0.1^{r_1}$ gives some indication on the {\em absolute precision} of quality profiles, even when using linear scaling, in accordance with the conclusions in Section \ref{sec:rel_abs_quality}. Note that in this specific case, linear scaling is unable to clearly depict the comparison among solvers, because for $\tau^{r_1} \geq 0.1$ the four algorithms show very similar trends, and also tuning the value of $r_2$ does not yield significant enhancements.}
\begin{figure}[htbp]
	\centering
	\includegraphics[width=0.9\textwidth]{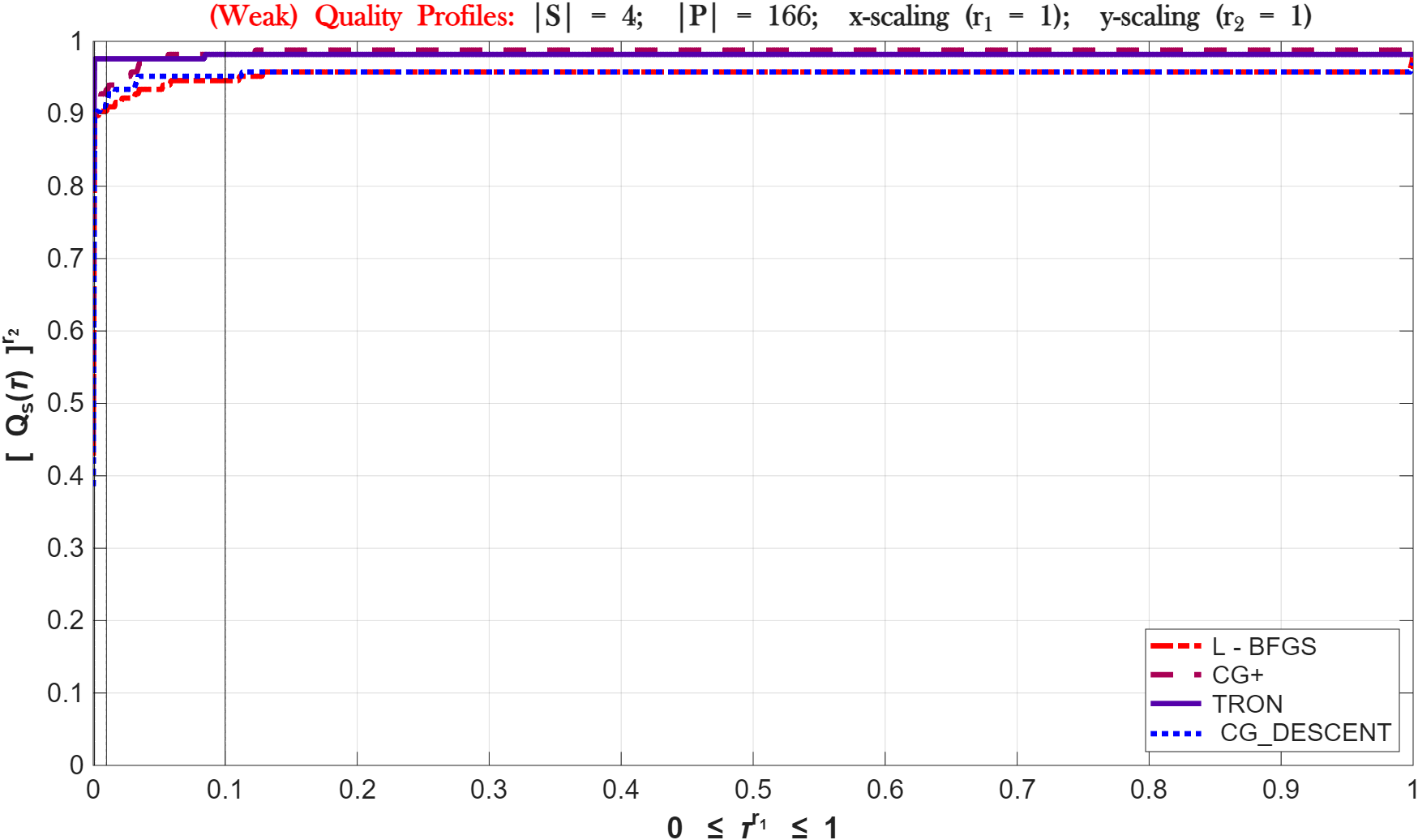}
	\caption{Quality profiles comparing {\sf L--BFGS}, {\sf CG+}, {\sf TRON}, {\sf CG$\_$DESCENT}, obtained with $r_1=1$ and $r_2=1$ in \eqref{eq:def_Qs_b}. Linear scaling is adopted in place of semilog scaling, showing a poor representation.}
	\label{fig:1new-lin}
\end{figure}

\subsubsection{Results of a second numerical experience with quality profiles for smooth optimization}
\rev{To further exploit the potentialities of the quality profiles
when ranking algorithms for smooth optimization, we now report the results of another experimentation. \revnew{\gianni{Considering} the recent renewed interest for algorithms that use negative curvature directions (see \cite{curtis.robinson:19})}, we focus on Truncated Newton methods that include them. It is well--known that the combined use of Newton--type directions and appropriate negative curvature directions\chri{,  within a Truncated Newton method,} has a twofold advantage\gianni{. From} a theoretical \chri{standpoint, in the unconstrained case, it entails} convergence towards second order critical points (stationary points where the Hessian matrix is positive semidefinite); from a computational point of view, \gianni{it shows} a better capability of the iterates to escape \gianni{from a} region of non-convexity of the objective function (see, \eg the seminal \gianni{papers}, \cite{more.sorensen:use,lucidi.rochetich.roma,glrt2} and the more recent \gianni{one \cite{curtis.robinson:19}, including therein references}). This latter property provides such algorithms \revnew{with} an enhanced capability to determine better local minimizers, so that quality profiles may enable users to clearly highlight \gianni{and gauge} this property.}
\par
\rev{We consider the Truncated Newton method proposed in \cite{fapiro:2025COAP}. It derives from the algorithm introduced in \cite{Ca-Fa-Po-Ro:2020}, where the SYMMBK \cite{chandra} algorithm is used within the inner iterations, to compute a gradient--related search direction. In \cite{fapiro:2025COAP}, the SYMMBK algorithm is also used to compute an appropriate negative curvature direction. We refer the reader to \cite{fapiro:2025COAP} for any detail; we only recall that, at each outer iteration $k$ of the method, a pair of search directions $(d_k, s_k)$ is computed: $d_k$ is a gradient-related direction, $s_k$ is an appropriate direction of negative curvature \revmax{(if any)} and the new iterate is generated according to the scheme
\begin{equation}
	x_{k+1}=x_k + \alpha_k^2 d_k + \alpha_k s_k.       \label{equ:recurrence}
\end{equation}
The steplength $\alpha_k$ is computed by a second--order Armijo-type linesearch \cite{mccormick:mathprogr}.}
\par

\rev{In order to assess the effectiveness of different negative curvature directions that can be adopted, it is crucial to have a tool for evaluating the capability of an algorithm to converge towards local minimizers with a lower objective function value. In the following, we compare the four algorithms proposed in \cite{fapiro:2025COAP}. To briefly describe the different choices concerning the adopted negative curvature directions, we denote by \gianni{$G_k \in \re^{m \times m}$} the matrix whose \gianni{$m$} columns are the $\nabla^2 f(x_k)$--conjugate directions generated by the SYMMBK algorithm \gianni{at the $k$-th outer iteration, while $D_k\in \re^{m \times m}$ represents a diagonal matrix generated by the Bunch--Kaufman decomposition within SYMMBK, and its entries are related to the eigenvalues of $\nabla^2 f(x_k)$
(see \cite{fapiro:2025COAP} for a detailed description). Observe that the closer $x_k$ to the solution point $x^{\ast}$, the larger the number of columns in $G_k$ and the number of diagonal entries in $D_k$.} The \chri{main features} of the four algorithms are the following:
\begin{itemize}
	\item {\sf TN--NONegCurv}: it does not use any piece of information related to negative curvature directions;
    \item {\sf TN--NC1}: the adopted negative curvature \gianni{direction} is computed as the sum of the contributions of all the columns of the matrix $G_k$ associated to the negative eigenvalues (if any) \gianni{of $D_k$};
    \item {\sf TN--NC2}: the adopted negative curvature \gianni{direction} corresponds to the \gianni{(say any)}
    column of the matrix $G_k$ associated to the most negative eigenvalue of $D_k$;
    \item {\sf TN--NC3}: the adopted negative curvature \gianni{direction} corresponds to the
    column of the matrix $G_k$ associated to the first \gianni{negative} eigenvalue  (if any) of $D_k$.
\end{itemize}
Therefore, each algorithm adopts a different \gianni{strategy to build the} negative curvature direction \gianni{$s_k$ in \eqref{equ:recurrence}}, and this could significantly affect the capability of the algorithms to determine better local minimizers. \revnew{The stopping criterion for all the algorithms is the standard one
$ {\| \nabla f(x_k) \|_2  \leq 10^{-5}  \max \{ 1 , \|x_k\|_2\}}$.}
To carry out an investigation on the use of negative curvature directions, we now use quality profiles to compare  the results obtained \gianni{over the set of solvers}} 
	$${\cal S}= \{\hbox{{\sf TN--NONegCurv}, {\sf TN--NC1}, {\sf TN--NC2}, {\sf TN--NC3}} \}.$$

Figure~\ref{fig:1-case-smooth} reports the weak quality profiles for the case $r_1=1$ and $r_2=1$ in \eqref{eq:def_Qs_b} (\ie no scaling is imposed on either the axes of the plot).
\begin{figure}
	\centering
	\includegraphics[width=0.8\textwidth]{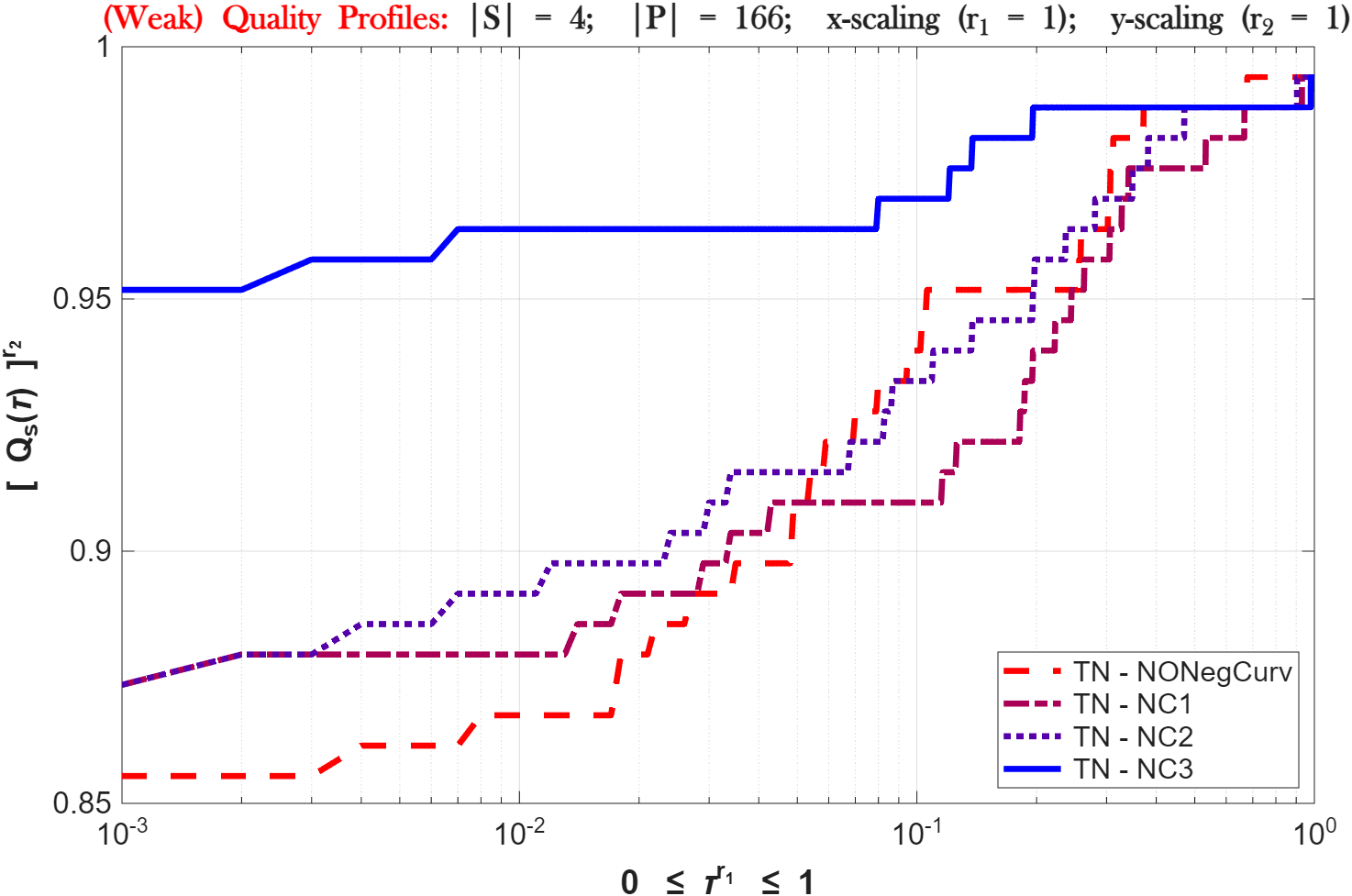}
	\medskip
	\caption{Example of weak quality profiles for the solvers {{\sf TN--NONegCurv}, {\sf TN--NC1}, {\sf TN--NC2} and {\sf TN--NC3}}, with $r_1=1$ and $r_2=1$ in \eqref{eq:def_Qs_b}.}
	\label{fig:1-case-smooth}
\end{figure}

Conversely, Figures~\ref{fig:6-cases-smooth} and \ref{fig:6-cases-smooth_b} report the \gianni{corresponding} quality profiles for different values of $r_1$ and $r_2$ in \eqref{eq:def_Qs_b}\gianni{, as by Observation \ref{observ:5}}.
\begin{figure}[htb]
		\centering
		\begin{subfigure}{0.69\textwidth}
			\includegraphics[width=\textwidth]{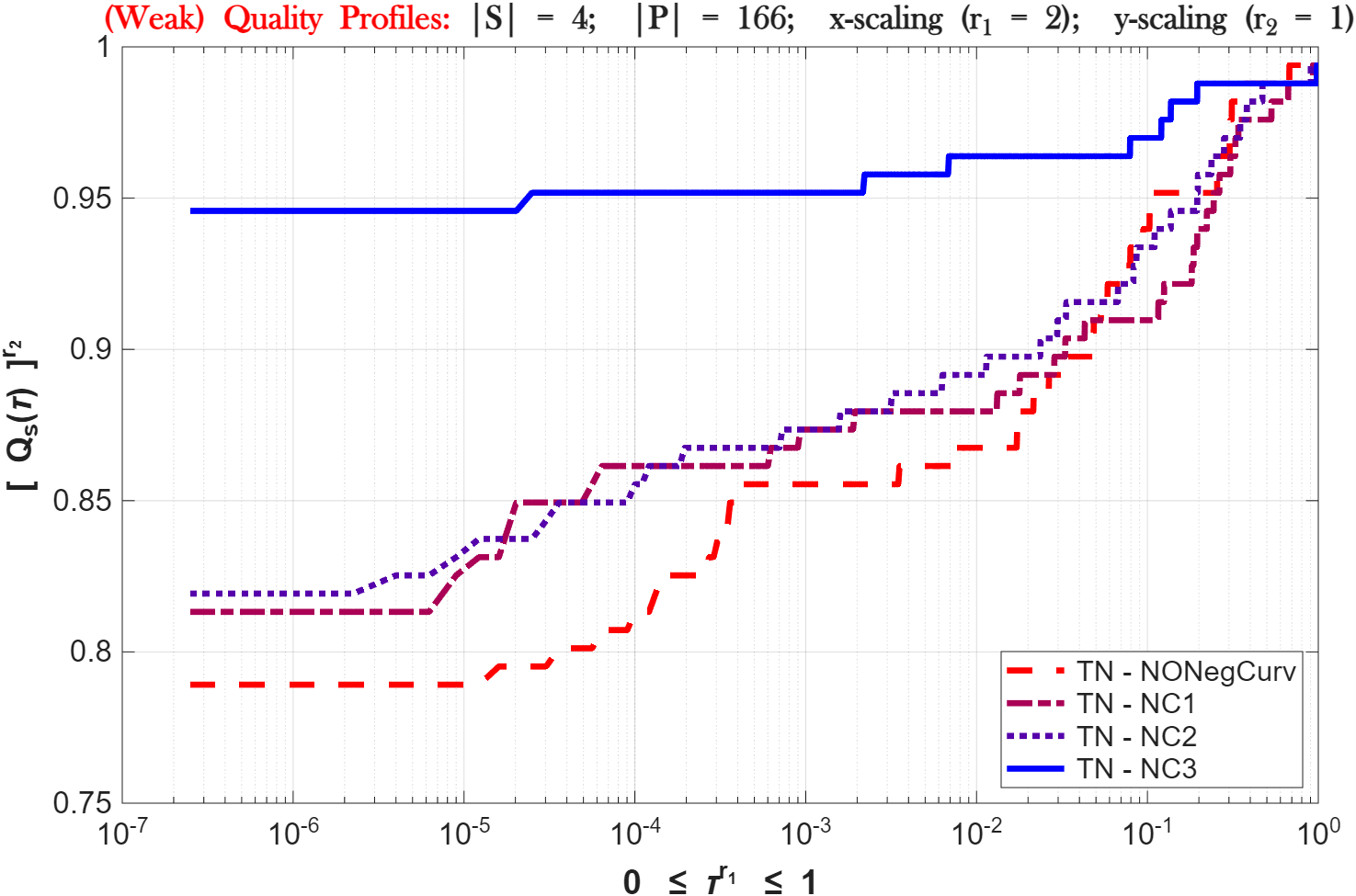}
			\label{fig:6-cases-smooth1.2}
		\end{subfigure}
		\hfill
		\begin{subfigure}{0.69\textwidth}
			\includegraphics[width=\textwidth]{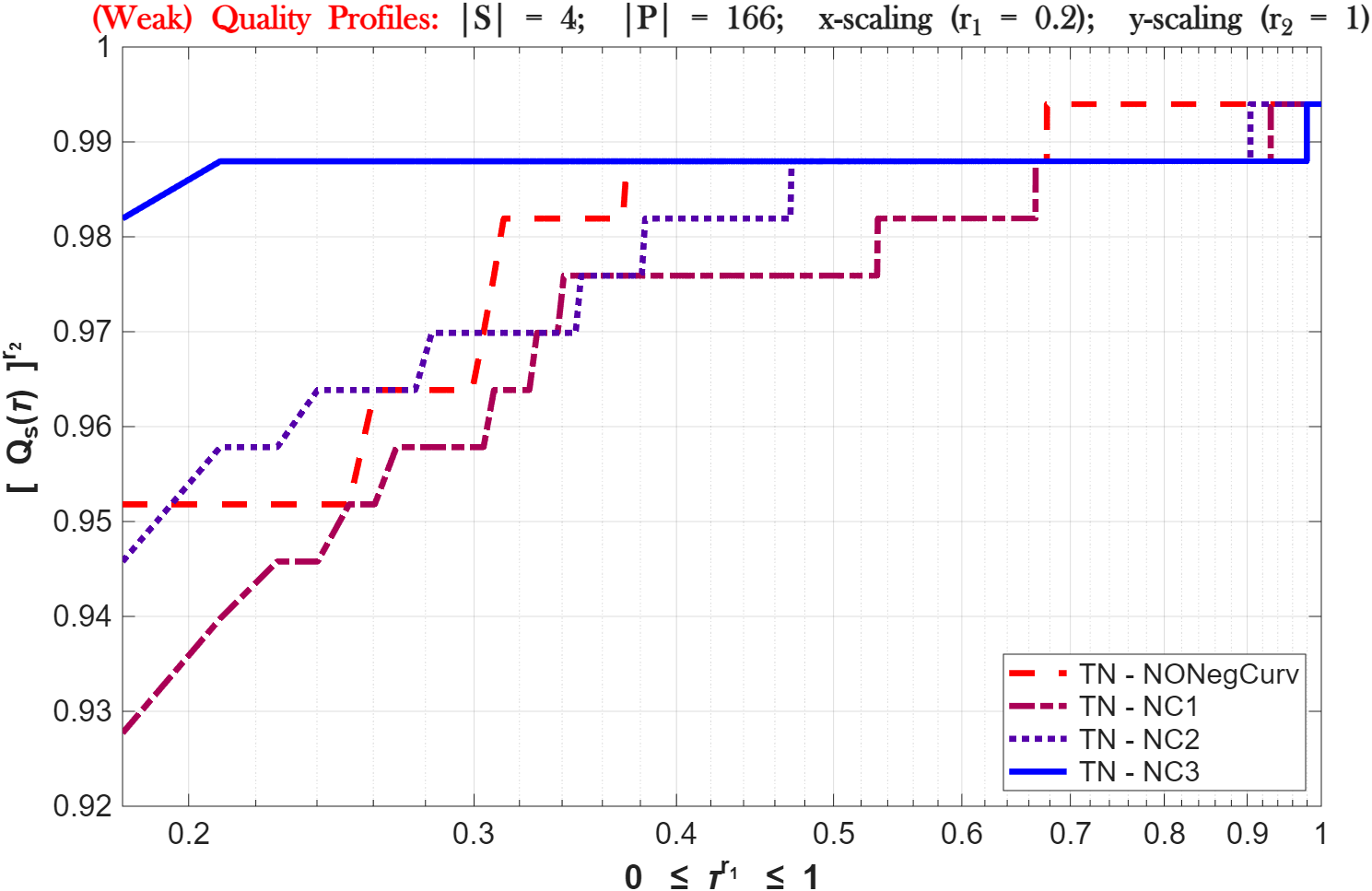}
			\label{fig:6-cases-smooth2.2}
		\end{subfigure}
		\begin{subfigure}{0.69\textwidth}
			\includegraphics[width=\textwidth]{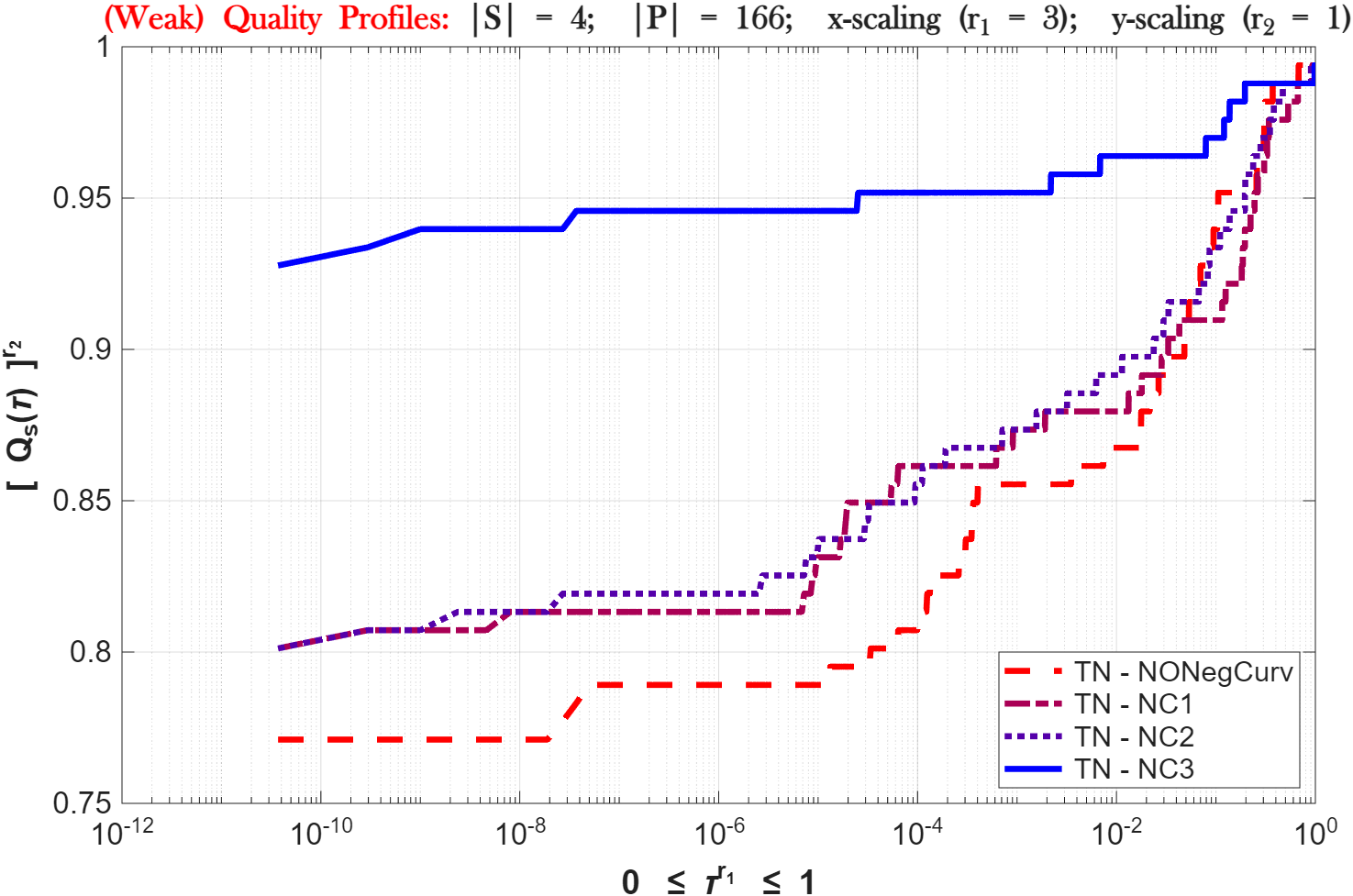}
			\label{fig:6-cases-smooth3.2}
		\end{subfigure}
		\caption{\gianni{Ex}amples of quality profiles for the solvers {{\sf TN--NONegCurv}, {\sf TN--NC1}, {\sf TN--NC2} and {\sf TN--NC3}}, selecting different values of the parameters $r_1$ and $r_2$.}
		\label{fig:6-cases-smooth}
\end{figure}
\begin{figure}[htb]
	\centering
	\begin{subfigure}{0.69\textwidth}
		\includegraphics[width=\textwidth]{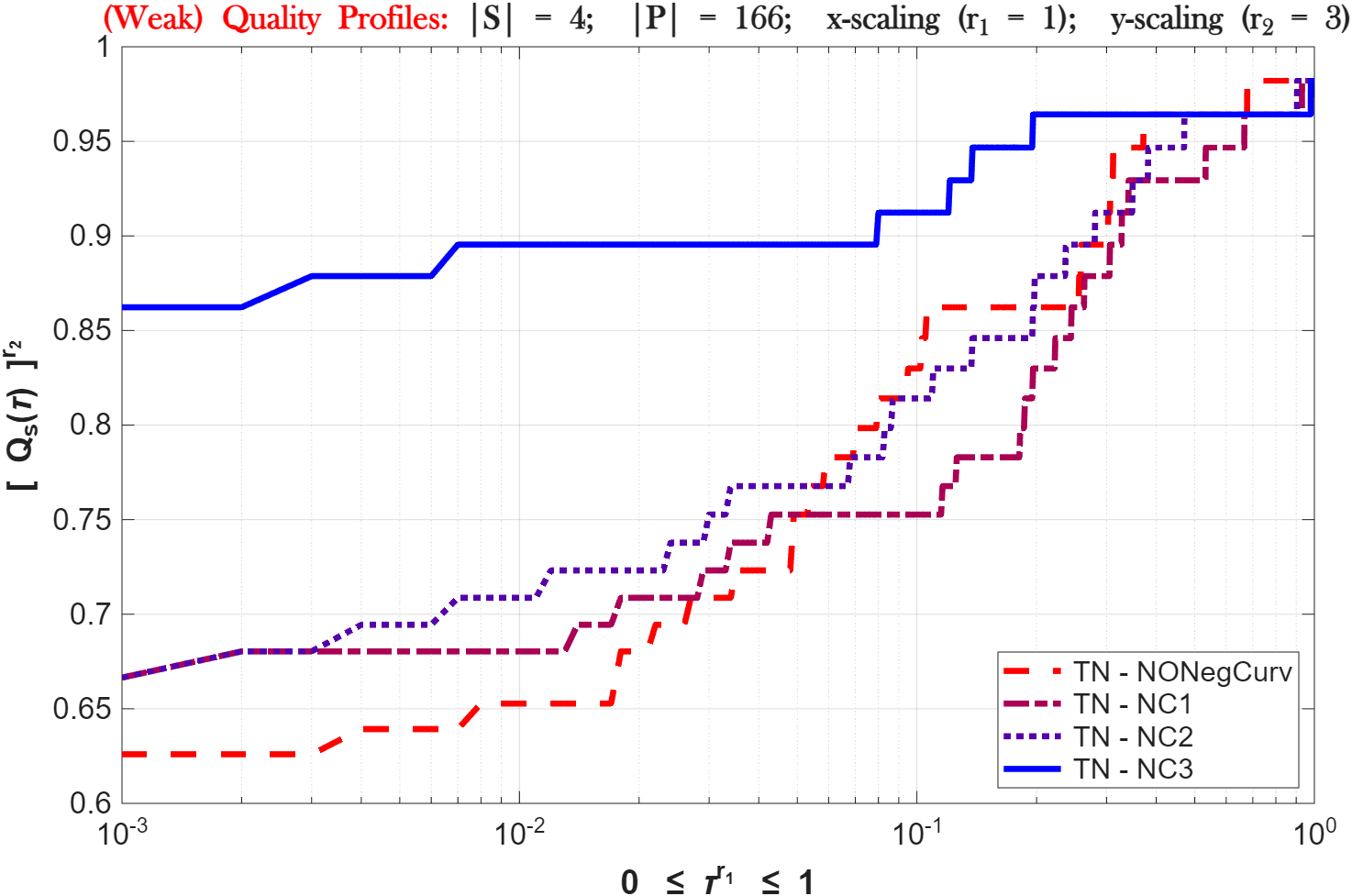}
		\label{fig:6-cases-smooth4.2}
	\end{subfigure}
	\hfill
	\begin{subfigure}{0.69\textwidth}
		\includegraphics[width=\textwidth]{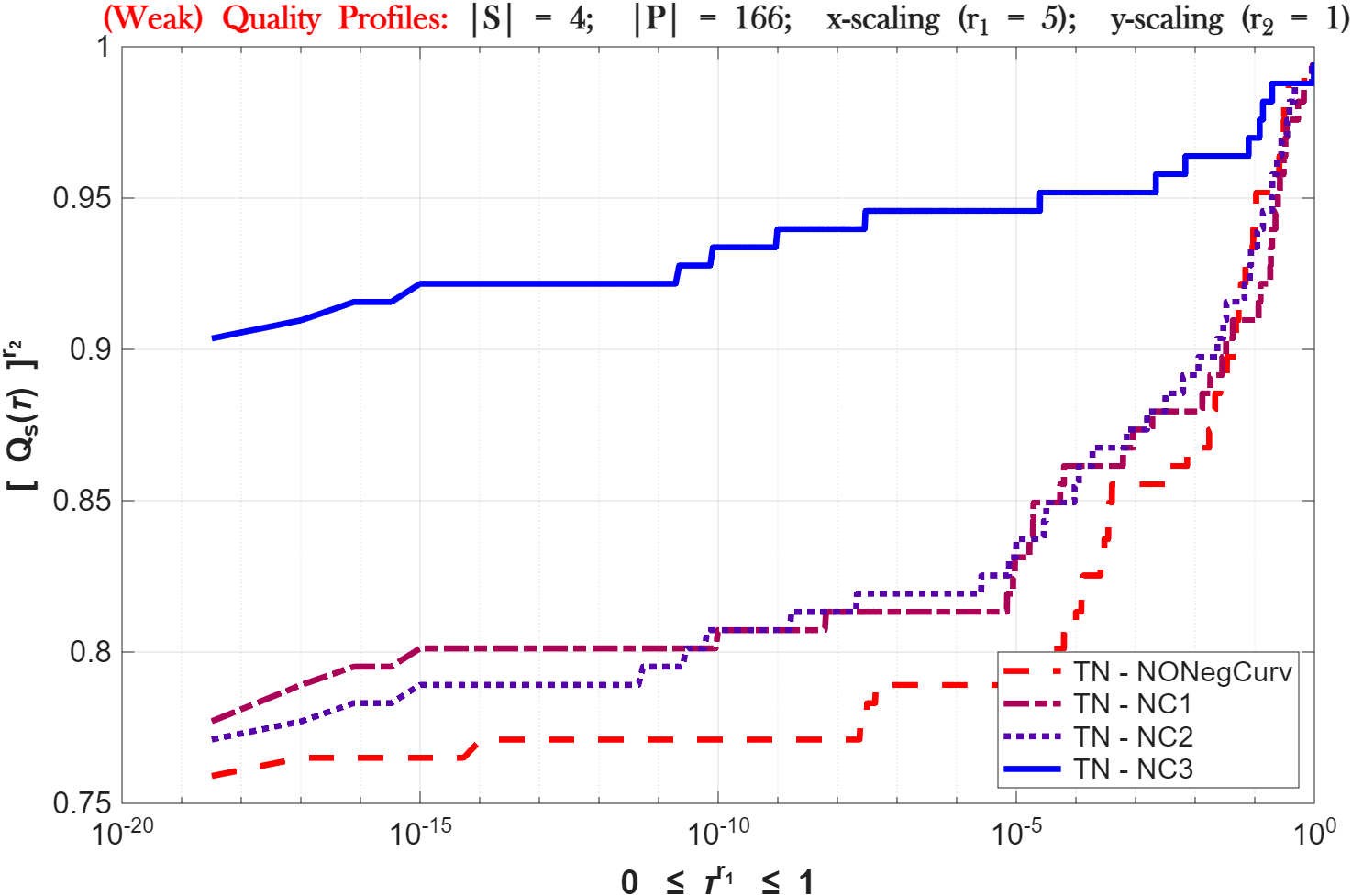}
		\label{fig:6-cases-smoothe5.2}
	\end{subfigure}
	\hfill
	\begin{subfigure}{0.69\textwidth}
		\includegraphics[width=\textwidth]{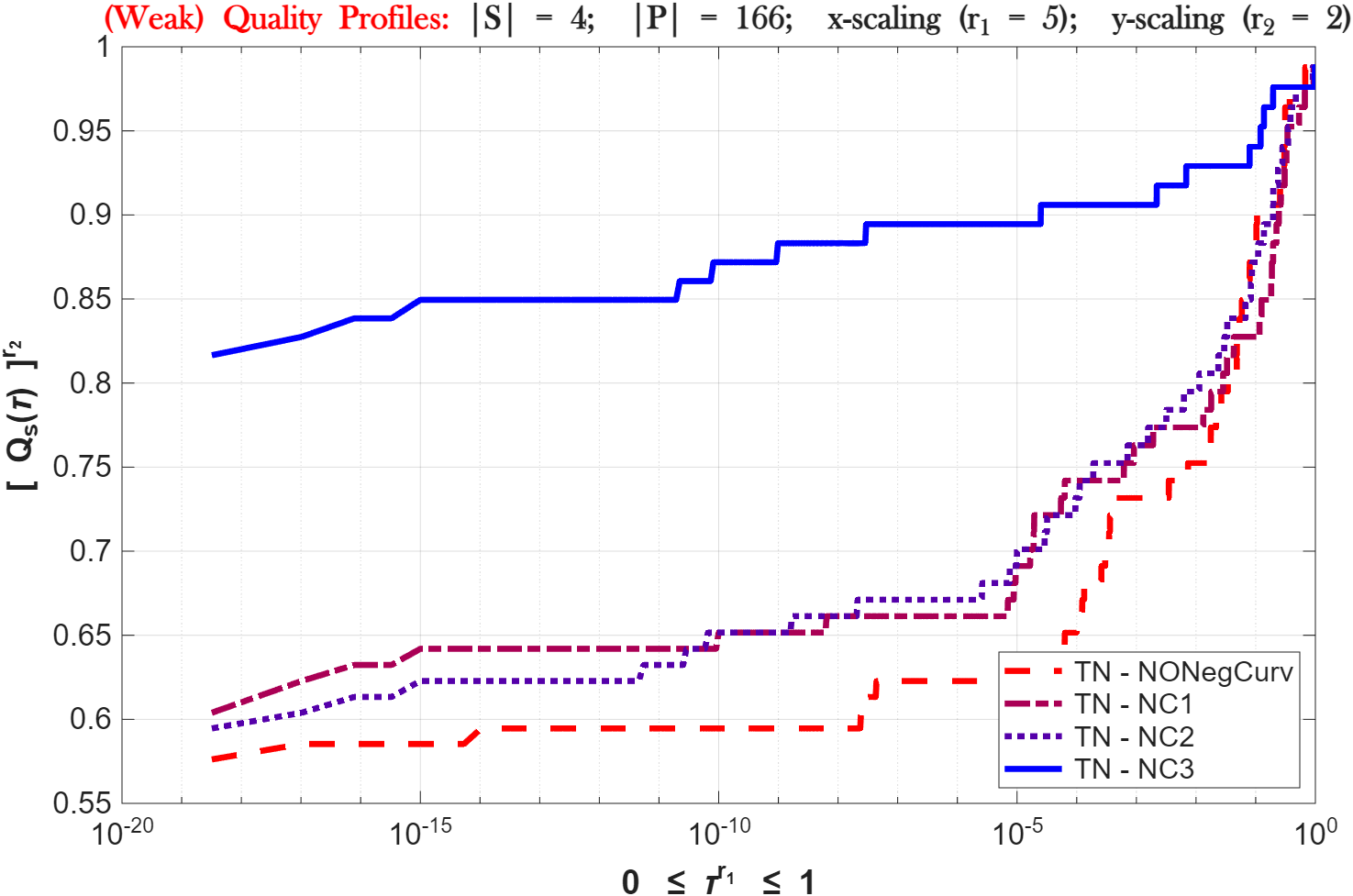}
		\label{fig:6-cases-smooth6.2}
	\end{subfigure}
	\caption{\gianni{Ex}amples of quality profiles for the solvers {{\sf TN--NONegCurv}, {\sf TN--NC1}, {\sf TN--NC2} and {\sf TN--NC3}}, selecting different values of the parameters $r_1$ and $r_2$.}
	\label{fig:6-cases-smooth_b}
\end{figure}
\par
Once more, we strongly highlight the importance of introducing the scaling parameters $r_1$ and $r_2$. Indeed, Figure~\ref{fig:1-case-smooth} shows that the \chri{default} choice $r_1=r_2=1$ may \gianni{attenuate} our capability of discriminating and ranking the quality among solvers, in the relatively small range of accuracy $[10^{-3},1]$. That is again an expected effect, associated to solvers that might have close tracks \chri{over a given region of the abscissa axis (\ie the solvers {\sf TN--NONegCurv}, {\sf TN--NC1} and {\sf TN--NC2})}. Hence, a suitable alternative choice for $r_1$ and $r_2$ may represent a winning strategy for precisely benchmarking multiple codes, on even larger  subsets of the abscissa axis. On this guideline, let us recall again that \chri{such a} precise choice for scaling parameters within performance and data profiles is far from being possible. This is due to the fact that unlike quality profiles (where we always find a unique interval for the abscissa values, say $[0,1]$), performance and data profiles may have to consider different abscissa intervals, depending on the achieved results for the compared solvers. \gianni{This drawback further implies that some thresholds need to be specified to yield performance/data profiles, identifying performance out of an {\em `a priori'} allowed range.} 
Finally, we highlight that the semilog scale may be replaced by a linear one, but {\em we truly discourage it in case a comparison where higher level precision plays a relevant role}.   
\par
\gianni{We remark that as expected,} those plots in Figures~\ref{fig:6-cases-smooth} and \ref{fig:6-cases-smooth_b}, corresponding to $r_1 > 1$,  are conveniently expanded along the abscissa axis, by simply increasing the value of $r_1$. In particular, in case {\tt MATLAB} is used\footnote{The user may download a {\tt MATLAB} code for creating quality profiles at https://github.com/fasano-g/Quality-Profiles.git}, it typically plots by default three/four orders of magnitude when using the semilog scaling. Then the plots in Figures~\ref{fig:6-cases-smooth} and \ref{fig:6-cases-smooth_b} with $r_1 > 1$ straightforwardly report an expansion of the abscissa axis nearby the origin, without sacrificing the clarity or the remaining portion of the plots. \gianni{Quality profiles confirm that} the most effective \chri{method}, among the three codes incorporating negative curvature directions, seems to be {\sf TN--NC3}. We recall that in this last algorithm the negative curvature direction is computed only selecting information associated to the first negative eigenvalue of the diagonal matrix $D_k$. The last fact confirms \gianni{the outcomes} in \cite{fasano.lucidi:2009}, and reveals that on large scale settings \gianni{a too} accurate and expensive computation of a negative curvature direction \gianni{might} be possibly dodged.
\par\medskip	
	
\begin{observation}
\label{observ:6}{\em
	As a further comment, observe that given the set of benchmark \gianni{problems ${\cal P}$}, the area below any plot of a quality profile (and above the abscissa axis) is uniquely associated with the corresponding algorithm. Hence, recalling the rationale behind performance and quality profiles, the larger that area, the better the corresponding algorithm in terms of performance. Furthermore, in quality profiles the range of values of the abscissa is always $[0,1]$, while for performance profiles it can be much different depending on the performance of any algorithm. Hence, for a given set of benchmark functions, the idea of associating a unique positive number to any track (\ie the area below the track) of a quality profile is less naturally extendable also to performance profiles and data profiles. 
	\revnewmax{
	It is worth noting an insightful connection between the geometrical properties of our quality profiles and the recent summary statistics proposed in the literature. As outlined in Observation~\ref{observ:6}, the area below any track of a quality profile is uniquely associated with a specific algorithm's capability, providing a direct, positive scalar metric for global solver ranking within the bounded domain $[0, 1] \times [0, 1]$. This metric strongly echoes the rationale behind the AOCC discussed by Lopez-Ibanez et al. \cite{lopez2024using} as a proxy for the integral of the EAF. Nevertheless, a structural advantage of the quality profiles presented here lies in the flexibility offered by the parameters $r_1$ and $r_2$ in equation \eqref{eq:def_Qs_b}. By altering these exponents, the underlying distribution of the precision samples $\tau$ is non-linearly compressed or expanded. Consequently, the resulting area under the curve can be dynamically weighted toward either low-accuracy intervals ($r_1 = 1$) or high-accuracy thresholds ($r_1 > 1$), a feature that is fundamentally unavailable in standard EAF, ECDF, or AOCC mathematical formulations.
}
}
\end{observation}

%
\par\medskip
\begin{observation}
{\em
The arrangement in \eqref{eq:def_Qs_b}, with respect to \eqref{eq:def_Qs}, is subject to further generalizations, since the power $\tau^{r_{1}}$ can be replaced by another continuous function $\phi_{1}(\tau)$, such that
\begin{itemize}[aaaaa]
	\item[(a)] $0 \leq \phi_{1}(\tau) \leq 1$;
	\item[(b)] $\phi_{1}$ is strictly increasing;
	\item[(c)] $\phi_{1}(0)= 0$ and $\phi_{1}(1)= 1$.
\end{itemize}
As an example, the choice of $\phi_{1}(\tau)$ may be driven by the functions $\{Q_s(\tau)\}$, first computed using \eqref{eq:def_Qs}, with $s \in {\cal S}$, so that these functions can be better scattered in $[0,1] \times [0,1]$ after replacing $\tau^{r_1}$ with $\phi_{1}(\tau)$ in \eqref{eq:def_Qs_b}, and defining a measure for scattering. A similar alternative arrangement can be conceived for the ordinate axis in the quality profiles, \ie replacing the power $[Q_{s}(\tau)]^{r_{2}}$  by an alternative function $\phi_{2}[Q_{s}(\tau)]$ endowed with properties analogous to (a)--(c). Observe, that in principle, the idea of rewriting \eqref{eq:def_Qs_b} as $$
\phi_{2} \left[Q_s(\tau)\right]= \frac{1}{|{\cal P}|} size \left\{ p \in {\cal P} \ : \ f_s^{(p)}(x^\ast) - f_L^{(p)}  \leq \phi_{1}(\tau) \left[ f^{(p)} ( x_0^{(p)} ) - f_L^{(p)} \right] \right\},
$$
$\tau\in[0 , 1]$,
where both $\phi_1$ and $\phi_{2}$ satisfy the above assumptions and the tracks $\{ \phi_{2} \left[Q_s(\tau)\right]\}_{s \in {\cal S}}$ are better scattered in the area $[0,1] \times [0,1]$, may correspond to find $\phi_1$ and $\phi_{2}$ such that they solve the maximization problem
\begin{equation} \label{eq: maximization.integral}
	\max_{\phi_{1},\phi_{2}} \ \  \int_{0}^{1} \sum_{s,t \in {\cal S}, \atop{s \neq t}}\Bigl\{ \phi_{2}[Q_{s}(\tau)] - \phi_{2}[Q_{t}(\tau)] \Bigr\}^{2} \ \  d\tau.
\end{equation}
This last issue (strictly related to quality profiles) deserves additional attention, since proving the existence and (possibly) the uniqueness of the solution of \eqref{eq: maximization.integral} needs further investigation, with the specific reference to the criticisms raised in \cite{gould2016note}. Indeed, the chance to select $\phi_{1}$ and $\phi_{2}$ so that a clear ranking among solvers is also identified, would be definitely appealing.
}
\end{observation}

\section{Quality Profiles for derivative--free optimization}
\label{sec:quality_profiles_derivative--free}
As for the derivative--free frameworks, where an $n$--dimensional nonsmooth function is minimized, we know that data profiles \cite{more.wild} are an effective tool to measure performances, depending on different precision levels  allowed when running the different solvers. In particular, each data profile basically represents a performance profile among solvers (where in the abscissa axis we consider the number of simplex gradient evaluations) built using partial information. The last partial information is obtained for each solver $s \in {\cal S}$, and represents the computational effort (\ie the number of function evaluations) needed by $s$ to solve different fractions of the test set, assuming that the accuracy required for the solution is given.

As an example, we can plot several data profiles, 
\revnewmax{where each plot is depicted with respect to the number of function evaluations} (one simplex gradient evaluation is equivalent to $n+1$ function evaluations), and the relative accuracy is given by setting a value $\sigma_i \in [0,1]$. In a data profile $\sigma_i$ is used to compute the performance of each solver $s$, on the set of benchmark problems ${\cal P}$, so that (borrowing the taxonomy used in Section \ref{sec:qp})
\begin{equation}
	f_s^{(p)}(x^\ast) - f_L^{(p)}  \leq \sigma_i \left[ f^{(p)} ( x_0^{(p)} ) - f_L^{(p)} \right].   \label{equ:data_profiles}
\end{equation}
This yields a sequence of $\ell$ data profile plots, each of which is associated to a different precision level
        $$ \sigma_1 < \cdots < \sigma_{\ell}. $$

As additional tools in the literature to monitor the progress of solvers for derivative--free optimization, we find more standard {\em convergence plots} and {\em trajectory plots} (see \eg \cite{audet.hare.book}). In the former tool, for each solver $s \in {\cal S}$ the best (feasible) objective function value is plotted vs. the number of function evaluations used by the solver $s$. Conversely, when only a couple of unknowns are present in the problem $p \in {\cal P}$, trajectory plots report the path of the points generated by a solver $s \in {\cal S}$ when overlapped to the plot of the contours of the objective function of $s$: this gives a visual representation of the geometry behind the progress towards the solution.

The basic idea of adopting quality profiles for derivative--free frameworks is that of {\em reversing the rationale behind data profiles}. Indeed, each quality profile is explicitly built in accordance with a relation similar to \eqref{equ:data_profiles}, and where $\sigma_i$ ranges in the entire interval $[0,1]$. Hence, we might conceive a sequence of quality profiles where, from one plot to another, {\em the budget allowed for the computation (by each solver) changes}. Hence, considering \eg the above problem of comparing the performances in terms of function evaluations, we might set the budget levels $nf_1 < \cdots < nf_{\ell}$ and build a quality profile for each value in the sequence $\{nf_i\}$. Thus, as an example, the corresponding first plot in the sequence will report a quality profile allowing for each solver a maximum number of function evaluations equal to $nf_1$.
\par
Now quoting from \cite{more.wild} ``{\em As $[\sigma_i$ in \eqref{equ:data_profiles}$]$ decreases, the accuracy of $f(x)$ as an approximation to $f_L^{(p)}$ increases; the accuracy of $x$ as an approximation to some minimizer depends on the growth of $f$ in a neighborhood of the minimizer}'', that remarks the importance for performance tools to monitor the progress of the compared codes in a neighborhood of minimizers. In particular, the authors of \cite{more.wild} focus on twice continuously differentiable functions and prove in their Theorem 2.1 and Lemma 2.2 that (see relation \cite{more.wild}--(2.8)) the next bound holds (assuming $\nabla^2 f(x)$ is positive definite in a neighborhood of the minimizers)
	$$ \|\nabla f(x)\|_{*} \leq \gamma \sigma_i^{1/2} \left(f^{(p)}(x_0^{(p)}) - f_L^{(p)} \right)^{1/2}, $$ 
being \revmax{`$\|  \cdot  \|_{*} $'} a suitable norm, $\gamma$ a given positive constant, and for any $x$ in a neighborhood of a minimum point $x^*$. Hence, the convergence test (see relation \cite{more.wild}--(2.2))
	$$  f^{(p)}(x) - f_L^{(p)}  \leq \sigma \left(f^{(p)}(x_0^{(p)}) - f_L^{(p)} \right) $$
indicates that smaller and smaller values for the precision coefficient $\sigma$ correspond to eventually approach stationarity on the $p$--th test problem.  

Broadly speaking, in this section we  aim at proposing a somewhat reversed analysis for quality profiles, showing that if the sequence $\{nf_i\}$ is increasing, then the resulting quality profiles correspond to progressively \revnewmax{approaching} 
the convergence of the compared codes to the  minima they detected. In particular, we are committed to show that when dealing with solvers for derivative--free optimization, \gianni{\em their overall computational cost per iteration is uniformly bounded} and does not depend on the choice of the initial point $x_0^{(p)}$, so that increasing $nf_i$ equivalently allows for a larger number of iterations to be performed.

In this regard we carry on our analysis considering simplex derivatives as in \revmax{\cite[Section 9.1]{CSV2009}}. Nevertheless, a similar analysis also holds considering linesearch--based and model--based derivative--free optimization methods. Following this guideline, let us consider the set of $n+1$ interpolation points $\left\{y_0^{(k)}, y_1^{(k)}, \ldots,y_n^{(k)}\right\}$ and the corresponding samples $\left\{f(y_0^{(k)}), f(y_1^{(k)}), \ldots,f(y_n^{(k)})\right\}$ of the objective function $f(y)$, being $x_k \equiv y_0^{(k)}$. Then, \gianni{as well--known}, it is possible to exploit the last two sequences to build an approximate gradient information at the iterate $x_k$, through the simplex gradient\footnote{Observe that the subscript in the nabla operator $\nabla_s$ stands for {\em simplex} and does not refer to the solver $s$.} $\nabla_s f(x_k)$ given by
	$$ \nabla_s f(x_k) = \left[y_1^{(k)} - y_0^{(k)} \cdots \ y_n^{(k)} - y_0^{(k)} \right]^{-T}
	                     \left[ 
	                     \begin{array}{c}
	                     	f(y_1^{(k)})-f(y_0^{(k)})	\\
	                     	 \vdots			\\
	                     	f(y_n^{(k)})-f(y_0^{(k)})
	                     \end{array} 
	                     \right] $$
(provided that $\left[y_1^{(k)} - y_0^{(k)} \cdots \ y_n^{(k)} - y_0^{(k)} \right]^T$ is nonsingular), and we can define the quantity
	$$ \Delta_k = \max_{1 \leq i \leq n } \left\{ \left\| y_i^{(k)} - y_0^{(k)} \right\|_2 \right\}. $$
Then, by \revmax{\cite[Chapter 6]{CSV2009}} we know that there exists a constant value $\kappa_g > 0$ such that the computation of $\nabla_s f(x_k)$ and the fulfillment of the condition
	$$ \| \nabla f (x_k) - \nabla_s f (x_k) \|_2 \leq \kappa_g \Delta_k $$
can be ensured for any value of $\Delta_k > 0$, within a finite number of so called {\em improvement steps}, \ie  within a {\em finite uniformly bounded} number $N$ of function evaluations\footnote{This achievement is obtained by applying at most $n$ {\em improvement steps} to the sample set $\left\{y_0^{(k)}, y_1^{(k)}, \ldots,y_n^{(k)}\right\}$, where any improvement step reduces to apply the algorithms in \revmax{\cite[Chapter 6]{CSV2009}} in order to possibly replace one of the points in $\left\{y_0^{(k)}, y_1^{(k)}, \ldots,y_n^{(k)}\right\}$.}. Hence, $N$ is independent of both the point $x_k$ and the sample set $\left\{y_0^{(k)}, y_1^{(k)}, \ldots,y_n^{(k)}\right\}$.

Furthermore, at the point $x_k$ the simplex gradient method updates the next iterate $x_{k+1}$ so that the standard (sufficient decrease) Armijo--like condition ($0 < \eta < 1$) is fulfilled
\begin{equation}
	 f(x_{k+1}) = f \left(x_k - \alpha_k \nabla_s f (x_k)\right) \leq f (x_k) - \eta \alpha_k \| \nabla_s f (x_k)\|^2, \qquad {\rm for \ some \ } \alpha_k > 0.		\label{equ:armijo-type} 
\end{equation}
It is possible to prove (see  \revmax{\cite[Chapter 9]{CSV2009}}, Lemma 9.1 and Lemma 9.2) that, as long as the gradient $\nabla f(x)$ is Lipschitz continuous on a suitable compact set containing the level set
	$$ {\cal L}(x_0) = \{x \in \re^n \ : \ f (x) \leq f (x_0) \}, $$
then the computation of $\alpha_k$ in \eqref{equ:armijo-type} similarly requires a {\em finite number of improvement steps}, \ie a finite number $F$ of function evaluations, at any iteration.

Hence, any iteration of the {\em simplex gradient} method will require at most the finite number
\begin{equation}
	N + F		\label{equ:max_feval_x_iteration} 
\end{equation} 
of function evaluations, in order to generate the next iterate \gianni{and fulfill the conditions
	$$ f(x_0) < f(x_1) < \cdots $$
}This definitely motivates our analysis, through the introduction of a sequence of quality profiles, when dealing with  derivative--free optimization frameworks. Indeed, creating a sequence of quality profiles, corresponding to an increasingly larger maximum allowed number of function evaluations for the compared codes, can comparatively gauge the progress of each code towards the final minimum point associated to each test problem, as allowed by the budget of function evaluations. In fact, by \eqref{equ:max_feval_x_iteration}, the larger the allowed number of function evaluations, the larger the number of iterations that each solver for derivative--free optimization performs (before optimality can be outreached). Equivalently, increasing the budgets $nf_1, \ldots, nf_{\ell}$ allows to  realize, through a comparison among the corresponding quality profiles, how effective and fast is each solver with respect to the others (see Figures \ref{fig:1-case-dfo_1000} -- \ref{fig:1-case-dfo_250}).     
  
We strongly remark \gianni{once more} that, by the respective definitions of data profiles and quality profiles for derivative--free optimization problems, since they are respectively built focusing on discrete precision levels and on the number of function evaluations, on the overall they can definitely infer different and complementary conclusions on the solvers.

\subsection{Application of the quality profiles to derivative--free optimization}
As for a numerical \revnewmax{experiment} adopting quality profiles for derivative--free optimization, we use the same selection of {\sf CUTEst} problems as in Section~2 of the recent paper \cite{shi.2023}. The latter paper is devoted to investigate the practical performance of finite-difference-based derivative--free optimization methods, and in the unconstrained case 73 problems have been selected varying the dimension, whenever possible, up to $n=300$ variables. The resulting set ${\cal P}$ of test problems consists of 145 problems included in Table~\ref{tab:problemidfo}.
\begin{table}[htbp]
	\begin{tabular}{p{1.4cm}r|p{1.4cm}r|p{1.4cm}r|p{1.4cm}r|p{1.4cm}r}
		\textit{Problem} & $n$ & \textit{Problem} & $n$ & \textit{Problem} & $n$ & \textit{Problem} & $n$ & \textit{Problem} & $n$ \\
		\hline
		AIRCRFTB           &  8 &
		ALLINITU           &  4 &
		ARWHEAD            & 100&
		BARD               &  3 &
		BDQRTIC            & 100 \\
		BIGGS3             & 6 &
		BIGGS5             & 6 &
		BIGGS6             & 6 &
		BOX2               & 3 &
		BOX3               & 3 \\
		BRKMCC             &  2&
		BROWNAL            & 10&
		BROWNAL            &100&
		BROWNAL            &200&
		BROWNDEN             &4 \\
		CLIFF               & 2&
		CRAGGLVY            & 4&
		CRAGGLVY            &10&
		CRAGGLVY            &50&
		CRAGGLVY           &100 \\
		CUBE                & 2&
		DENSCHNA            & 2&
		DENSCHNB            & 2&
		DENSCHNC            & 2&
		DENSCHND            & 3\\
		DENSCHNE             &3&
		DENSCHNF             &2&
		DIXMAANA            &15&
		DIXMAANA            &90&
		DIXMAANA           &300\\
		DIXMAANB            &15&
		DIXMAANB            &90&
		DIXMAANB           &300&
		DIXMAANC            &15&
		DIXMAANC            &90 \\
		DIXMAANC           &300&
		DIXMAAND            &15&
		DIXMAAND            &90&
		DIXMAAND           &300&
		DIXMAANE            &15\\
		DIXMAANE            &90&
		DIXMAANE           &300&
		DIXMAANF            &15&
		DIXMAANF            &90&
		DIXMAANF           &300\\
		DIXMAANG            &15&
		DIXMAANG            &90&
		DIXMAANG           &300&
		DIXMAANH            &15&
		DIXMAANH            &90\\
		DIXMAANH           &300&
		DIXMAANI            &15&
		DIXMAANI            &90&
		DIXMAANI           &300&
		DIXMAANJ            &15 \\
		DIXMAANJ            &90&
		DIXMAANJ           &300&
		DIXMAANK            &15&
		DIXMAANK            &90&
		DIXMAANK           &300\\
		DIXMAANL            &15&
		DIXMAANL            &90&
		DIXMAANL           &300&
		DQRTIC              &10&
		DQRTIC              &50 \\
		DQRTIC             &100&
		EDENSCH             &36&
		EIGENALS            & 6&
		EIGENALS           &110&
		EIGENBLS             &6\\
		EIGENBLS           &110&
		EIGENCLS            &30&
		ENGVAL1              &2&
		ENGVAL1             &50&
		ENGVAL1            &100\\
		EXPFIT               &2&
		FLETCBV3            &10&
		FLETCBV3           &100&
		FLETCHBV            &10&
		FLETCHBV           &100\\
		FREUROTH             &2&
		FREUROTH            &10&
		FREUROTH            &50&
		FREUROTH           &100&
		GENROSE              &5\\
		GENROSE             &10&
		GENROSE            &100&
		GULF                 &3&
		HAIRY                &2&
		HELIX                &3\\
		JENSMP               &2&
		KOWOSB               &4&
		MEXHAT               &2&
		MOREBV              &10&
		MOREBV              &50\\
		MOREBV             &100&
		NCB20B              &21&
		NCB20B              &22&
		NCB20B              &50&
		NCB20B             &100\\
		NCB20B             &180&
		NONDIA              &10&
		NONDIA              &20&
		NONDIA              &30&
		NONDIA              &50\\
		NONDIA              &90&
		NONDIA             &100&
		NONDQUAR           &100&
		OSBORNEA            &5&
		OSBORNEB            &11\\
		PENALTY1             &4&
		PENALTY1            &10&
		PENALTY1            &50&
		PENALTY1           &100&
		PFIT1LS              &3\\
		PFIT2LS              &3&
		PFIT3LS              &3&
		PFIT4LS              &3&
		QUARTC              &25&
		QUARTC             &100\\
		SINEVAL              &2&
		SINQUAD              &5&
		SINQUAD             &50&
		SINQUAD            &100&
		SISSER               &2\\
		SPARSQUR            &10&
		SPARSQUR            &50&
		SPARSQUR           &100&
		TOINTGSS            &10&
		TOINTGSS            &50\\
		TOINTGSS           &100&
		TQUARTIC             &5&
		TQUARTIC            &10&
		TQUARTIC            &50&
		TQUARTIC           &100\\
		TRIDIA              &10&
		TRIDIA              &20&
		TRIDIA              &30&
		TRIDIA              &50&
		TRIDIA             &100\\
		WATSON              &12&
		WATSON              &31&
		WOODS                &4&
		WOODS              &100&
		ZANGWIL2             &2\\
		\hline
	\end{tabular}
	\caption{Selection of {\sf CUTEst} problems adopted for experimenting the quality profiles in derivative--free optimization.}\label{tab:problemidfo}
\end{table}

\revnewmax{We choose some well established and publicly available codes from the literature for benchmarking different derivative--free solvers by using the quality profiles.}
Our selection criteria basically follow the ones already listed at the beginning of Section~\ref{sec:results-first}.
\par
The list of the selected codes is given below, and includes: the code name, 
the acronym,
the reference paper[s] and a brief description.
\revnewmax{All codes were executed using their respective default parameter settings.}
Moreover, for all codes we set the budget levels $nf_1 < nf_2 < nf_3 < nf_4$ 
with $f_1=250$, $f_2=500$, $f_3=750$ and $f_4=1000$, and build a quality profile for each value in the sequence $\{nf_i\}$.
\begin{itemize}
	\item {\sf NEWUOA} ({\sf NEWUOA}) \cite{powell.06,powell.08}: it is one of the model--based algorithms proposed by M.J.D. Powell. At each iteration, the algorithm establishes a model function by quadratic interpolation, and then minimizes it within a trust region.
	\\
	\item {\sf Coordinate search} ({\sf COORD}) \cite{polak.71}:  it is the
	simplest directional direct--search method based on a very simple search along the coordinate directions.
	\\
	\item {\sf Compass search} ({\sf COMPASS)}\cite{dolan.99}: it is a revisited version of the previous method, using coordinate directions.
	\\
	\item {\sf Hooke \& Jeeves} ({\sf H\&J}) \cite{hooke.61}: it is the original pattern--search method due to Hooke and Jeeves.
	\\
	\item {\sf Simplex search} ({\sf SH\&H}) \cite{spendley.62}: this code implements a sequence of experimental designs each in the form of a regular or irregular simplex.
    \\	
	\item {\sf Nelder \& Mead} ({\sf NM}) \cite{nelder.65}: it is based on the classical Nelder and Mead method revisited according to Lagarias et al.  \cite{lagarias.98}.
	\\
	\item {\sf Sequential Multidirectional Search} ({\sf  SMD}) \cite{torczon.98}: it is a sequential implementation of Multidirectional search algorithm originally proposed by Torczon for parallel machines.
\end{itemize}
{\sf COORD}, {\sf COMPASS} and {\sf H\&J } belong to the class of \textit{pattern--search} algorithms; {\sf SH\&H}, {\sf NM} and {\sf SMD} are classified as \textit{simplex search} methods ({\sf SMD} can be also seen as pattern--search method). Finally, {\sf NEWUOA} is a model--based method.
\par
With the exception of {\sf NEWUOA}, all the other codes are taken from the {\sf Direct Search} suite of C++ derivative--free codes for unconstrained 
optimization written by Dolan, Gurson, Shepherd, Siefert, Torczon 
and Yates publicly available\footnote{\tt https://www.cs.wm.edu/$\tilde{\ } $va/software/DirectSearch/direct\_code/}.
\par
The complete set ${\cal S}$ of the considered solvers is given by
$$
{\cal S}=\{\text{{\sf NEWUOA}}, \ \text{{\sf COORD}},  \  \text{{\sf COMPASS}},  \ \text{{\sf H\&J}},  \ \text{{\sf SH\&H}},  \ \text{{\sf NM}}, \text{{\sf SMD}} \},$$
and to respect the taxonomy by Dolan, Gurson, Shepherd, Siefert, Torczon 
and Yates, in the labels of the Figures \ref{fig:1-case-dfo_1000} -- \ref{fig:1-case-dfo_250} they are reported as (numbering is taken from the {\sf Direct Search} suite) 
$$
{\cal S}=\{\text{{\sf NEWUOA}},  \ \text{{\sf ALG1}},  \  \text{{\sf ALG2}},  \ \text{{\sf ALG4}},  \ \text{{\sf ALG6}},  \ \text{{\sf ALG7}},  \ \text{{\sf ALG8}} \}.$$

Each of the Figures \ref{fig:1-case-dfo_1000} -- \ref{fig:1-case-dfo_250} confirms the importance of properly selecting the parameter $r_1$ in \eqref{eq:def_Qs_b}, so that when zooming close to $\tau^{r_1}=0$ (upper picture) or $\tau^{r_1}=1$ (lower picture) may reveal further details to precisely benchmark different solvers, with respect to the {\em neutral} choice $r_1=1$ (central picture).

People in the community of derivative--free optimization know well that the solver \text{\sf NEWUOA} may typically yield more accurate solutions, when compared to pattern--search--based codes.
\revnewmax{As shown in  Figures \ref{fig:1-case-dfo_1000} -- \ref{fig:1-case-dfo_250}, this holds true even for a relatively modest budget of $250 \cdot n$ function evaluations.}
Moreover, regardless of the allowed budget, the model--based code \text{\sf NEWUOA} yields improvements at a fairly constant rate (see the upper picture in each of the Figures \ref{fig:1-case-dfo_1000} -- \ref{fig:1-case-dfo_250}). Conversely, pattern--search--based algorithms show a faster initial progress (say when $10^{-5} \leq \tau^{r_1} \leq 1$), and eventually struggle to improve their results when higher precision is sought.

\revnewmax{The role of the function evaluation budget is best observed by comparing the central and or upper panels in Figures~\ref{fig:1-case-dfo_1000} -- \ref{fig:1-case-dfo_250}.}
In fact, each of the tracks in the central (upper) picture of Figure \ref{fig:1-case-dfo_1000}, for each of the values in the abscissa axis (\ie the precision), tends to give a higher value with respect to the corresponding picture in Figure \ref{fig:1-case-dfo_750}. A similar conclusion holds when comparing Figure \ref{fig:1-case-dfo_750} with Figure \ref{fig:1-case-dfo_500}, or Figure \ref{fig:1-case-dfo_500} with Figure \ref{fig:1-case-dfo_250}, revealing once more the complementarity between the outcomes of data profiles and quality profiles. In particular, in each data profile the precision is first set, and then on the abscissa axis the plot is reported with respect to the computational budget (\ie the equivalent number of function/simplex evaluations). On the contrary, in quality profiles for derivative--free optimization we first set the budget for solvers, and then we benchmark codes with respect to the precision of the outcomes. 

Finally, observe that Figures \ref{fig:1-case-dfo_1000} -- \ref{fig:1-case-dfo_250} also confirm an additional couple of conclusions: 1) the semilog scale is definitely advisable also when dealing with derivative--free optimization; 2) it is clear from the lower pictures that values of $r_1$ smaller than 1 allow for zooming close to the precision level $\tau^{r_1}=1$ (this appears to be one of those few cases where both the semilog and the linear scale for the abscissa axis can be fruitfully exploited). 

\begin{figure}[htbp]
	\centering
	\begin{subfigure}{0.69\textwidth}
		\centering
		\includegraphics[width=\textwidth]{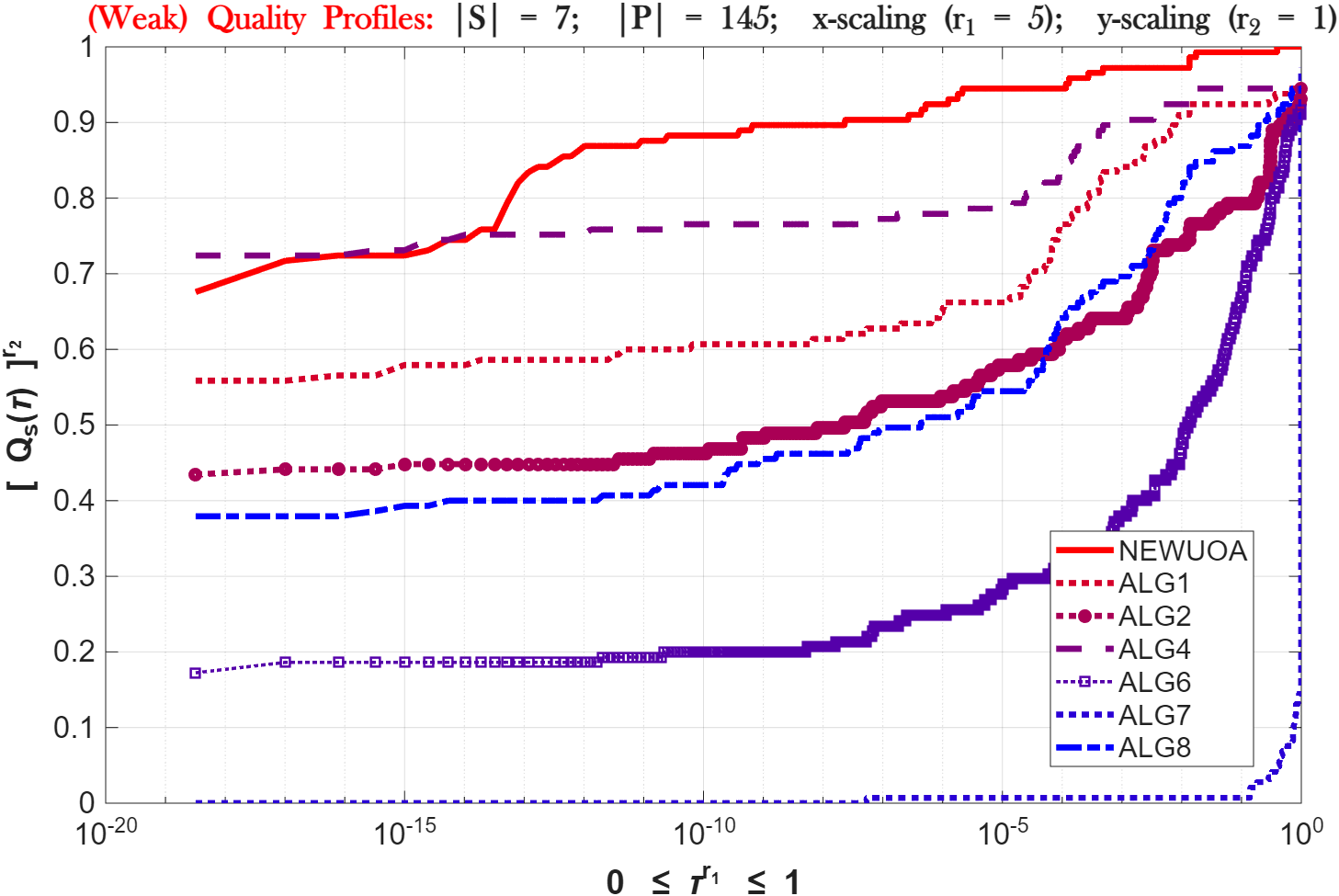} 
		\label{fig:1-case-dfo_5-1_1000}
	\end{subfigure}
	\hfill
	\begin{subfigure}{0.69\textwidth}
		\centering
		\includegraphics[width=\textwidth]{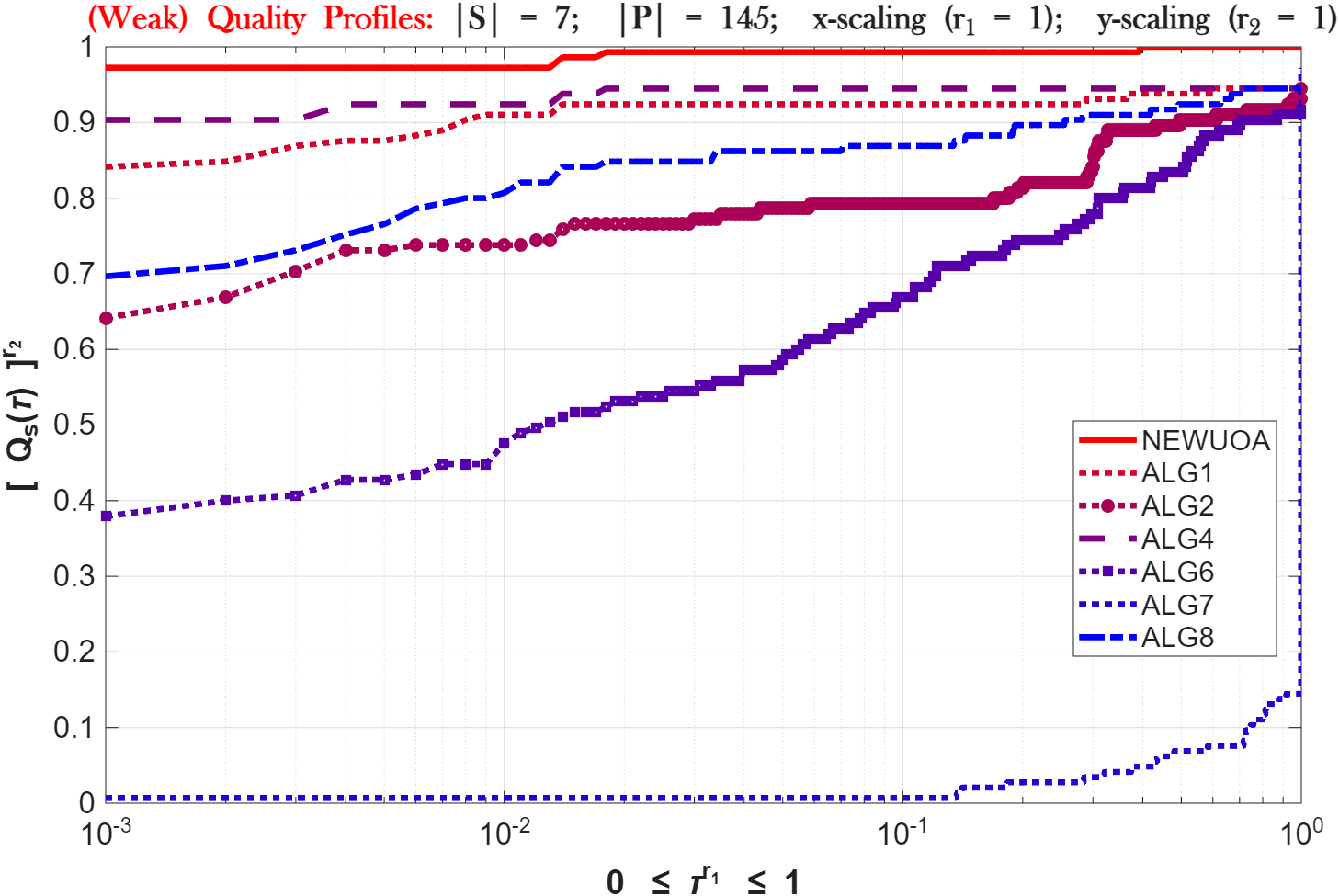} 
		\label{fig:1-case-dfo_1-1_1000}
	\end{subfigure}
	\hfill
	\begin{subfigure}{0.69\textwidth}
		\centering
		\includegraphics[width=\textwidth]{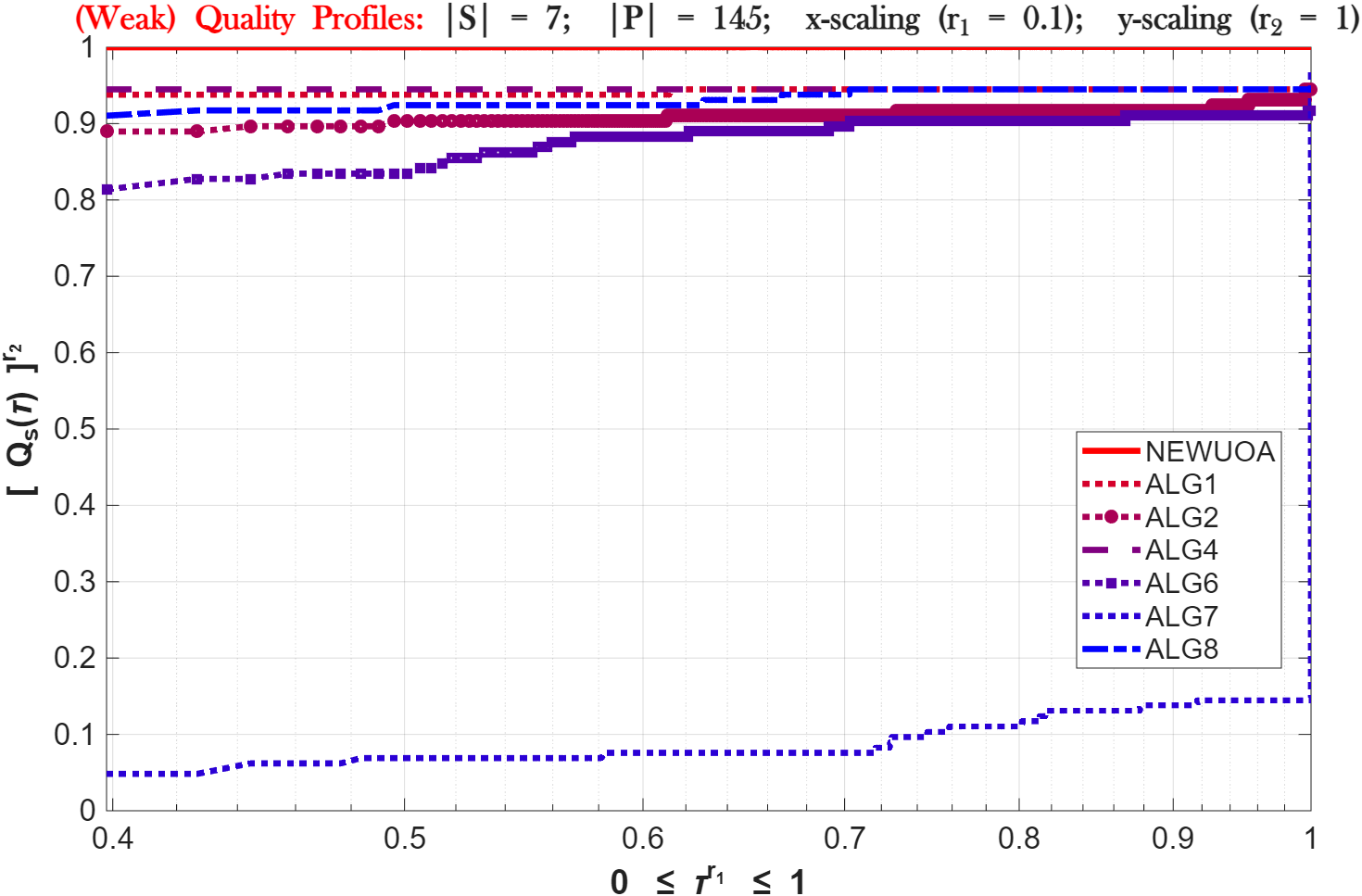} 
		\label{fig:1-case-dfo_01-1_1000}
	\end{subfigure}
	\caption{Example of weak quality profiles for the derivative--free solvers in ${\cal S}$. The allotted maximum budget for solvers is set to $1000 \cdot n$ {\em function evaluations}.}
	\label{fig:1-case-dfo_1000}
\end{figure}

\begin{figure}[htbp]
	\centering
	\begin{subfigure}[b]{0.7\textwidth}
		\centering
		\includegraphics[width=\textwidth]{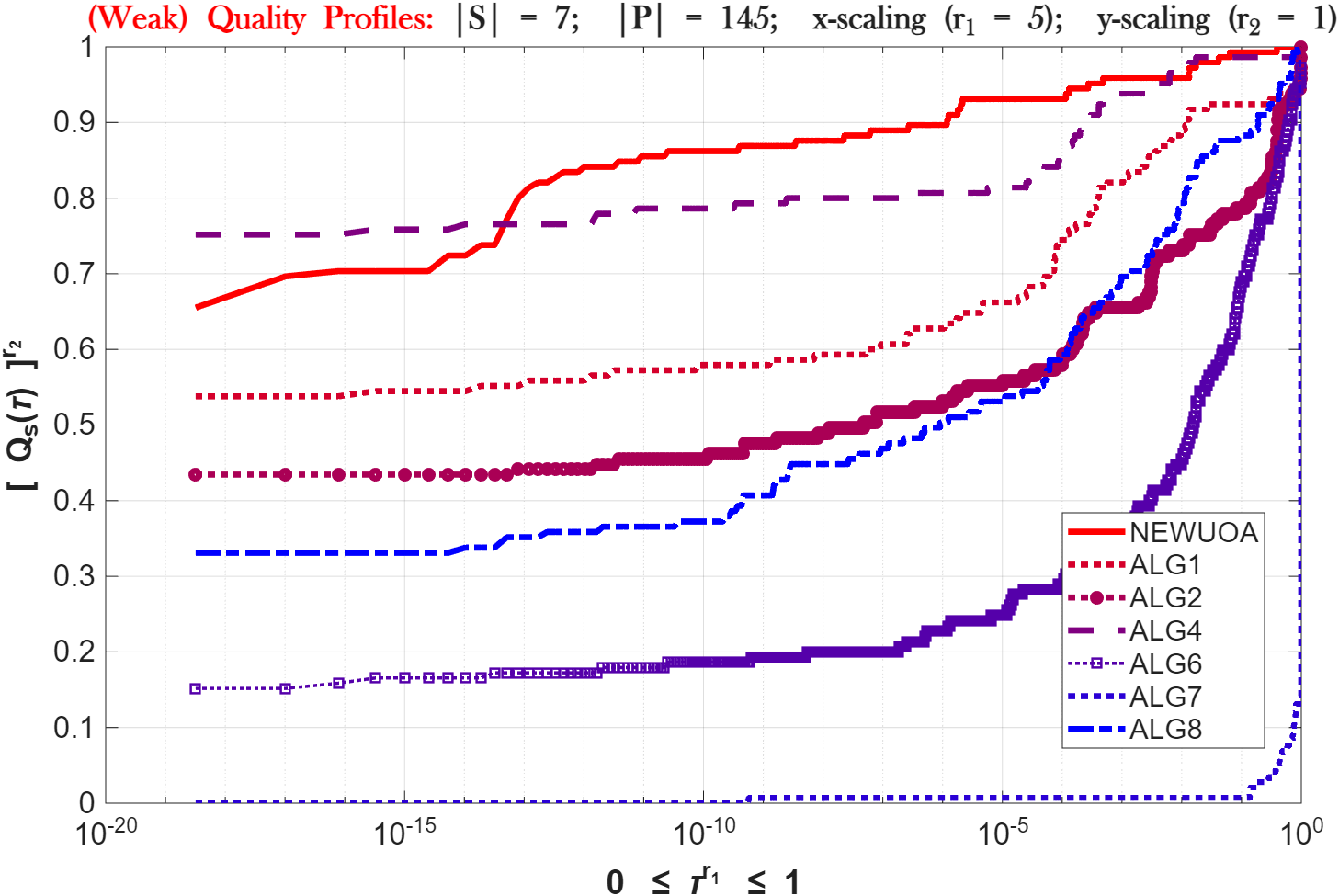} 
		\label{fig:1-case-dfo_5-1_750}
	\end{subfigure}
	\hfill
	\begin{subfigure}[b]{0.7\textwidth}
		\centering
		\includegraphics[width=\textwidth]{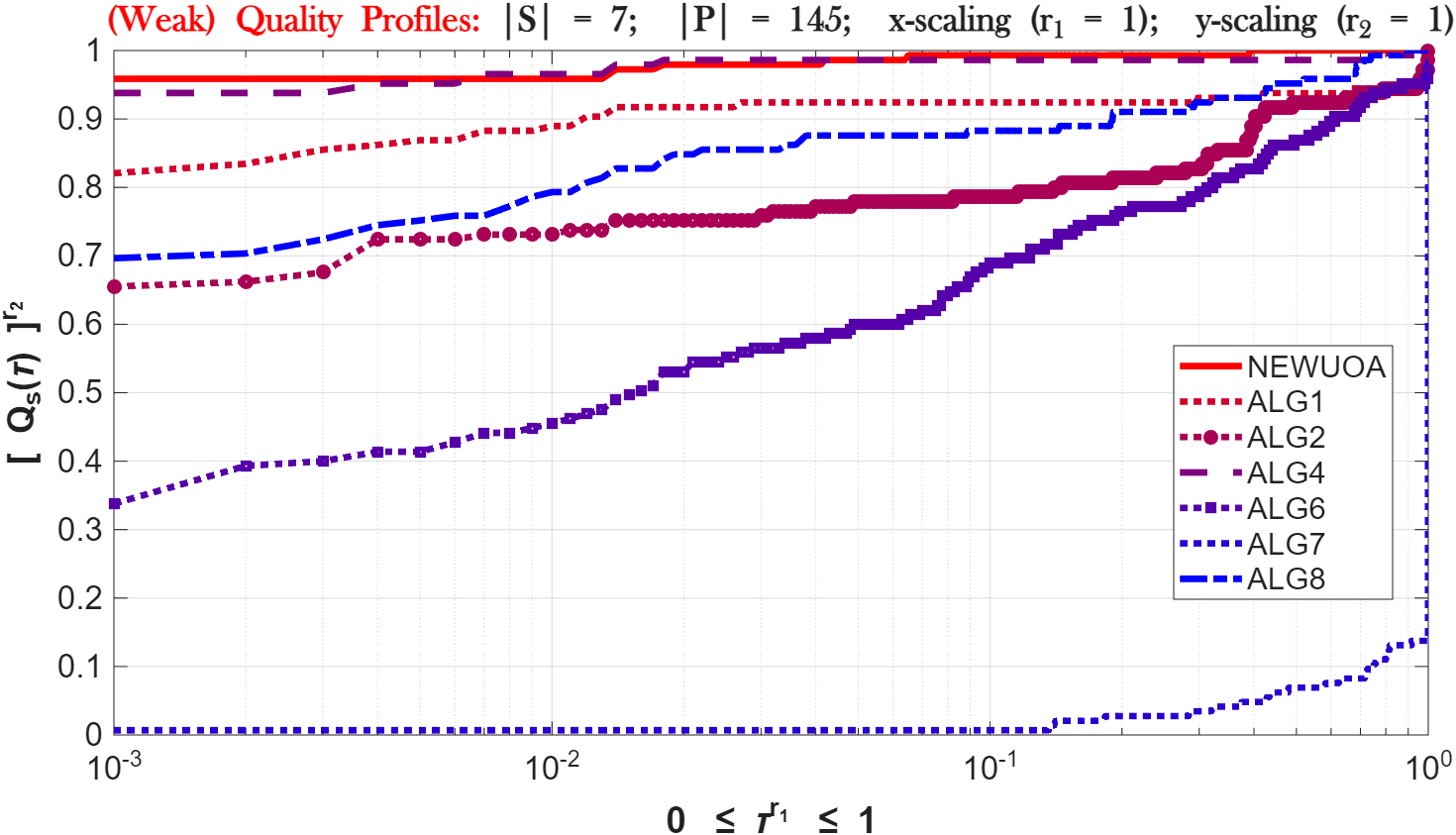} 
		\label{fig:1-case-dfo_1-1_750}
	\end{subfigure}
	\hfill
	\begin{subfigure}[b]{0.7\textwidth}
		\centering
		\includegraphics[width=\textwidth]{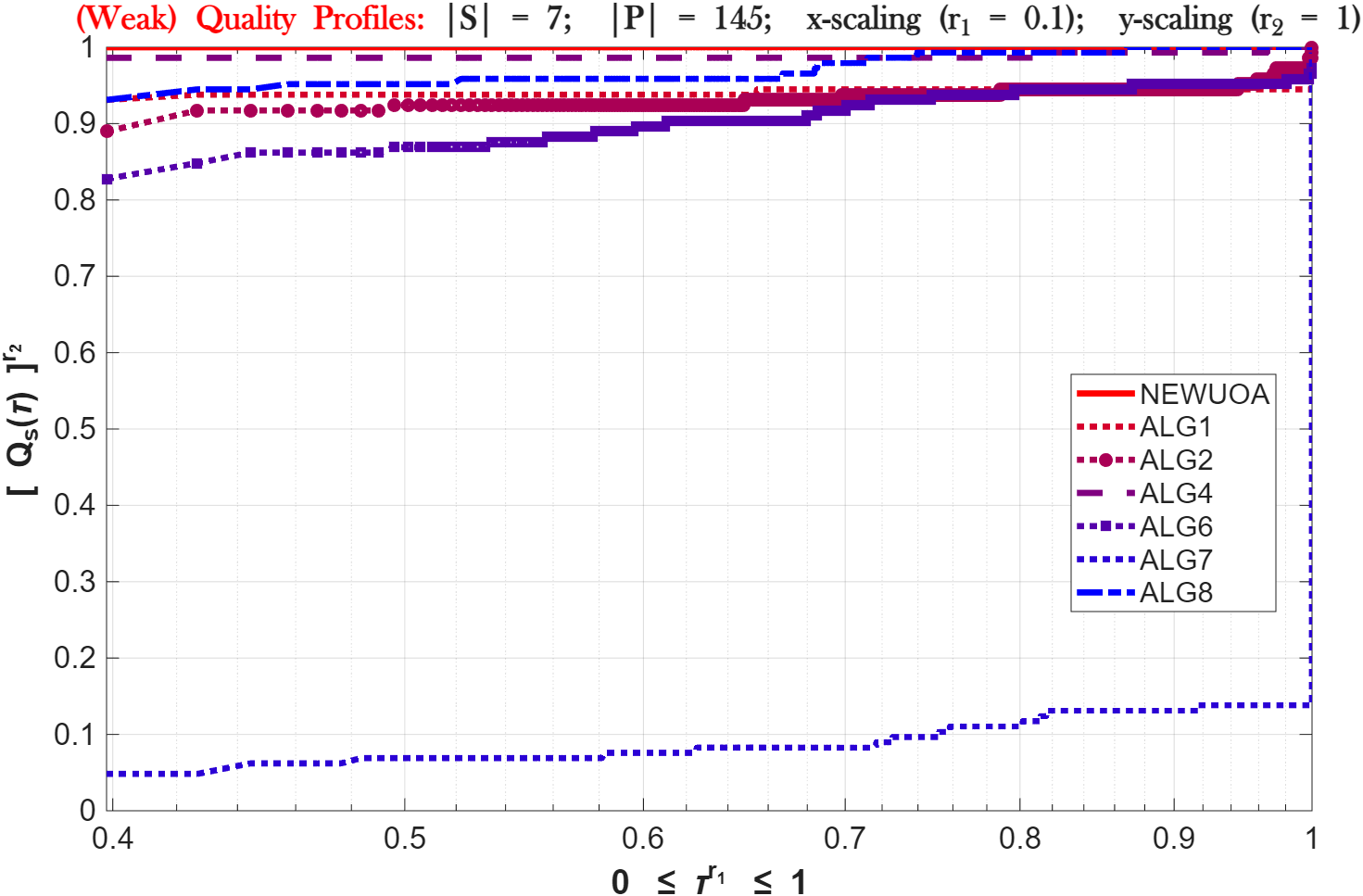} 
		\label{fig:1-case-dfo_01-1_750}
	\end{subfigure}
	\caption{Example of weak quality profiles for the derivative--free solvers in ${\cal S}$. The allotted maximum budget for solvers is set to $750 \cdot n$ {\em function evaluations}.}
	\label{fig:1-case-dfo_750}
\end{figure}

\begin{figure}[htbp]
	\centering
	\begin{subfigure}[b]{0.68\textwidth}
		\centering
		\includegraphics[width=\textwidth]{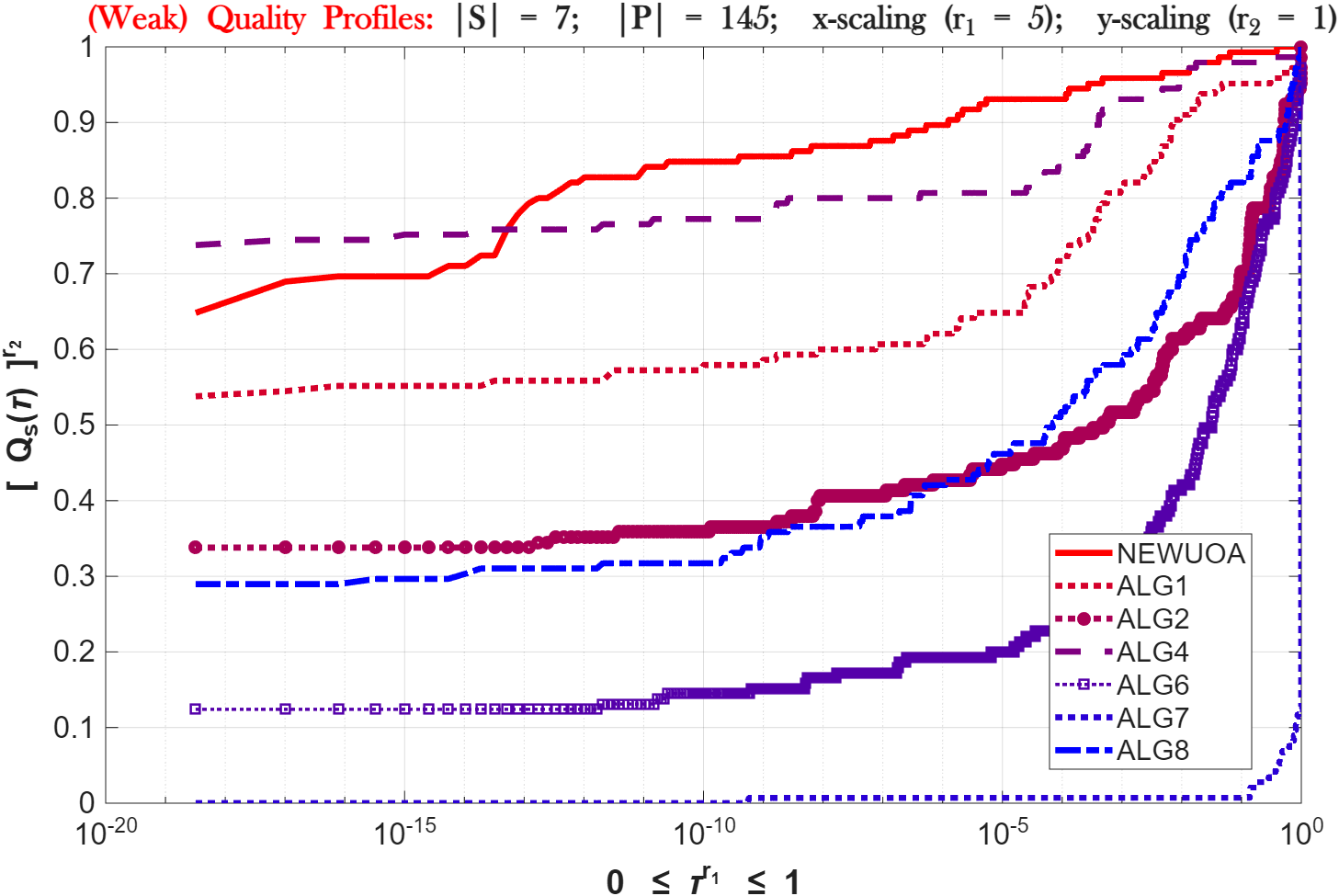} 
		\label{fig:1-case-dfo_5-1_500}
	\end{subfigure}
	\hfill
	\begin{subfigure}[b]{0.68\textwidth}
		\centering
		\includegraphics[width=\textwidth]{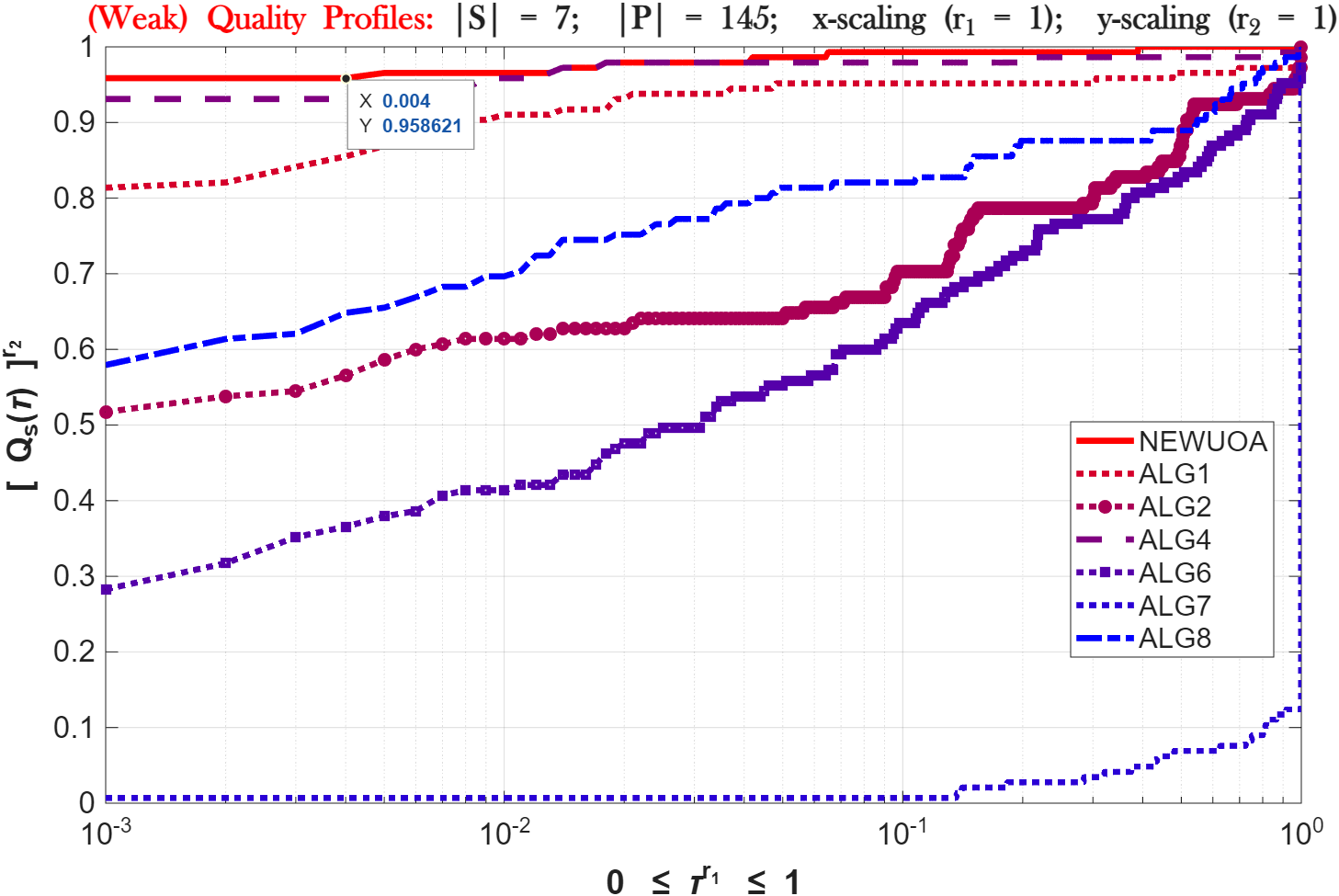} 
		\label{fig:1-case-dfo_1-1_500}
	\end{subfigure}
	\hfill
	\begin{subfigure}[b]{0.68\textwidth}
		\centering
		\includegraphics[width=\textwidth]{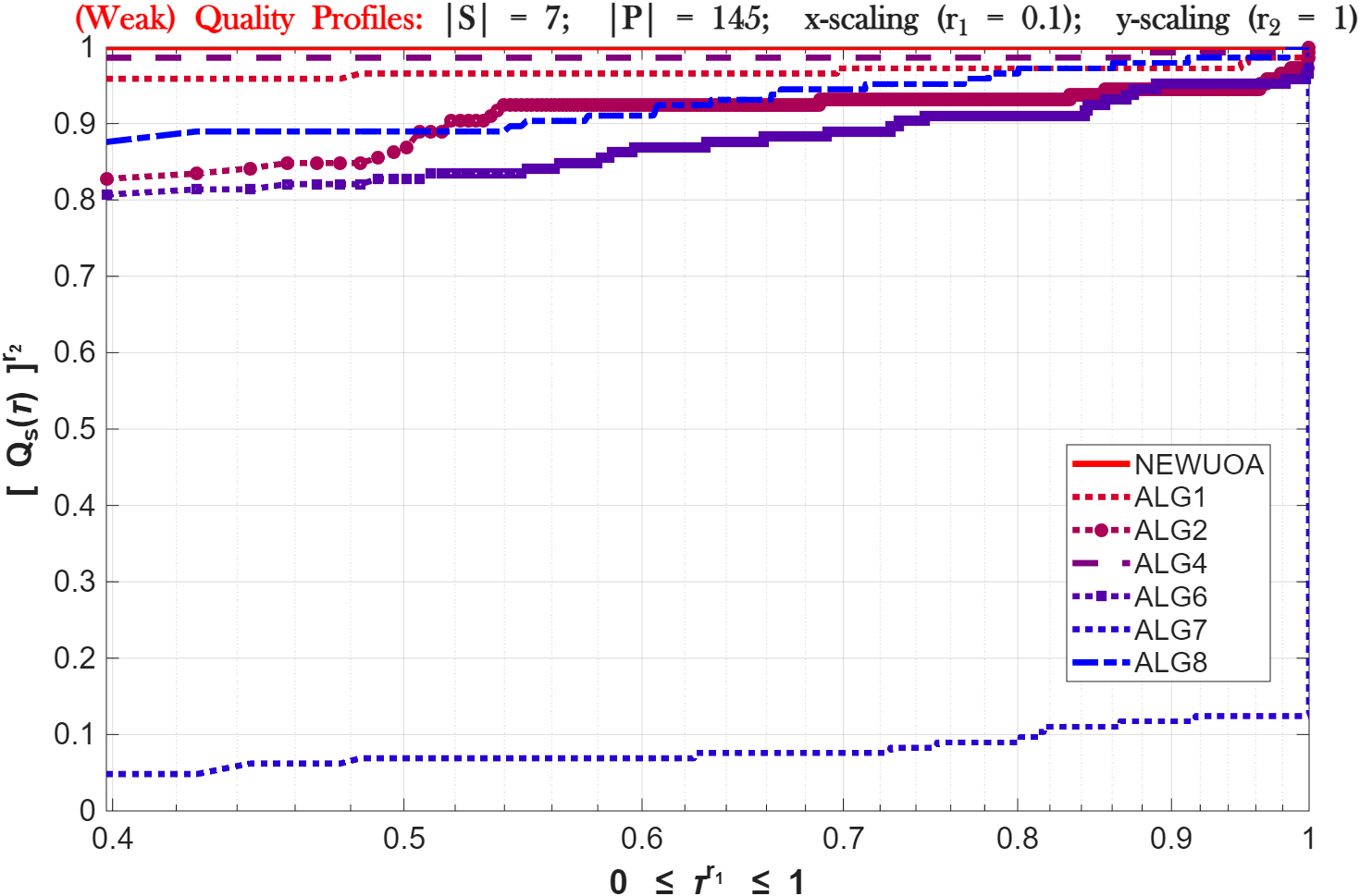} 
		\label{fig:1-case-dfo_01-1_500}
	\end{subfigure}
	\caption{Example of weak quality profiles for the derivative--free solvers in ${\cal S}$. The allotted maximum budget for solvers is set to $500 \cdot n$ {\em function evaluations}.}
	\label{fig:1-case-dfo_500}
\end{figure}

\begin{figure}[htbp]
	\centering
	\begin{subfigure}[b]{0.7\textwidth}
		\centering
		\includegraphics[width=\textwidth]{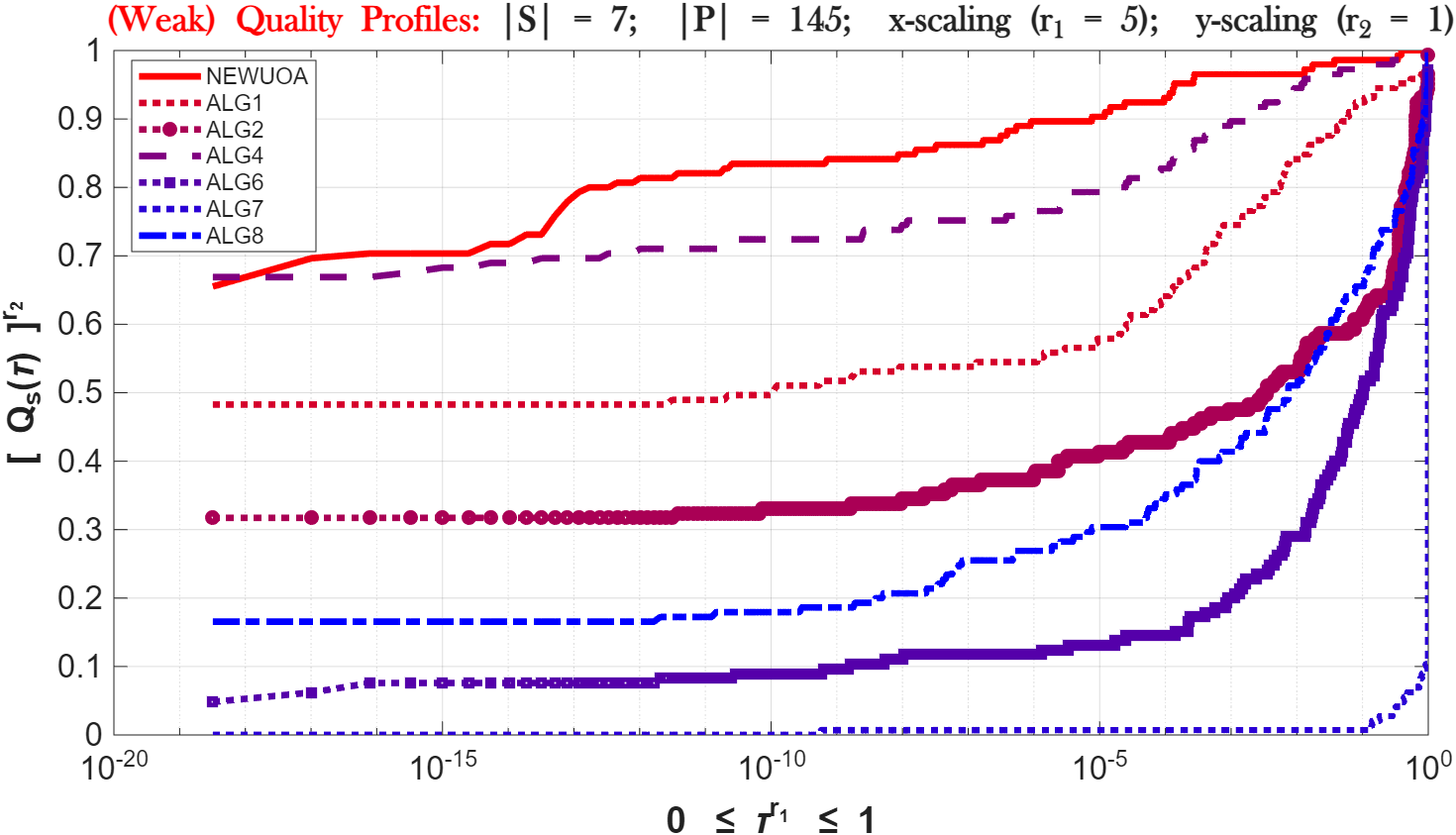} 
		\label{fig:1-case-dfo_5-1_250}
	\end{subfigure}
	\hfill
	\begin{subfigure}[b]{0.7\textwidth}
		\centering
		\includegraphics[width=\textwidth]{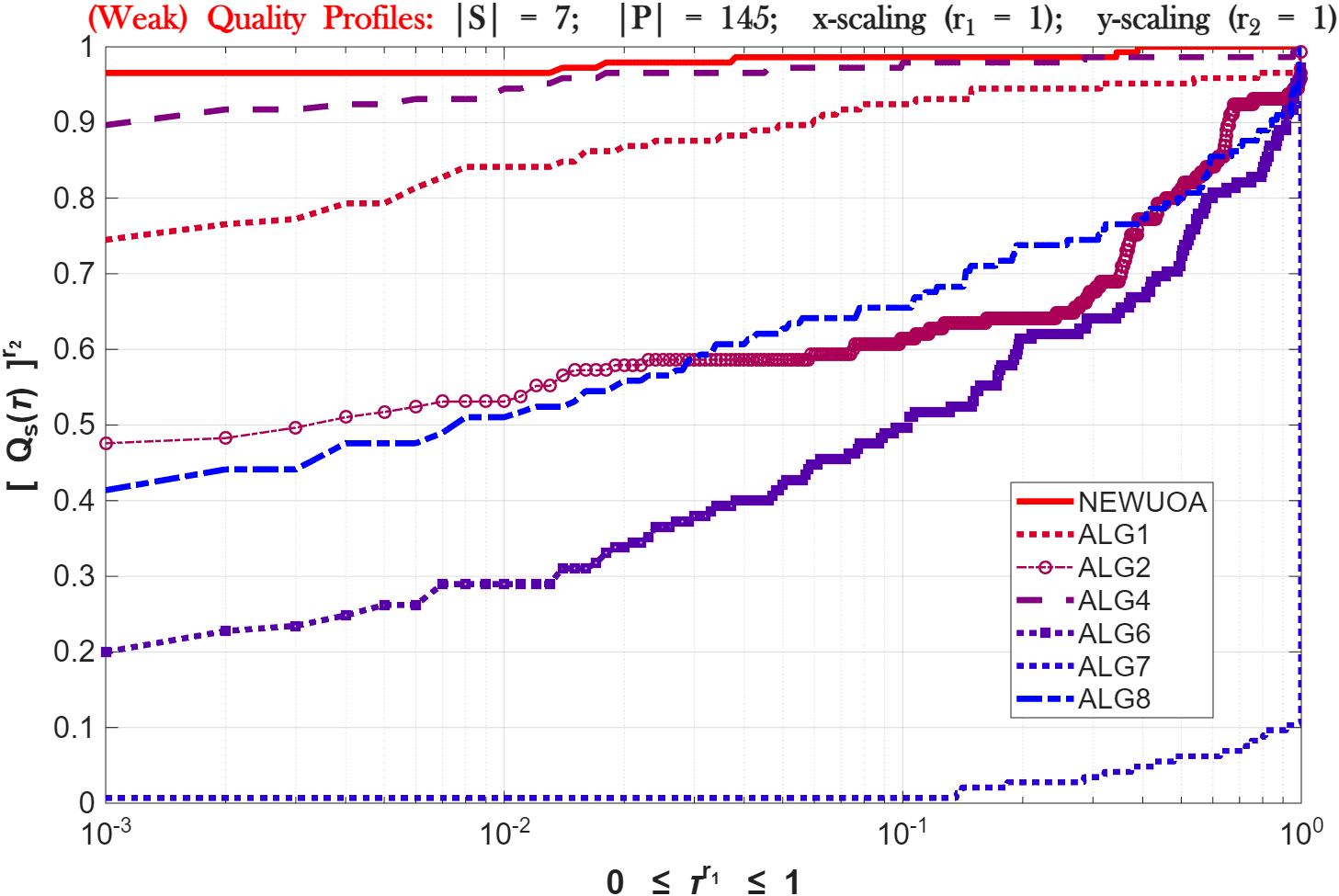} 
		\label{fig:1-case-dfo_1-1_250}
	\end{subfigure}
	\hfill
	\begin{subfigure}[b]{0.7\textwidth}
		\centering
		\includegraphics[width=\textwidth]{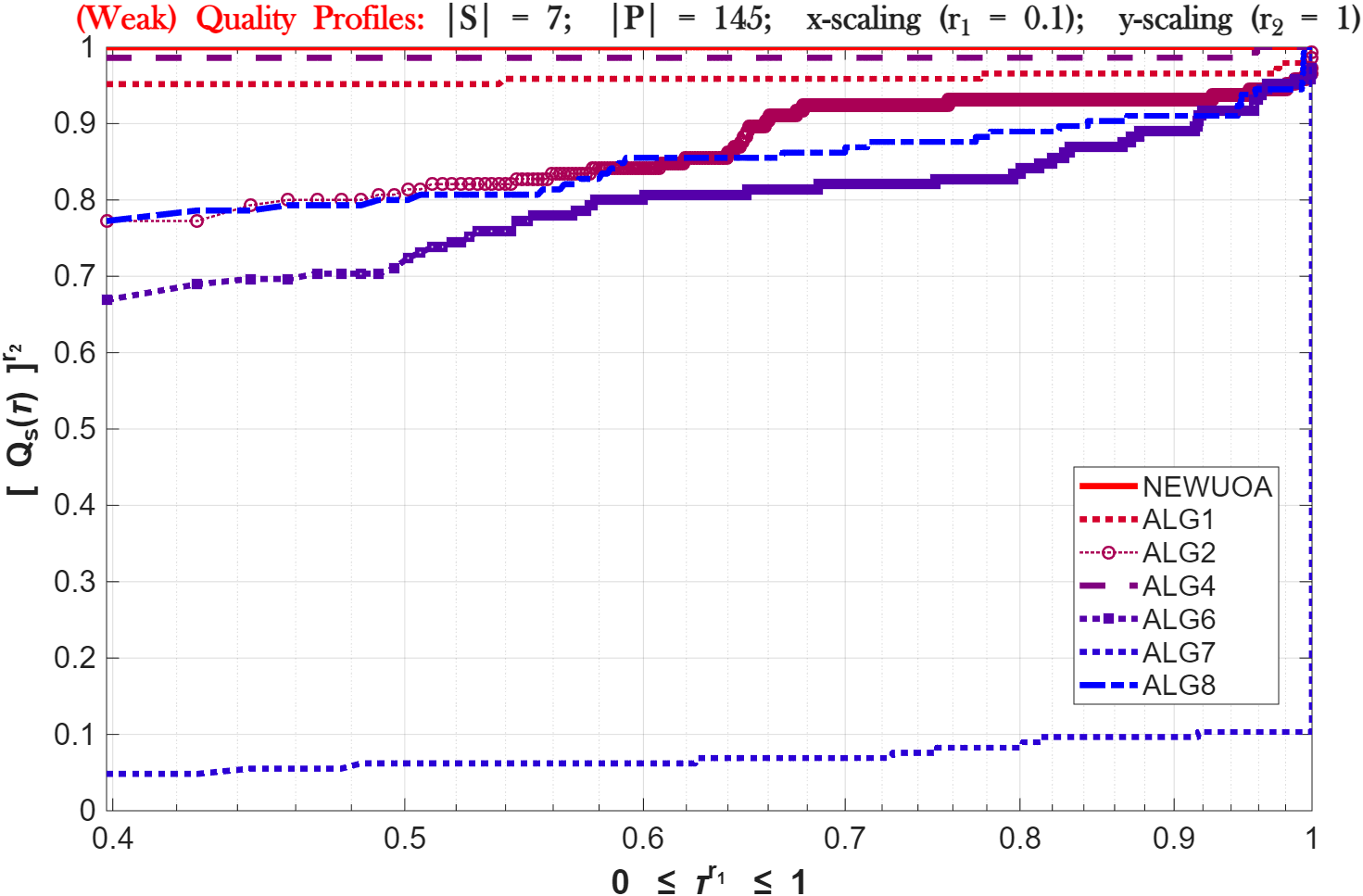} 
		\label{fig:1-case-dfo_01-1_250}
	\end{subfigure}
	\caption{Example of weak quality profiles for the derivative--free solvers in ${\cal S}$. The allotted maximum budget for solvers is set to $250 \cdot n$ {\em function evaluations}.}
	\label{fig:1-case-dfo_250}
\end{figure}

\section{Test Set Profiles}
\label{sec:test_set_profiles}
Here we propose a graphical tool to be coupled with quality profiles, in order both to further analyze the robustness of the solvers in ${\cal S}$, and to give a quantitative measure of the reliability of the test set ${\cal P}$ adopted for benchmarking.

The analysis in this section draws its inspiration from Lemma \ref{lem:robustness_ratios}, as well as from \revnewmax{\cite[Theorem 1]{dolan.more}}, where the potential impact of perturbing the test set (say slightly modifying it by adding or erasing a few problems) may have on the performance profiles. The final indication in \cite{dolan.more} is that the larger $|{\cal P}|$, the smaller the impact, which is surely not a surprising outcome. Nevertheless, the definition of the performance profiles allows for a more analytical conclusion in \cite{dolan.more}--Theorem 1, that possibly cannot be extended neither to data profiles nor to quality profiles.

The fact that the performance of the solvers in ${\cal S}$, on the problems in ${\cal P}$, might be strongly affected by the composition of the test set, yields the obvious conclusion that besides the relative benchmarking of the solvers in ${\cal S}$, the reliability of the test set ${\cal P}$ used for benchmarking is not a less relevant issue. In this regard, a classic reference yielding the last conclusion is given by \cite{wolpert1997no}, where in any case no practical quantitative tools seem indicated to treat the issue. In the case of quality profiles, their assessment may easily suggest one of such a quantitative tools, that we addressed here as {\em test set profiles}. Basically, they are a combination of the procedure adopted for computing quality profiles and the application of the well--known {\em bootstrapping} methodology (the reader is referred to \cite{Davison_Hinkley_1997} for details, and to \cite{Ca_Co_Fa:2024} for a recent practical application of bootstrapping in finance). 

Here we are not interested in giving precise details on the theoretical foundations of bootstrapping and related cross--validation techniques. Nevertheless, we observe that bootstrapping is a statistical technique used to estimate the sampling distribution of an estimator, by repeatedly resampling from the original data set with replacement. It is typically employed to obtain robust estimates of standard errors and confidence intervals for population parameters, based on the Central Limit Theorem. We recall that a number of bootstrapping different techniques have been adopted in the literature. 
\revnewmax{Note that bootstrapping is actually a special kind of {\em Montecarlo simulation} where the random generator is the empirical distribution of the observed data (sampling with replacement).}
\par
In summary, given the test set ${\cal P}$ and the solvers in ${\cal S}$, in the context of the current paper bootstrapping is applied in the following steps:
\begin{enumerate}
	\item compute the quality profiles associated to the pair ${\cal P}$ and ${\cal S}$, as by the previous sections in this paper;
	\item generate a novel test set ${\cal \hat P}$, using the information contained in ${\cal P}$, and compute the quality profiles associated to the pair ${\cal \hat P}$ and ${\cal S}$;
	\item repeat the item 2. a number {\tt cycles} of times.  
\end{enumerate}
The results of the above procedure trivially depend on the bootstrapping technique we are adopting, \ie on the choice for selecting the set ${\cal \hat P}$. A number of different configurations were proposed in the literature, including the next three choices:
\begin{itemize}
	\item resampling the elements in ${\cal P}$ by a simple shuffling and allowing replacement, \ie $|{\cal \hat P}| = |{\cal P}|$ (our choice); 
	\item resampling the elements in ${\cal P}$ by setting a smaller sample size for $|{\cal \hat P}|$, but possibly allowing replacement; 
	\item resampling the elements in ${\cal P}$ by setting a smaller sample size for $|{\cal \hat P}|$, but without replacement.
\end{itemize}
The output of the above procedure 1.--\,3. is given by the family of quality profiles $\{[Q_s(\tau)]^{r_2}\}_{1, \ldots, {\tt cycles}}$, 
so that for any given solver $ s \in {\cal S}$ and for any given value $\tau \in [0,1]$, we can compute the {\em variance} $\sigma_s^2(\tau)$ of the data $\{[Q_s( \tau)]^{r_2}\}_{1, \ldots, {\tt cycles}}$, and then plot it with respect to $\tau \in [0,1]$. 
\par
\revnewmax{From a statistical point of view, test set profiles can be interpreted as an empirical measure of the sensitivity of quality profiles with respect to random perturbation of the benchmarking test set.
Let ${\cal P}$ be regarded as a finite sample extracted from a hypothetical population of optimization problems. The quality profile $Q_s(\tau)$ associated with a solver $s \in {\cal S}$,
may then be interpreted as a statistic computed from the sample ${\cal P}$. Since different samples of optimization problems would generally produce different quality profiles, the quantity $Q_s(\tau)$ is itself a random variable, whose variability reflects the dependence of the benchmarking conclusions on the chosen test set. The bootstrapping procedure described above approximates the sampling distribution of $Q_s(\tau)$ by repeatedly generating resampled test sets
$P^{(1)}, \ldots ,P^{(cycles)},$
and recomputing the corresponding quality profiles. Consequently, the collection
 $\{[Q_s(\tau)]^{r_2}\}_{1, \ldots, {\tt cycles}}$,
can be viewed as an empirical realization of the sampling distribution of the estimator $Q_s(\tau)$.
Under this interpretation, the variance
$\sigma_s^2(\tau)$
estimates the statistical uncertainty associated with the value of the quality profile at the threshold $\tau$. Therefore:
\begin{itemize}
	\item small values of $\sigma_s^2(\tau)$ indicate that the observed quality profile is stable under perturbations of the test set, suggesting that the benchmarking conclusions are robust and weakly dependent on the particular composition of ${\cal P}$;
	\item large values of $\sigma_s^2(\tau)$ indicate instead a strong dependence on the sampled problems, meaning that small modifications of the test set may significantly alter the relative ranking or perceived efficiency of the solvers.
	\end{itemize}
In this sense, test set profiles do not measure the absolute quality of a solver, but rather the statistical reliability of the conclusions drawn from the benchmark itself. Moreover, the behavior of $\sigma_s^2(\tau)$ as a function of $\tau$ provides additional information on the regions where the benchmarking process is most uncertain. For example, large variances concentrated close to $\tau\approx 1$ may indicate instability in the discrimination among high-accuracy solvers, whereas large variances for smaller values of $\tau$ may reveal sensitivity in the identification of overall superior methods.}
\par
An example of the resulting picture on randomly generated data is reported in Figure \ref{fig:test_set_2}, for the case of $|{\cal P}|=100$, ${\cal S}=\{Solver-1,\,Solver-2,\,Solver-3\}$ and respectively
	$$ \begin{array}{lcl}
			f^{(p)}(x_0^{(p)}) = -1 + 2 \cdot rand(p), & \qquad & p=1, \ldots, 100, \\
				& & 	\\
			f_1^{(p)}(x^*) = 
			-2 + 2 \cdot rand(p), & & p=1, \ldots, 100,	\\
				& & 	\\
			f_2^{(p)}(x^*) = 
			-2.01 + 2 \cdot rand(p), & & p=1, \ldots, 100,	\\
				& & 	\\
			f_3^{(p)}(x^*) = -3 + 2 \cdot rand(p),  & & p=1, \ldots, 100,
	   \end{array} $$ 
\begin{figure}[h]
	\centering
	\includegraphics[width=0.9\textwidth]{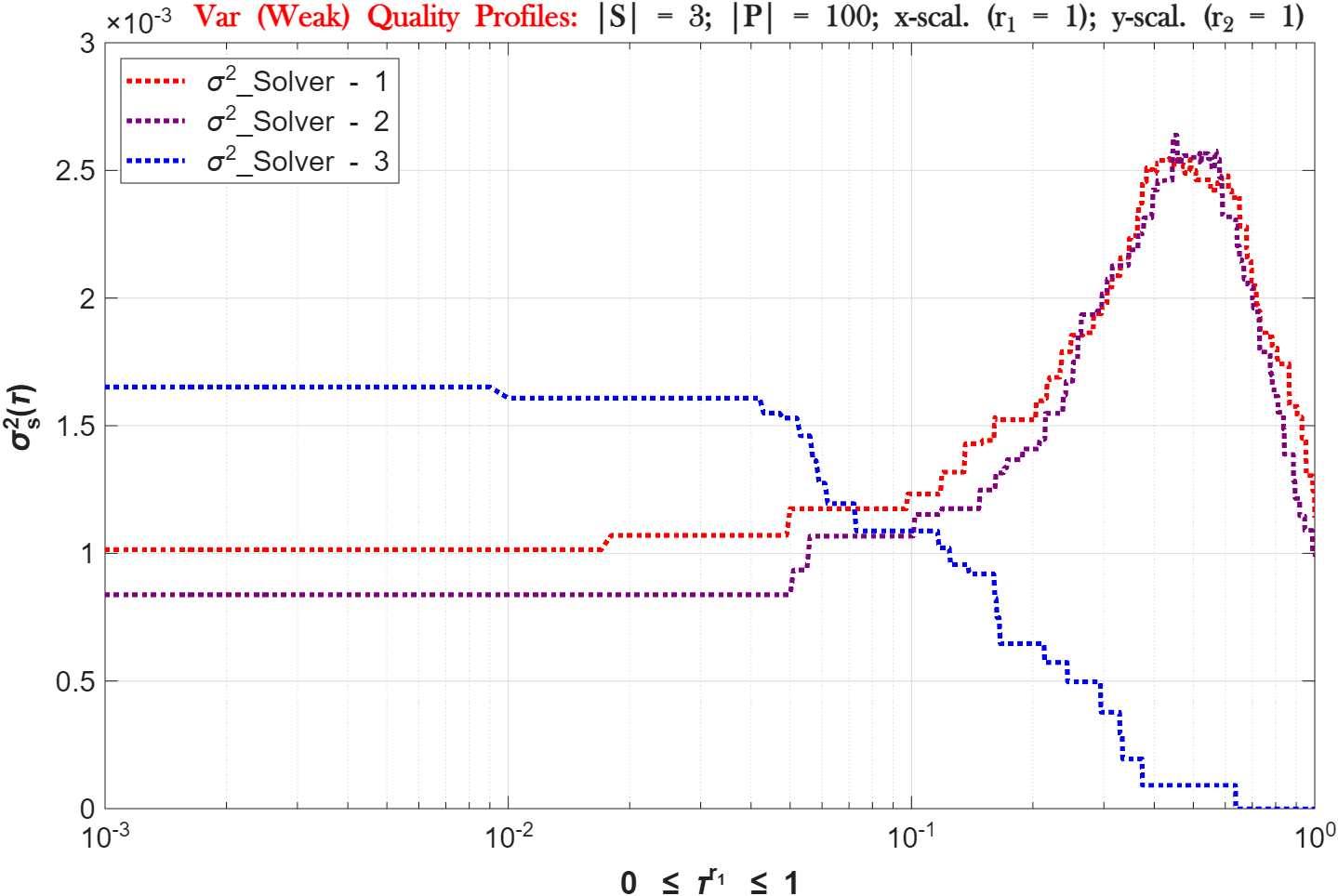}
	\medskip
	\caption{Example of test set profiles for three solvers, where for the solver $s \in {\cal S}$ the relative plot represents the variance $\sigma_s^2(\tau)$ of the data $\{[Q_s(\tau)]^{r_2}\}_{1, \ldots, {\tt cycles}}$.}
	\label{fig:test_set_2}
\end{figure}%
%
being $rand(p)$, $p=1, \ldots, 100$, uniformly distributed random numbers in $[0,1]$, and {\tt cycles}\,=\,1000. Observe that by the last data, the first and second solvers likely show similar values $f_1^{(p)}(x^*)$ and $f_2^{(p)}(x^*)$, for any $p \in {\cal P}$, and they are both worse than $f_3^{(p)}(x^*)$ (\ie the first and second solvers are less accurate than the third one). 

Hence, the following interpretation holds. Given the plots of $\sigma_s^2(\tau)$ and $\sigma_t^2(\tau)$, for any  $s,t \in {\cal S}$, if $\sigma_s^2(\tau) > \sigma_t^2(\tau)$, for $\tau \in T \subseteq [0,1]$, it means that for $\tau \in T$ the code $s$ exhibits relatively higher variance of outcomes (with respect to the solver $t$) in case the test set ${\cal P}$ is perturbed. Hence, $s$ shows less robustness than a solver $t$, for $\tau \in T \subseteq [0,1]$, when they are both applied to ${\cal P}$. Conversely, in case $\sigma_s^2(\tau) < \sigma_t^2(\tau)$, for $\tau \in T \subseteq [0,1]$, then it implies both the robustness of $s$ (with respect to the solver $t$) on ${\cal P}$, as well as the reliability of the test set ${\cal P}$ when used to benchmark $s$. 

Figure \ref{fig:test_set_1} reports the corresponding quality profiles, setting $r_1=r_2=1$ and using the same data for the three solvers and the test set ${\cal P}$ of the Figure \ref{fig:test_set_2}. We can observe the importance of considering at once both Figure~\ref{fig:test_set_2} and Figure~\ref{fig:test_set_1}: the $Solver-3$ appears more robust on the test set ${\cal P}$ when a relatively small precision (say $\tau \leq 10^{-1}$) is required to the solver (see Figure \ref{fig:test_set_2}). Conversely, it shows a relatively smaller robustness (with respect to both $Solver\!-\!1$ and $Solver\!-\!2$) when the precision increases (see again Figure~\ref{fig:test_set_2}). Nevertheless, the $Solver\!-\!3$ shows a higher accuracy as shown in Figure~\ref{fig:test_set_1}, on a wide range of the parameter $\tau \in [0,1]$. As an alternative viewpoint, we may also conclude that the test set ${\cal P}$ is comparably less appropriate to benchmark the $Solver\!-\!3$ (with respect to the solvers $Solver\!-\!1$ and $Solver\!-\!2$) when high precision is required. In any case, we also note that values of the plots on the ordinate axis in Figure \ref{fig:test_set_1} are around $10^{-3}$, so that the test set ${\cal P}$, considering the bootstrapping procedure we adopted, to a large extent shows to be reasonably consistent for the set of solvers ${\cal S}$.
\begin{figure}[h]
	\centering
	\includegraphics[width=0.9\textwidth]{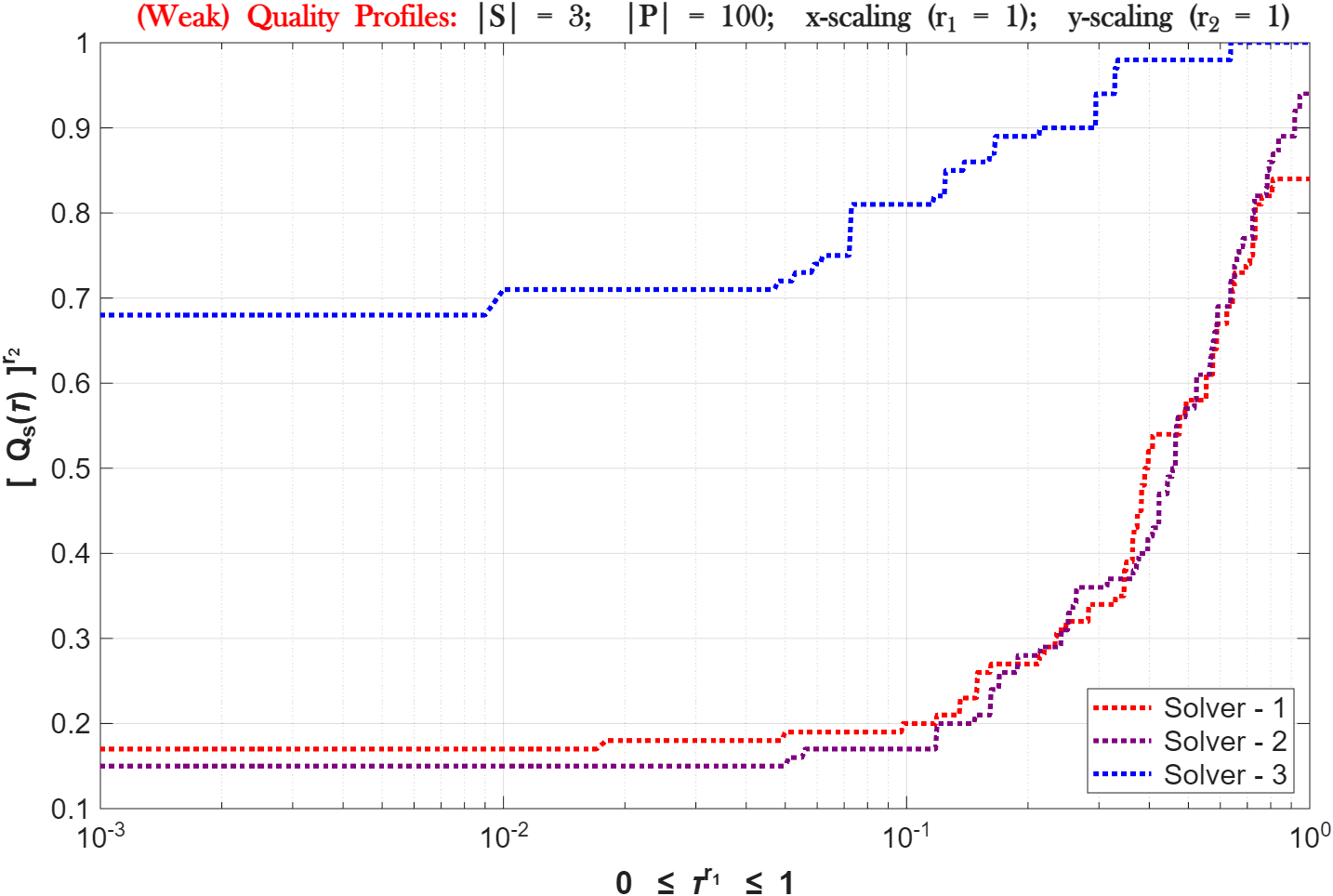}
	\medskip
	\caption{Quality profiles for the 3 solvers in Figure \ref{fig:test_set_2}.}
	\label{fig:test_set_1}
\end{figure}

We now investigate the effect of perturbing ${\cal P}$, by respectively drastically and slightly modifying $|{\cal P}|$. In particular, in case we adopt the above setting for generating data, along with the next two scenarios:
\begin{itemize}
	\item[]
	\item{\bf Scenario 1:} $|{\cal P}|=90$, \ $|{\cal S}|=3$ and {\tt cycles}\,=\,1000;
	\item[]
	\item{\bf Scenario 2:} $|{\cal P}|=50$, \ $|{\cal S}|=3$ and {\tt cycles}\,=\,1000;
	\item[]
\end{itemize}
we obtain the results in Figures \ref{fig:test_set_4} and \ref{fig:test_set_3} (Scenario~1), and Figures \ref{fig:test_set_6} and \ref{fig:test_set_5} (Scenario~2), respectively.
\begin{figure}[h]
	\centering
	\includegraphics[width=0.9\textwidth]{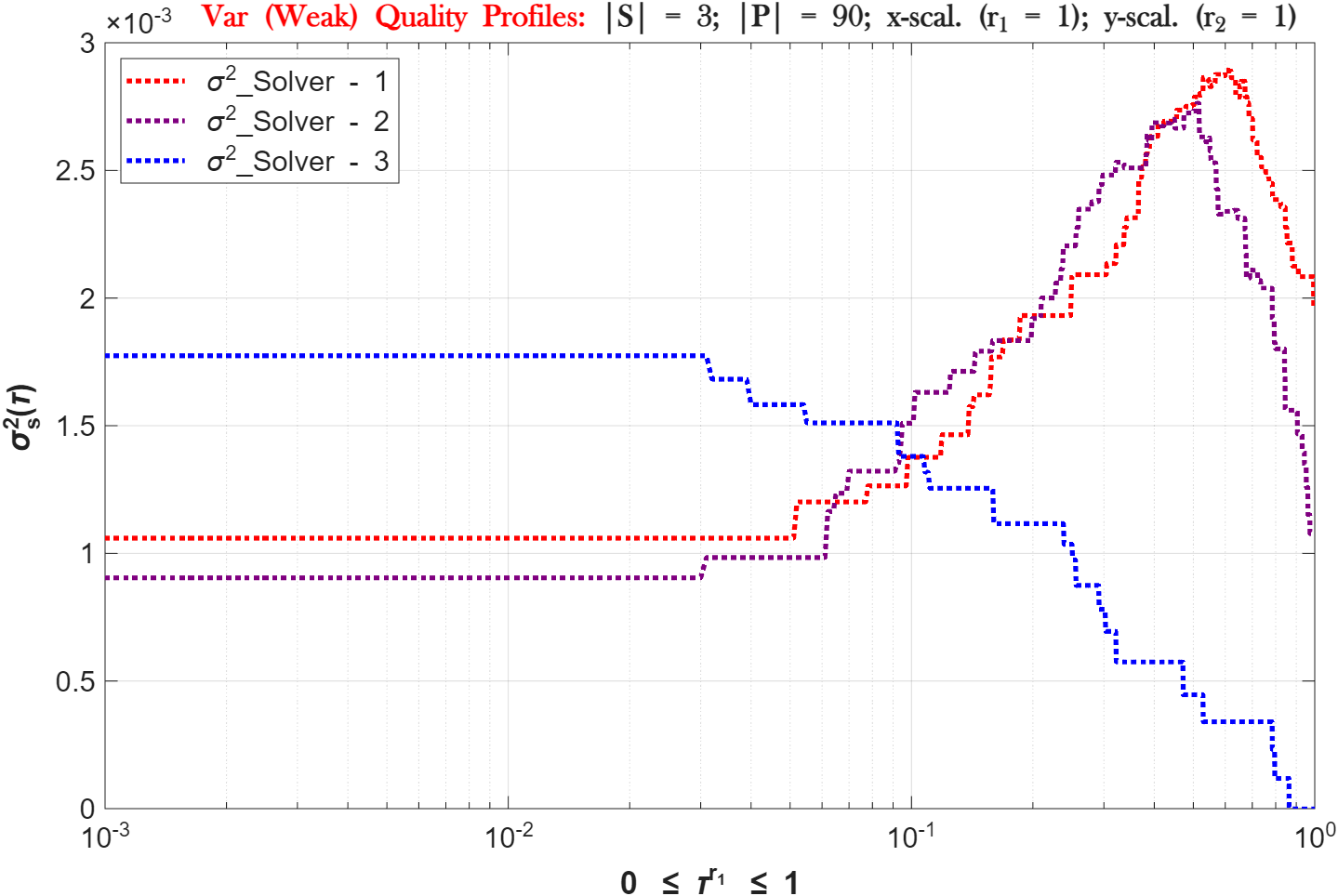}
	\medskip
	\caption{Example of test set profiles for three solvers, where for the solver $s \in {\cal S}$ the relative plot represents the variance $\sigma_s^2(\tau)$ of the data $\{[Q_s(\tau)]^{r_2}\}_{1, \ldots, {\tt cycles}}$. With respect to Figure  \ref{fig:test_set_2} here the test set contains just 90\% of the test problems.}
	\label{fig:test_set_4}
\end{figure}
\begin{figure}[h]
	\centering
	\includegraphics[width=0.9\textwidth]{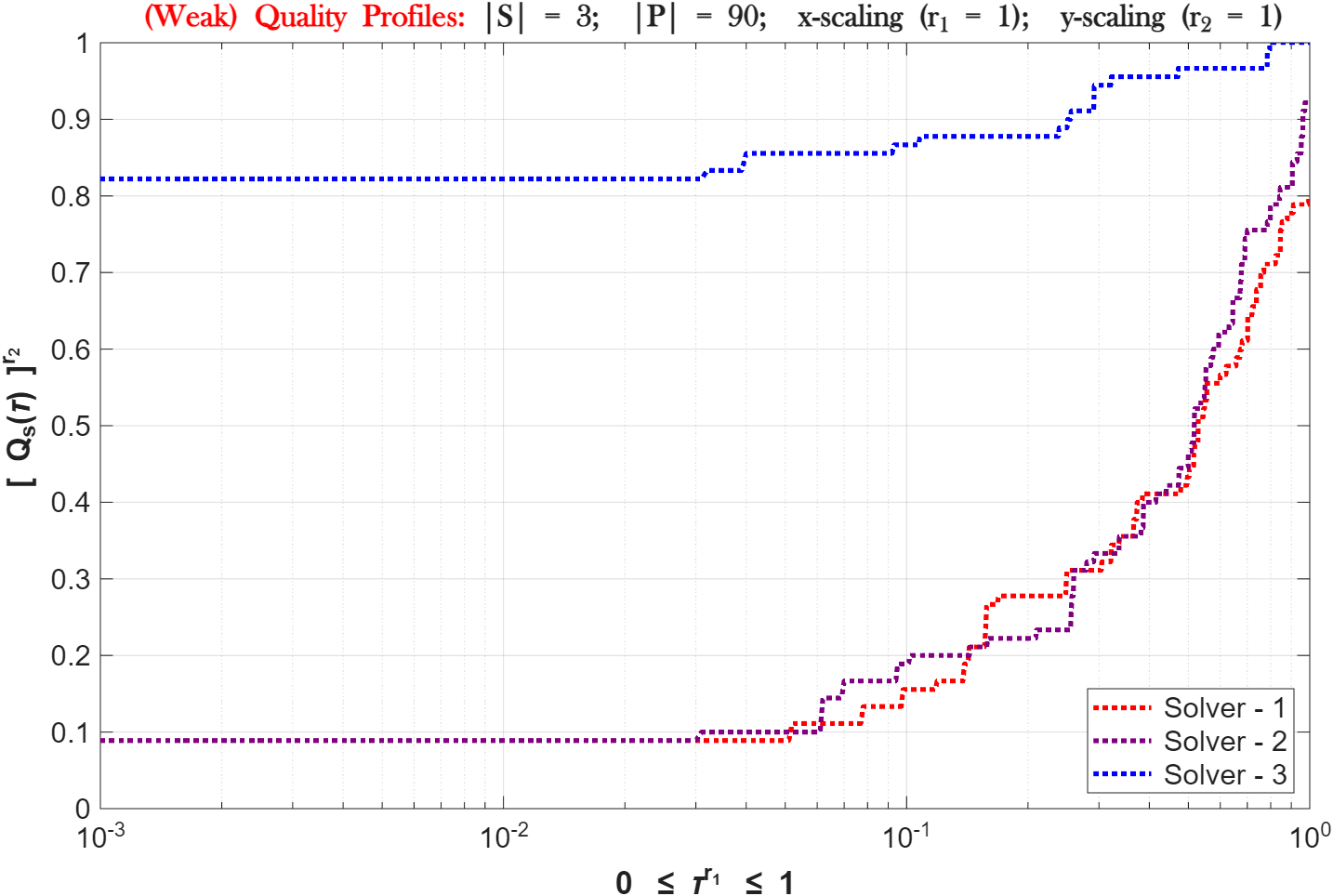}
	\medskip
	\caption{Quality profiles for the 3 solvers in Figure \ref{fig:test_set_4}. With respect to Figure  \ref{fig:test_set_2} here the test set contains just 90\% of the test problems.}
	\label{fig:test_set_3}
\end{figure}
\begin{figure}[h]
	\centering
	\includegraphics[width=0.9\textwidth]{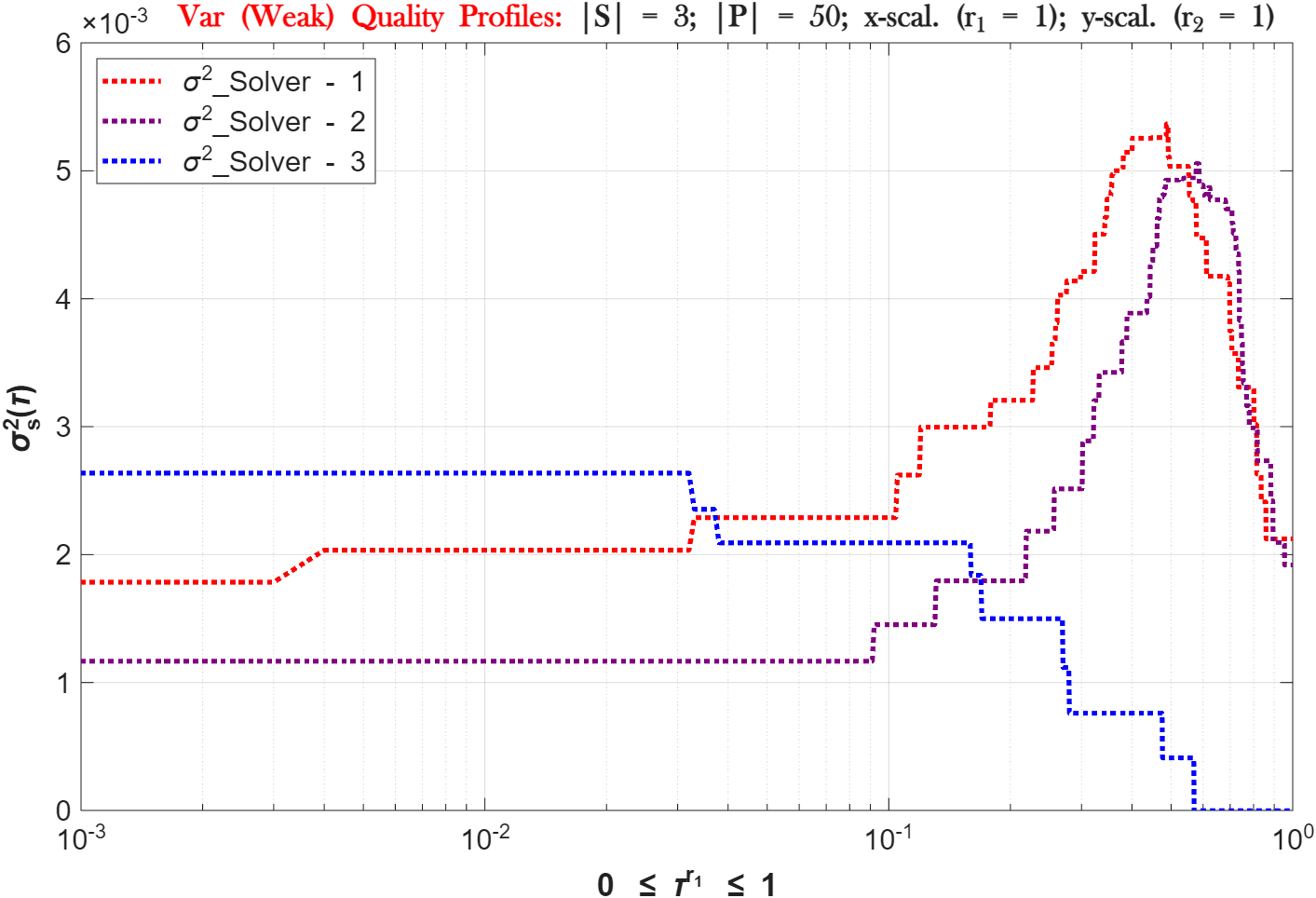}
	\medskip
	\caption{Example of test set profiles for three solvers, where for the solver $s \in {\cal S}$ the relative plot represents the variance $\sigma_s^2(\tau)$ of the data $\{[Q_s(\tau)]^{r_2}\}_{1, \ldots, {\tt cycles}}$. With respect to Figure  \ref{fig:test_set_2} here the test set contains just 50\% of the test problems.}
	\label{fig:test_set_6}
\end{figure}
\begin{figure}[h]
	\centering
	\includegraphics[width=0.9\textwidth]{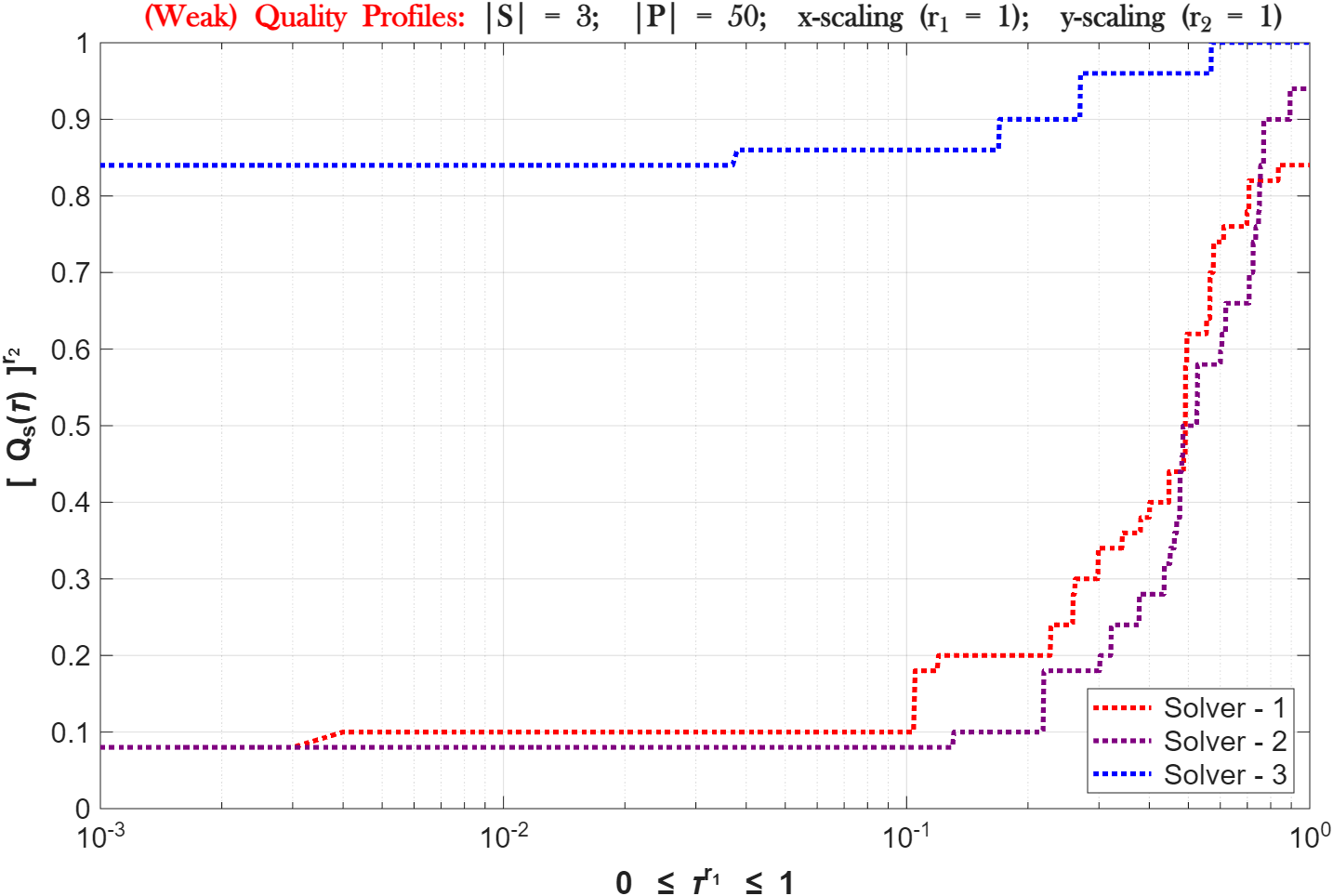}
	\medskip
	\caption{Quality profiles for the 3 solvers in Figure \ref{fig:test_set_6}. With respect to Figure  \ref{fig:test_set_2} here the test set contains just 50\% of the test problems.}
	\label{fig:test_set_5}
\end{figure}
Evidently, a quick comparison among the last pictures evidences that the smaller the cardinality $|{\cal P}|$, the worse the robustness of the solvers, as revealed by an increase of the variance $\sigma_s^2$, with $s \in {\cal S}$, associated with the solvers. This was a largely expected outcome, that perfectly also matches the results reported in \cite{dolan.more}. In addition we also observe that decreasing $|{\cal P}|$, the robustness of the solvers may tend to deteriorate even when smaller values of the precision parameter $\tau$ are considered.

Similarly to what was done in this section for quality profiles, the above bootstrapping technique can be easily coupled with performance profiles, too, showing how robust is the {\em performance} (and no more the {\em accuracy}) of any solver subject to a given perturbation of the test set ${\cal P}$. On the contrary, an immediate extension of the last arrangement also to data profiles shows some criticality, since unlike performance and quality profiles, multiple data profiles are usually generated and must be considered at once.

As a final remark, on the numerical experience reported in this section, observe that setting $|{\cal P}|=100$ and $|{\cal S}|=3$, along with {\tt cycles}\,=\,1000, required a computational time of less than a couple of seconds, on a common laptop, for generating both the related quality profiles and test set profiles.

\section{Conclusions}\label{sec:conclusions}
We propose a couple of novel tools specifically \gianni{designed for benchmarking} algorithms which possibly show convergence to different solution points on a given test set: namely the quality profiles and the test set profiles. They allow us to duly consider the quality of the obtained solutions when comparing algorithms'  outcomes, as well as to issue and analyze the consistency of the test set used for benchmarking the selected solvers. 
\par
\revnewmax{Unlike performance and data profiles, quality profiles prioritize solution quality, though their design also allows for a secondary assessment robustness.} 
\revnewmax{
The goal of the proposed quality profiles is not to repair all weaknesses of performance/data/accuracy profiles, nor to replace them with a universal benchmarking principle. Their goal is narrower and complementary: to provide a comparative tool focused on the quality of the obtained objective values, especially in benchmarking settings where different solvers may converge to different solutions and where a comparison based only on computational effort is incomplete.
}
	
\revnewmax{Moreover,
the present paper does not pursue an axiomatic characterization of benchmarking procedures. In particular, unlike the line of research by Liu et al. \cite{liu2020paradoxes,ZHI2024,liu2024meanborda}, which studies consistency properties and paradox-freeness of comparison rules, our focus is on introducing a practical data-driven profile-based tool for measuring one specific dimension of solver behavior, namely the quality of the final objective value.
Quality profiles are introduced to capture one specific dimension of solver behaviour (the quality of the solution) in a concise graphical form. Accordingly, they should be regarded as complementary to paradox-free ranking or comparison methodologies, rather than as an alternative to them. A theoretical study of quality profiles from an axiomatic or paradox-freeness perspective would be an interesting direction for future research.
}
\par
\revnewmax{We remark once more that the benefit of quality profiles should not be interpreted as solving the benchmarking problem in full generality, but rather as enriching the set of available tools for analyzing solver behavior. They are complementary tools within a broader benchmarking ecosystem that includes paradox-free comparison rules, instance-space analysis, sound experimental methodology and principled benchmark-selection guidelines.
}
\par
We considered two applications of the quality profiles  within large scale smooth optimization: the first \gianni{one considers} the comparison among four well--known algorithms; the second \gianni{one compares algorithms using} different negative curvature directions within a Truncated Newton method. 
\revnewmax{Both applications clearly demonstrate the benefit of using quality profiles to evaluate codes.}
Moreover, unlike performance and data profiles, the chance of possibly zooming in any region of the quality profiles (\ie the area $[0,1] \times [0,1]$) is definitely appealing, and allows one to discriminate among solvers' performances with greater precision. Observe that our proposal, to large extent, shares some similarities with the benchmarking procedure in \cite{beiranvand2017best}, with distinctive differences as detailed in Section \ref{sec:rel_abs_quality}. 
	
In this regard, these authors have also extended the analysis in the current paper to the case where the sets ${\cal S}$ and ${\cal P}$ specifically include derivative--free solvers and nonsmooth problems, respectively. This last consideration follows from recalling that indeed data profiles were clearly conceived to suitably benchmark performance in nonsmooth optimization frameworks, inasmuch as performance profiles for continuous optimization might be inadequate. The numerical comparison reported here highlights the complementarity between the outcomes from data profiles and quality profiles.
\par
Finally, \revnewgianni{in this paper} test set profiles are introduced as a tool for assessing the reliability of a set of test problems when benchmarking algorithms. This issue, even if often neglected, represents a relevant aspect to be considered for an in-depth comparison of different optimization codes. 
\par
\revnewmax{Typically, guidelines address benchmark selection from an ex ante perspective, by classifying test functions according to relevant properties and recommending more representative benchmark suites. By contrast, test set profiles address the complementary ex post question of whether a chosen suite is actually informative and discriminating for the specific solver comparison being carried out. Bootstrapping is used as a supporting inferential tool rather than as the core novelty of the proposal, that lies on the fact that test set profiles target the evaluation of the benchmark suite itself, whereas previous works focus on uncertainty quantification for solver-comparison diagnostics.
}
\par
A possible specific extension to multiobjective optimization frameworks represents a natural future direction which deserves attention, too.

\revnewgianni{ Conversely, the present paper deliberately avoids introducing an overly general procedure for evaluating solver performance, which could also encompass {\em stochastic optimization frameworks}. This choice is mainly motivated by the need to focus on tools designed for more specific purposes, rather than proposing a general framework that might yield inappropriate outcomes for certain classes of optimization benchmarks and algorithms. In future research, we will also investigate the assessment of solver quality in stochastic optimization frameworks, by weakening the definition \eqref{eq:def_Qs} as follows:
\begin{equation}
	\hat Q_s(\tau) = \frac{1}{|{\cal P}|} {\rm size} \left\{ p \in {\cal P} \ : \ \left( \hat f_s^{(p)}(x^\ast) + \xi \right) - \hat f_L^{(p)}  \leq \tau \left[ \hat f^{(p)} ( x_0^{(p)} ) - \hat f_L^{(p)} \right] \right\}.        \label{eq:def_Qs_bis}
\end{equation}
Here $\hat f_s^{(p)}(x^\ast)$ represents an {\em optimal reference value} (e.g., the mean value after a finite number of runs) of a solver $s$ on the test problem $p$, $\xi$ is a random variable with Gaussian distribution (mean value equal to $0$ and standard deviation equal to $1$). Furthermore, $\hat f_L^{(p)}$ is a {\em target reference value} and $\hat f^{(p)} ( x_0^{(p)} )$ in an {\em initial reference value}, to be separately assessed on the optimization framework in hand. We believe that a specific benchmarking approach for stochastic optimization problems will also be useful for better explaining and distinguishing the role of the outcomes produced by quality profiles and test set profiles.
}

\bigskip

\bmhead{Acknowledgements} 
G. Fasano thanks the Institute of Marine Engineering (CNR-INM), Italy. G. Fasano and M. Roma are
grateful to the working group GNCS of IN$\delta$AM (Istituto Nazionale di Alta Matematica), Italy, for the support they received. 

\bmhead{Funding} No financial support was received to carry on the research detailed in the current paper. 
\bmhead{Data Availability} Both data and source codes used for the reported numerical experience are publicly available (see the references cited in the text).
\bmhead{Code Availability} {\tt MATLAB} code for creating  quality profiles is available at https://github.com/fasano-g/Quality-Profiles.git, under GPL-3.0 license.
\bigskip

\section*{Declarations}
\bmhead{Conflict of interest} The authors have no conflict of interest.
\bigskip


\begin{thebibliography}{51}
	\ifx \bisbn   \undefined \def \bisbn  #1{ISBN #1}\fi
	\ifx \binits  \undefined \def \binits#1{#1}\fi
	\ifx \bauthor  \undefined \def \bauthor#1{#1}\fi
	\ifx \batitle  \undefined \def \batitle#1{#1}\fi
	\ifx \bjtitle  \undefined \def \bjtitle#1{#1}\fi
	\ifx \bvolume  \undefined \def \bvolume#1{\textbf{#1}}\fi
	\ifx \byear  \undefined \def \byear#1{#1}\fi
	\ifx \bissue  \undefined \def \bissue#1{#1}\fi
	\ifx \bfpage  \undefined \def \bfpage#1{#1}\fi
	\ifx \blpage  \undefined \def \blpage #1{#1}\fi
	\ifx \burl  \undefined \def \burl#1{\textsf{#1}}\fi
	\ifx \doiurl  \undefined \def \doiurl#1{\url{https://doi.org/#1}}\fi
	\ifx \betal  \undefined \def \betal{\textit{et al.}}\fi
	\ifx \binstitute  \undefined \def \binstitute#1{#1}\fi
	\ifx \binstitutionaled  \undefined \def \binstitutionaled#1{#1}\fi
	\ifx \bctitle  \undefined \def \bctitle#1{#1}\fi
	\ifx \beditor  \undefined \def \beditor#1{#1}\fi
	\ifx \bpublisher  \undefined \def \bpublisher#1{#1}\fi
	\ifx \bbtitle  \undefined \def \bbtitle#1{#1}\fi
	\ifx \bedition  \undefined \def \bedition#1{#1}\fi
	\ifx \bseriesno  \undefined \def \bseriesno#1{#1}\fi
	\ifx \blocation  \undefined \def \blocation#1{#1}\fi
	\ifx \bsertitle  \undefined \def \bsertitle#1{#1}\fi
	\ifx \bsnm \undefined \def \bsnm#1{#1}\fi
	\ifx \bsuffix \undefined \def \bsuffix#1{#1}\fi
	\ifx \bparticle \undefined \def \bparticle#1{#1}\fi
	\ifx \barticle \undefined \def \barticle#1{#1}\fi
	\bibcommenthead
	\ifx \bconfdate \undefined \def \bconfdate #1{#1}\fi
	\ifx \botherref \undefined \def \botherref #1{#1}\fi
	\ifx \url \undefined \def \url#1{\textsf{#1}}\fi
	\ifx \bchapter \undefined \def \bchapter#1{#1}\fi
	\ifx \bbook \undefined \def \bbook#1{#1}\fi
	\ifx \bcomment \undefined \def \bcomment#1{#1}\fi
	\ifx \oauthor \undefined \def \oauthor#1{#1}\fi
	\ifx \citeauthoryear \undefined \def \citeauthoryear#1{#1}\fi
	\ifx \endbibitem  \undefined \def \endbibitem {}\fi
	\ifx \bconflocation  \undefined \def \bconflocation#1{#1}\fi
	\ifx \arxivurl  \undefined \def \arxivurl#1{\textsf{#1}}\fi
	\csname PreBibitemsHook\endcsname
	
	\bibitem[\protect\citeauthoryear{Mor\'e and Wild}{2009}]{more.wild}
	\begin{barticle}
		\bauthor{\bsnm{Mor\'e}, \binits{J.}},
		\bauthor{\bsnm{Wild}, \binits{S.}}:
		\batitle{Benchmarking derivative--free optimization algorithms}.
		\bjtitle{SIAM Journal on Optimization}
		\bvolume{20},
		\bfpage{172}--\blpage{191}
		(\byear{2009})
	\end{barticle}
	\endbibitem
	
	\bibitem[\protect\citeauthoryear{Hansen et~al.}{2021}]{hansen2021coco}
	\begin{barticle}
		\bauthor{\bsnm{Hansen}, \binits{N.}},
		\bauthor{\bsnm{Auger}, \binits{A.}},
		\bauthor{\bsnm{Ros}, \binits{R.}},
		\bauthor{\bsnm{Mersmann}, \binits{O.}},
		\bauthor{\bsnm{Tu{\v{s}}ar}, \binits{T.}},
		\bauthor{\bsnm{Brockhoff}, \binits{D.}}:
		\batitle{Coco: A platform for comparing continuous optimizers in a black-box
			setting}.
		\bjtitle{Optimization Methods and Software}
		\bvolume{36}(\bissue{1}),
		\bfpage{114}--\blpage{144}
		(\byear{2021})
	\end{barticle}
	\endbibitem
	
	\bibitem[\protect\citeauthoryear{Beiranvand et~al.}{2017}]{beiranvand2017best}
	\begin{barticle}
		\bauthor{\bsnm{Beiranvand}, \binits{V.}},
		\bauthor{\bsnm{Hare}, \binits{W.}},
		\bauthor{\bsnm{Lucet}, \binits{Y.}}:
		\batitle{Best practices for comparing optimization algorithms}.
		\bjtitle{Optimization and Engineering}
		\bvolume{18},
		\bfpage{815}--\blpage{848}
		(\byear{2017})
	\end{barticle}
	\endbibitem
	
	\bibitem[\protect\citeauthoryear{Dolan and Mor\'e}{2002}]{dolan.more}
	\begin{barticle}
		\bauthor{\bsnm{Dolan}, \binits{E.D.}},
		\bauthor{\bsnm{Mor\'e}, \binits{J.}}:
		\batitle{Benchmarking optimization software with performance profiles}.
		\bjtitle{Mathematical Programming}
		\bvolume{91},
		\bfpage{201}--\blpage{213}
		(\byear{2002})
	\end{barticle}
	\endbibitem
	
	\bibitem[\protect\citeauthoryear{Gould and Scott}{2016}]{gould2016note}
	\begin{barticle}
		\bauthor{\bsnm{Gould}, \binits{N.}},
		\bauthor{\bsnm{Scott}, \binits{J.}}:
		\batitle{A note on performance profiles for benchmarking software}.
		\bjtitle{ACM Transactions on Mathematical Software (TOMS)}
		\bvolume{43}(\bissue{2}),
		\bfpage{1}--\blpage{5}
		(\byear{2016})
	\end{barticle}
	\endbibitem
	
	\bibitem[\protect\citeauthoryear{Hekmati and
		Mirhajianmoghadam}{2019}]{hekmati2019nested}
	\begin{barticle}
		\bauthor{\bsnm{Hekmati}, \binits{R.}},
		\bauthor{\bsnm{Mirhajianmoghadam}, \binits{H.}}:
		\batitle{Nested performance profiles for benchmarking software}.
		\bjtitle{Statistics, Optimization \& Information Computing}
		\bvolume{7}(\bissue{4}),
		\bfpage{709}--\blpage{715}
		(\byear{2019})
	\end{barticle}
	\endbibitem
	
	\bibitem[\protect\citeauthoryear{Audet and Hare}{2017}]{audet.hare.book}
	\begin{bbook}
		\bauthor{\bsnm{Audet}, \binits{C.}},
		\bauthor{\bsnm{Hare}, \binits{W.}}:
		\bbtitle{Derivative--Free and Blackbox Optimization}.
		\bsertitle{Springer Series in Operations Research and Financial Engineering}.
		\bpublisher{Springer},
		\blocation{Cham, Switzerland}
		(\byear{2017})
	\end{bbook}
	\endbibitem
	
	\bibitem[\protect\citeauthoryear{Hare and
		Sagastiz{\'a}bal}{2006}]{hare2006benchmark}
	\begin{barticle}
		\bauthor{\bsnm{Hare}, \binits{W.}},
		\bauthor{\bsnm{Sagastiz{\'a}bal}, \binits{C.}}:
		\batitle{Benchmark of some nonsmooth optimization solvers for computing
			nonconvex proximal points}.
		\bjtitle{Pacific Journal on Optimization}
		\bvolume{3},
		\bfpage{545}--\blpage{573}
		(\byear{2006})
	\end{barticle}
	\endbibitem
	
	\bibitem[\protect\citeauthoryear{Wolpert et~al.}{1997}]{wolpert1997no}
	\begin{barticle}
		\bauthor{\bsnm{Wolpert}, \binits{D.H.}},
		\bauthor{\bsnm{Macready}, \binits{W.G.}}, \betal:
		\batitle{No free lunch theorems for optimization}.
		\bjtitle{IEEE Transactions on Evolutionary Computation}
		\bvolume{1}(\bissue{1}),
		\bfpage{67}--\blpage{82}
		(\byear{1997})
	\end{barticle}
	\endbibitem
	
	\bibitem[\protect\citeauthoryear{Liu et~al.}{2020}]{liu2020paradoxes}
	\begin{barticle}
		\bauthor{\bsnm{Liu}, \binits{Q.}},
		\bauthor{\bsnm{Gehrlein}, \binits{W.V.}},
		\bauthor{\bsnm{Wang}, \binits{L.}},
		\bauthor{\bsnm{Yan}, \binits{Y.}},
		\bauthor{\bsnm{Cao}, \binits{Y.}},
		\bauthor{\bsnm{Chen}, \binits{W.}},
		\bauthor{\bsnm{Li}, \binits{Y.}}:
		\batitle{Paradoxes in numerical comparison of optimization algorithms}.
		\bjtitle{IEEE Transactions on Evolutionary Computation}
		\bvolume{24}(\bissue{4}),
		\bfpage{777}--\blpage{791}
		(\byear{2020})
		\doiurl{10.1109/TEVC.2019.2955110}
	\end{barticle}
	\endbibitem
	
	\bibitem[\protect\citeauthoryear{Yan et~al.}{2022}]{yan2022paradox}
	\begin{botherref}
		\oauthor{\bsnm{Yan}, \binits{Y.}},
		\oauthor{\bsnm{Liu}, \binits{Q.}},
		\oauthor{\bsnm{Li}, \binits{Y.}}:
		Paradox-free analysis for comparing the performance of optimization algorithms.
		IEEE Transactions on Evolutionary Computation
		(2022)
	\end{botherref}
	\endbibitem
	
	\bibitem[\protect\citeauthoryear{Zhi et~al.}{2024}]{ZHI2024}
	\begin{botherref}
		\oauthor{\bsnm{Zhi}, \binits{Y.}},
		\oauthor{\bsnm{Liu}, \binits{Q.}},
		\oauthor{\bsnm{Yan}, \binits{Y.}},
		\oauthor{\bsnm{Jing}, \binits{Y.}}:
		Paradox-free data analysis in numerical comparisons of optimization
		algorithms:.
		International Journal of Swarm Intelligence Research
		\textbf{15}(1)
		(2024)
		\doiurl{10.4018/IJSIR.356514}
	\end{botherref}
	\endbibitem
	
	\bibitem[\protect\citeauthoryear{da~Fonseca and
		Fonseca}{2010}]{da2010attainment}
	\begin{bchapter}
		\bauthor{\bsnm{Fonseca}, \binits{V.G.}},
		\bauthor{\bsnm{Fonseca}, \binits{C.M.}}:
		\bctitle{The attainment-function approach to stochastic multiobjective
			optimizer assessment and comparison}.
		In: \bbtitle{Experimental Methods for the Analysis of Optimization Algorithms},
		pp. \bfpage{103}--\blpage{130}.
		\bpublisher{Springer},
		\blocation{Berlin, Heidelberg}
		(\byear{2010})
	\end{bchapter}
	\endbibitem
	
	\bibitem[\protect\citeauthoryear{L{\'o}pez-Ib{\'a}{\~n}ez
		et~al.}{2024}]{lopez2024using}
	\begin{botherref}
		\oauthor{\bsnm{L{\'o}pez-Ib{\'a}{\~n}ez}, \binits{M.}},
		\oauthor{\bsnm{Vermetten}, \binits{D.}},
		\oauthor{\bsnm{Dreo}, \binits{J.}},
		\oauthor{\bsnm{Doerr}, \binits{C.}}:
		Using the empirical attainment function for analyzing single-objective
		black-box optimization algorithms.
		IEEE Transactions on Evolutionary Computation
		(2024)
	\end{botherref}
	\endbibitem
	
	\bibitem[\protect\citeauthoryear{Smith-Miles et~al.}{2014}]{smith2014objective}
	\begin{barticle}
		\bauthor{\bsnm{Smith-Miles}, \binits{K.}},
		\bauthor{\bsnm{Baatar}, \binits{D.}},
		\bauthor{\bsnm{Wreford}, \binits{B.}},
		\bauthor{\bsnm{Lewis}, \binits{R.}}:
		\batitle{Towards objective measures of algorithm performance across instance
			space}.
		\bjtitle{Computers \& Operations Research}
		\bvolume{45},
		\bfpage{12}--\blpage{24}
		(\byear{2014})
		\doiurl{10.1016/j.cor.2013.11.015}
	\end{barticle}
	\endbibitem
	
	\bibitem[\protect\citeauthoryear{Willemsen
		et~al.}{2024}]{willemsen2024methodology}
	\begin{botherref}
		\oauthor{\bsnm{Willemsen}, \binits{F.-J.}},
		\oauthor{\bsnm{Schoonhoven}, \binits{R.}},
		\oauthor{\bsnm{Filipovi{\v{c}}}, \binits{J.}},
		\oauthor{\bsnm{T{\o}rring}, \binits{J.O.}},
		\oauthor{\bsnm{Nieuwpoort}, \binits{R.}},
		\oauthor{\bsnm{Werkhoven}, \binits{B.}}:
		A methodology for comparing optimization algorithms for auto-tuning.
		Future Generation Computer Systems
		(2024)
	\end{botherref}
	\endbibitem
	
	\bibitem[\protect\citeauthoryear{Opara and
		Arabas}{2023}]{opara2023benchmarking}
	\begin{barticle}
		\bauthor{\bsnm{Opara}, \binits{K.}},
		\bauthor{\bsnm{Arabas}, \binits{J.}}:
		\batitle{Benchmarking procedures for continuous optimization algorithms}.
		\bjtitle{Journal of Telecommunications and Information Technology}
		(\byear{2023})
		\doiurl{10.26636/jtit.2011.4.1180}
	\end{barticle}
	\endbibitem
	
	\bibitem[\protect\citeauthoryear{Porcelli and Toint}{2017}]{porcelli2017note}
	\begin{barticle}
		\bauthor{\bsnm{Porcelli}, \binits{M.}},
		\bauthor{\bsnm{Toint}, \binits{P.L.}}:
		\batitle{A note on using performance and data profiles for training
			algorithms}.
		\bjtitle{ACM Transactions on Mathematical Software}
		(\byear{2017})
		\doiurl{10.1145/3310362}
	\end{barticle}
	\endbibitem
	
	\bibitem[\protect\citeauthoryear{Naser
		et~al.}{2025}]{https://doi.org/10.1002/wics.70028}
	\begin{barticle}
		\bauthor{\bsnm{Naser}, \binits{M.Z.}},
		\bauthor{\bsnm{Al-Bashiti}, \binits{M.K.}},
		\bauthor{\bsnm{Tapeh}, \binits{A.T.G.}},
		\bauthor{\bsnm{Naser}, \binits{A.}},
		\bauthor{\bsnm{Kodur}, \binits{V.}},
		\bauthor{\bsnm{Hawileh}, \binits{R.}},
		\bauthor{\bsnm{Abdalla}, \binits{J.}},
		\bauthor{\bsnm{Khodadadi}, \binits{N.}},
		\bauthor{\bsnm{Gandomi}, \binits{A.H.}},
		\bauthor{\bsnm{Eslamlou}, \binits{A.D.}}:
		\batitle{A review of benchmark and test functions for global optimization
			algorithms and metaheuristics}.
		\bjtitle{WIREs Computational Statistics}
		\bvolume{17}(\bissue{2}),
		\bfpage{70028}
		(\byear{2025})
		\doiurl{10.1002/wics.70028}
		{\href{https://arxiv.org/abs/https://wires.onlinelibrary.wiley.com/doi/pdf/10.1002/wics.70028}{{https://wires.onlinelibrary.wiley.com/doi/pdf/10.1002/wics.70028}}}.
		\bcomment{e70028 EOCS-767.R3}
	\end{barticle}
	\endbibitem
	
	\bibitem[\protect\citeauthoryear{Eckman et~al.}{2023}]{Eckman_et_al_2023}
	\begin{barticle}
		\bauthor{\bsnm{Eckman}, \binits{D.J.}},
		\bauthor{\bsnm{Henderson}, \binits{S.G.}},
		\bauthor{\bsnm{Shashaani}, \binits{S.}}:
		\batitle{Diagnostic tools for evaluating and comparing simulation-optimization
			algorithms}.
		\bjtitle{INFORMS Journal on Computing}
		\bvolume{35(2)},
		\bfpage{350}--\blpage{367}
		(\byear{2023})
	\end{barticle}
	\endbibitem
	
	\bibitem[\protect\citeauthoryear{{The MathWorks, Inc.}}{2025}]{MATLAB}
	\begin{bbook}
		\bauthor{\bsnm{{The MathWorks, Inc.}}}:
		\bbtitle{MATLAB}.
		\bpublisher{The MathWorks, Inc.},
		\blocation{Natick, Massachusetts}
		(\byear{2025}).
		\bcomment{The MathWorks, Inc.. Versione R2025a}.
		\burl{https://www.mathworks.com}
	\end{bbook}
	\endbibitem
	
	\bibitem[\protect\citeauthoryear{Gould et~al.}{2015}]{cutest}
	\begin{barticle}
		\bauthor{\bsnm{Gould}, \binits{N.I.M.}},
		\bauthor{\bsnm{Orban}, \binits{D.}},
		\bauthor{\bsnm{Toint}, \binits{P.L.}}:
		\batitle{{\sf {CUTEst:}} a constrained and unconstrained testing environment
			with safe threads}.
		\bjtitle{Computational Optimization and Applications}
		\bvolume{60},
		\bfpage{545}--\blpage{557}
		(\byear{2015})
	\end{barticle}
	\endbibitem
	
	\bibitem[\protect\citeauthoryear{Nocedal}{1980}]{nocedal:80}
	\begin{barticle}
		\bauthor{\bsnm{Nocedal}, \binits{J.}}:
		\batitle{Updating {Q}uasi-{N}ewton matrices with limited storage}.
		\bjtitle{Mathematics of Computation}
		\bvolume{35},
		\bfpage{773}--\blpage{782}
		(\byear{1980})
	\end{barticle}
	\endbibitem
	
	\bibitem[\protect\citeauthoryear{Liu and Nocedal}{1989}]{liu.nocedal}
	\begin{barticle}
		\bauthor{\bsnm{Liu}, \binits{D.C.}},
		\bauthor{\bsnm{Nocedal}, \binits{J.}}:
		\batitle{On the limited memory {BFGS} method for large scale optimization}.
		\bjtitle{Mathematical Programming}
		\bvolume{45},
		\bfpage{503}--\blpage{528}
		(\byear{1989})
	\end{barticle}
	\endbibitem
	
	\bibitem[\protect\citeauthoryear{Gilbert and Nocedal}{1992}]{gilbert.nocedal}
	\begin{barticle}
		\bauthor{\bsnm{Gilbert}, \binits{J.C.}},
		\bauthor{\bsnm{Nocedal}, \binits{J.}}:
		\batitle{Global convergence properties of conjugate gradient methods for
			optimization}.
		\bjtitle{{SIAM} Journal on Optimization}
		\bvolume{2},
		\bfpage{21}--\blpage{42}
		(\byear{1992})
	\end{barticle}
	\endbibitem
	
	\bibitem[\protect\citeauthoryear{Lin and Mor\'e}{1999}]{lin.more}
	\begin{barticle}
		\bauthor{\bsnm{Lin}, \binits{C.-J.}},
		\bauthor{\bsnm{Mor\'e}, \binits{J.}}:
		\batitle{{N}ewton's method for large bound--constrained optimization problems}.
		\bjtitle{{SIAM} Journal on Optimization}
		\bvolume{9},
		\bfpage{1100}--\blpage{1127}
		(\byear{1999})
	\end{barticle}
	\endbibitem
	
	\bibitem[\protect\citeauthoryear{Hager and Zhang}{2005}]{hager.zhang:05}
	\begin{barticle}
		\bauthor{\bsnm{Hager}, \binits{W.}},
		\bauthor{\bsnm{Zhang}, \binits{H.}}:
		\batitle{A new conjugate gradient method with guaranteed descent and efficient
			line search}.
		\bjtitle{{SIAM} Journal on Optimization}
		\bvolume{16},
		\bfpage{170}--\blpage{192}
		(\byear{2005})
	\end{barticle}
	\endbibitem
	
	\bibitem[\protect\citeauthoryear{Hager and Zhang}{2006}]{hager.zhang:06b}
	\begin{barticle}
		\bauthor{\bsnm{Hager}, \binits{W.}},
		\bauthor{\bsnm{Zhang}, \binits{H.}}:
		\batitle{Algorithm 851: {CG\_DESCENT}, a conjugate gradient method with
			guaranteed descent}.
		\bjtitle{ACM Transactions on Mathematical Software}
		\bvolume{32},
		\bfpage{113}--\blpage{137}
		(\byear{2006})
	\end{barticle}
	\endbibitem
	
	\bibitem[\protect\citeauthoryear{Curtis and
		Robinson}{2019}]{curtis.robinson:19}
	\begin{barticle}
		\bauthor{\bsnm{Curtis}, \binits{F.E.}},
		\bauthor{\bsnm{Robinson}, \binits{D.P.}}:
		\batitle{Exploiting negative curvature in deterministic and stochastic
			optimization}.
		\bjtitle{{Mathematical Programming}}
		\bvolume{176},
		\bfpage{69}--\blpage{94}
		(\byear{2019})
	\end{barticle}
	\endbibitem
	
	\bibitem[\protect\citeauthoryear{Mor\'e and Sorensen}{1979}]{more.sorensen:use}
	\begin{barticle}
		\bauthor{\bsnm{Mor\'e}, \binits{J.J.}},
		\bauthor{\bsnm{Sorensen}, \binits{D.C.}}:
		\batitle{On the use of directions of negative curvature in a modified {N}ewton
			method}.
		\bjtitle{Mathematical Programming}
		\bvolume{16},
		\bfpage{1}--\blpage{20}
		(\byear{1979})
	\end{barticle}
	\endbibitem
	
	\bibitem[\protect\citeauthoryear{Lucidi et~al.}{1998}]{lucidi.rochetich.roma}
	\begin{barticle}
		\bauthor{\bsnm{Lucidi}, \binits{S.}},
		\bauthor{\bsnm{Rochetich}, \binits{F.}},
		\bauthor{\bsnm{Roma}, \binits{M.}}:
		\batitle{Curvilinear stabilization techniques for truncated {N}ewton methods in
			large scale unconstrained optimization}.
		\bjtitle{{SIAM} Journal on Optimization}
		\bvolume{8},
		\bfpage{916}--\blpage{939}
		(\byear{1998})
	\end{barticle}
	\endbibitem
	
	\bibitem[\protect\citeauthoryear{Gould et~al.}{2000}]{glrt2}
	\begin{barticle}
		\bauthor{\bsnm{Gould}, \binits{N.I.M.}},
		\bauthor{\bsnm{Lucidi}, \binits{S.}},
		\bauthor{\bsnm{Roma}, \binits{M.}},
		\bauthor{\bsnm{Toint}, \binits{P.L.}}:
		\batitle{Exploiting negative curvature directions in linesearch methods for
			unconstrained optimization}.
		\bjtitle{Optimization Methods and Software}
		\bvolume{14},
		\bfpage{75}--\blpage{98}
		(\byear{2000})
	\end{barticle}
	\endbibitem
	
	\bibitem[\protect\citeauthoryear{Fasano et~al.}{2025}]{fapiro:2025COAP}
	\begin{barticle}
		\bauthor{\bsnm{Fasano}, \binits{G.}},
		\bauthor{\bsnm{Piermarini}, \binits{C.}},
		\bauthor{\bsnm{Roma}, \binits{M.}}:
		\batitle{Exploiting effective negative curvature directions via {SYMMBK}
			algorithm, in {N}ewton–{K}rylov methods}.
		\bjtitle{Computational Optimization and Applications}
		\bvolume{91},
		\bfpage{617}--\blpage{647}
		(\byear{2025})
	\end{barticle}
	\endbibitem
	
	\bibitem[\protect\citeauthoryear{Caliciotti et~al.}{2020}]{Ca-Fa-Po-Ro:2020}
	\begin{barticle}
		\bauthor{\bsnm{Caliciotti}, \binits{A.}},
		\bauthor{\bsnm{Fasano}, \binits{G.}},
		\bauthor{\bsnm{Potra}, \binits{F.}},
		\bauthor{\bsnm{Roma}, \binits{M.}}:
		\batitle{Issues on the use of a modified {B}unch and {K}aufman decomposition
			for large scale {N}ewton's equation}.
		\bjtitle{{Computational Optimization and Applications }}
		\bvolume{77},
		\bfpage{627}--\blpage{651}
		(\byear{2020})
	\end{barticle}
	\endbibitem
	
	\bibitem[\protect\citeauthoryear{Chandra}{1978}]{chandra}
	\begin{botherref}
		\oauthor{\bsnm{Chandra}, \binits{R.}}:
		Conjugate gradient methods for partial differential equations.
		PhD thesis,
		Yale University,
		New Haven
		(1978).
		Research Report 129
	\end{botherref}
	\endbibitem
	
	\bibitem[\protect\citeauthoryear{McCormick}{1977}]{mccormick:mathprogr}
	\begin{barticle}
		\bauthor{\bsnm{McCormick}, \binits{G.P.}}:
		\batitle{A modification of {A}rmijo's step-size rule for negative curvature}.
		\bjtitle{Mathematical Programming}
		\bvolume{13},
		\bfpage{111}--\blpage{115}
		(\byear{1977})
	\end{barticle}
	\endbibitem
	
	\bibitem[\protect\citeauthoryear{Fasano and Lucidi}{2009}]{fasano.lucidi:2009}
	\begin{barticle}
		\bauthor{\bsnm{Fasano}, \binits{G.}},
		\bauthor{\bsnm{Lucidi}, \binits{S.}}:
		\batitle{A nonmonotone truncated {N}ewton--{K}rylov method exploiting negative
			curvature directions, for large scale unconstrained optimization}.
		\bjtitle{{Optimization Letters}}
		\bvolume{3},
		\bfpage{521}--\blpage{535}
		(\byear{2009})
	\end{barticle}
	\endbibitem
	
	\bibitem[\protect\citeauthoryear{Conn et~al.}{2009}]{CSV2009}
	\begin{bbook}
		\bauthor{\bsnm{Conn}, \binits{A.R.}},
		\bauthor{\bsnm{Scheinberg}, \binits{K.}},
		\bauthor{\bsnm{Vicente}, \binits{L.N.}}:
		\bbtitle{Introduction to Derivative--Free Optimization}.
		\bpublisher{SIAM},
		\blocation{Philadelphia}
		(\byear{2009})
	\end{bbook}
	\endbibitem
	
	\bibitem[\protect\citeauthoryear{Shi et~al.}{2023}]{shi.2023}
	\begin{barticle}
		\bauthor{\bsnm{Shi}, \binits{H.-J.M.}},
		\bauthor{\bsnm{Xuan}, \binits{M.Q.}},
		\bauthor{\bsnm{Oztoprak}, \binits{F.}},
		\bauthor{\bsnm{Nocedal}, \binits{J.}}:
		\batitle{On the numerical performance of finite difference-- based methods for
			derivative--free optimization}.
		\bjtitle{Optimization Methods and Software}
		\bvolume{38}(\bissue{2}),
		\bfpage{289}--\blpage{311}
		(\byear{2023})
	\end{barticle}
	\endbibitem
	
	\bibitem[\protect\citeauthoryear{Powell}{2006}]{powell.06}
	\begin{bbook}
		\bauthor{\bsnm{Powell}, \binits{M.J.D.}}:
		In: \beditor{\bsnm{Di~Pillo}, \binits{G.}},
		\beditor{\bsnm{Roma}, \binits{M.}} (eds.)
		\bbtitle{The NEWUOA software for unconstrained optimization without
			derivatives},
		pp. \bfpage{255}--\blpage{297}.
		\bpublisher{Springer},
		\blocation{Boston, MA}
		(\byear{2006}).
		\doiurl{10.1007/0-387-30065-1_16}
	\end{bbook}
	\endbibitem
	
	\bibitem[\protect\citeauthoryear{Powell}{2008}]{powell.08}
	\begin{barticle}
		\bauthor{\bsnm{Powell}, \binits{M.J.D.}}:
		\batitle{Development of {NEWUOA} for minimization without derivatives}.
		\bjtitle{IMA Journal on Numerical Analysis}
		\bvolume{28},
		\bfpage{649}--\blpage{664}
		(\byear{2008})
	\end{barticle}
	\endbibitem
	
	\bibitem[\protect\citeauthoryear{Polak}{1971}]{polak.71}
	\begin{bbook}
		\bauthor{\bsnm{Polak}, \binits{E.}}:
		\bbtitle{Computational Methods in Optimization: A Unified Approach}.
		\bpublisher{Academic Press},
		\blocation{New York}
		(\byear{1971})
	\end{bbook}
	\endbibitem
	
	\bibitem[\protect\citeauthoryear{Dolan}{}]{dolan.99}
	\begin{botherref}
		\oauthor{\bsnm{Dolan}, \binits{E.D.}}:
		Pattern Search Behavior in Nonlinear Optimization.
		Honors Thesis, Department of Computer Science, College of William \& Mary,
		Williamsburg, Virginia. May 1999
	\end{botherref}
	\endbibitem
	
	\bibitem[\protect\citeauthoryear{Hooke and Jeeves}{1961}]{hooke.61}
	\begin{barticle}
		\bauthor{\bsnm{Hooke}, \binits{R.}},
		\bauthor{\bsnm{Jeeves}, \binits{T.A.}}:
		\batitle{Direct search solution of numerical and statistical problems}.
		\bjtitle{Journal of the Association for Computing Machinery}
		\bvolume{8}(\bissue{2}),
		\bfpage{212}--\blpage{229}
		(\byear{1961})
	\end{barticle}
	\endbibitem
	
	\bibitem[\protect\citeauthoryear{Spendley et~al.}{1962}]{spendley.62}
	\begin{barticle}
		\bauthor{\bsnm{Spendley}, \binits{W.}},
		\bauthor{\bsnm{Hext}, \binits{G.R.}},
		\bauthor{\bsnm{Himsworth}, \binits{F.R.}}:
		\batitle{Sequential application of simplex designs in optimisation and
			evolutionary operation}.
		\bjtitle{Technometrics}
		\bvolume{4}(\bissue{4}),
		\bfpage{441}--\blpage{461}
		(\byear{1962})
	\end{barticle}
	\endbibitem
	
	\bibitem[\protect\citeauthoryear{Nelder and Mead}{1965}]{nelder.65}
	\begin{barticle}
		\bauthor{\bsnm{Nelder}, \binits{J.A.}},
		\bauthor{\bsnm{Mead}, \binits{R.}}:
		\batitle{A simplex method for function minimization}.
		\bjtitle{The Computer Journal}
		\bvolume{7}(\bissue{4}),
		\bfpage{308}--\blpage{313}
		(\byear{1965})
	\end{barticle}
	\endbibitem
	
	\bibitem[\protect\citeauthoryear{Lagarias et~al.}{1998}]{lagarias.98}
	\begin{barticle}
		\bauthor{\bsnm{Lagarias}, \binits{J.C.}},
		\bauthor{\bsnm{Reeds}, \binits{J.A.}},
		\bauthor{\bsnm{Wright}, \binits{M.H.}},
		\bauthor{\bsnm{Wright}, \binits{P.E.}}:
		\batitle{Convergence properties of the {N}elder--{M}ead simplex method in low
			dimensions}.
		\bjtitle{SIAM Journal on Optimization}
		\bvolume{9}(\bissue{1}),
		\bfpage{112}--\blpage{147}
		(\byear{1998})
	\end{barticle}
	\endbibitem
	
	\bibitem[\protect\citeauthoryear{Torczon}{1991}]{torczon.98}
	\begin{barticle}
		\bauthor{\bsnm{Torczon}, \binits{V.}}:
		\batitle{On the convergence of the multidirectional search algorithm}.
		\bjtitle{SIAM Journal on Optimization}
		\bvolume{11}(\bissue{1}),
		\bfpage{123}--\blpage{145}
		(\byear{1991})
	\end{barticle}
	\endbibitem
	
	\bibitem[\protect\citeauthoryear{Davison and
		Hinkley}{1997}]{Davison_Hinkley_1997}
	\begin{bbook}
		\bauthor{\bsnm{Davison}, \binits{A.C.}},
		\bauthor{\bsnm{Hinkley}, \binits{D.V.}}:
		\bbtitle{Bootstrap Methods and Their Application}.
		\bsertitle{Cambridge Series in Statistical and Probabilistic Mathematics}.
		\bpublisher{Cambridge University Press},
		\blocation{Cambridge}
		(\byear{1997})
	\end{bbook}
	\endbibitem
	
	\bibitem[\protect\citeauthoryear{Caliciotti et~al.}{2024}]{Ca_Co_Fa:2024}
	\begin{barticle}
		\bauthor{\bsnm{Caliciotti}, \binits{A.}},
		\bauthor{\bsnm{Corazza}, \binits{M.}},
		\bauthor{\bsnm{Fasano}, \binits{G.}}:
		\batitle{From regression models to machine learning approaches for long term
			{B}itcoin price forecast}.
		\bjtitle{Annals of Operations Research}
		\bvolume{336},
		\bfpage{359}--\blpage{381}
		(\byear{2024})
	\end{barticle}
	\endbibitem
	
	\bibitem[\protect\citeauthoryear{Liu et~al.}{2024}]{liu2024meanborda}
	\begin{barticle}
		\bauthor{\bsnm{Liu}, \binits{Q.}},
		\bauthor{\bsnm{Jing}, \binits{Y.}},
		\bauthor{\bsnm{Yan}, \binits{Y.}},
		\bauthor{\bsnm{Li}, \binits{Y.}}:
		\batitle{Mean-based borda count for paradox-free comparisons of optimization
			algorithms}.
		\bjtitle{Information Sciences}
		\bvolume{660},
		\bfpage{120120}
		(\byear{2024})
		\doiurl{10.1016/j.ins.2024.120120}
	\end{barticle}
	\endbibitem
	
\end{thebibliography}

\end{document}